\documentclass[11pt]{amsart}
\usepackage{amsthm}
\usepackage{amsmath}
\usepackage{amssymb}
\usepackage{verbatim}
\usepackage{mathrsfs}
\usepackage{eucal}
\let\mathcal=\CMcal
\usepackage[dvipdfm,bookmarks=true,bookmarksnumbered=true]{hyperref}

\allowdisplaybreaks[4]

\begin{document}
\def\sbt{\raisebox{1.2pt}{$\scriptscriptstyle\,\bullet\,$}}

\def\alp{\alpha}
\def\bet{\beta}
\def\gam{\gamma}
\def\del{\delta}
\def\eps{\epsilon}
\def\zet{\zeta}
\def\tht{\theta}
\def\iot{\iota}
\def\kap{\kappa}
\def\lam{\lambda}
\def\sig{\sigma}
\def\ome{\omega}
\def\vep{\varepsilon}
\def\vth{\vartheta}
\def\vpi{\varpi}
\def\vrh{\varrho}
\def\vsi{\varsigma}
\def\vph{\varphi}
\def\Gam{\Gamma}
\def\Del{\Delta}
\def\Tht{\Theta}
\def\Lam{\Lambda}
\def\Sig{\Sigma}
\def\Ups{\Upsilon}
\def\Ome{\Omega}
\def\vka{\varkappa}
\def\vDe{\varDelta}
\def\vSi{\varSigma}
\def\vTh{\varTheta}
\def\vGm{\varGamma}
\def\vOm{\varOmega}
\def\vPi{\varPi}
\def\vPh{\varPhi}
\def\vPs{\varPsi}
\def\vUp{\varUpsilon}
\def\vXi{\varXi}

\def\frka{{\mathfrak a}}    \def\frkA{{\mathfrak A}}
\def\frkb{{\mathfrak b}}    \def\frkB{{\mathfrak B}}
\def\frkc{{\mathfrak c}}    \def\frkC{{\mathfrak C}}
\def\frkd{{\mathfrak d}}    \def\frkD{{\mathfrak D}}
\def\frke{{\mathfrak e}}    \def\frkE{{\mathfrak E}}
\def\frkf{{\mathfrak f}}    \def\frkF{{\mathfrak F}}
\def\frkg{{\mathfrak g}}    \def\frkG{{\mathfrak G}}
\def\frkh{{\mathfrak h}}    \def\frkH{{\mathfrak H}}
\def\frki{{\mathfrak i}}    \def\frkI{{\mathfrak I}}
\def\frkj{{\mathfrak j}}    \def\frkJ{{\mathfrak J}}
\def\frkk{{\mathfrak k}}    \def\frkK{{\mathfrak K}}
\def\frkl{{\mathfrak l}}    \def\frkL{{\mathfrak L}}
\def\frkm{{\mathfrak m}}    \def\frkM{{\mathfrak M}}
\def\frkn{{\mathfrak n}}    \def\frkN{{\mathfrak N}}
\def\frko{{\mathfrak o}}    \def\frkO{{\mathfrak O}}
\def\frkp{{\mathfrak p}}    \def\frkP{{\mathfrak P}}
\def\frkq{{\mathfrak q}}    \def\frkQ{{\mathfrak Q}}
\def\frkr{{\mathfrak r}}    \def\frkR{{\mathfrak R}}
\def\frks{{\mathfrak s}}    \def\frkS{{\mathfrak S}}
\def\frkt{{\mathfrak t}}    \def\frkT{{\mathfrak T}}
\def\frku{{\mathfrak u}}    \def\frkU{{\mathfrak U}}
\def\frkv{{\mathfrak v}}    \def\frkV{{\mathfrak V}}
\def\frkw{{\mathfrak w}}    \def\frkW{{\mathfrak W}}
\def\frkx{{\mathfrak x}}    \def\frkX{{\mathfrak X}}
\def\frky{{\mathfrak y}}    \def\frkY{{\mathfrak Y}}
\def\frkz{{\mathfrak z}}    \def\frkZ{{\mathfrak Z}}

\def\cal{\fam2}
\def\cala{{\cal A}}
\def\calb{{\cal B}}
\def\calc{{\cal C}}
\def\cald{{\cal D}}
\def\cale{{\cal E}}
\def\calf{{\cal F}}
\def\calg{{\cal G}}
\def\calh{{\cal H}}
\def\cali{{\cal I}}
\def\calj{{\cal J}}
\def\calk{{\cal K}}
\def\call{{\cal L}}
\def\calm{{\cal M}}
\def\caln{{\cal N}}
\def\calo{{\cal O}}
\def\calp{{\cal P}}
\def\calq{{\cal Q}}
\def\calr{{\cal R}}
\def\cals{{\cal S}}
\def\calt{{\cal T}}
\def\calu{{\cal U}}
\def\calv{{\cal V}}
\def\calw{{\cal W}}
\def\calx{{\cal X}}
\def\caly{{\cal Y}}
\def\calz{{\cal Z}}

\def\AA{{\mathbb A}}
\def\BB{{\mathbb B}}
\def\CC{{\mathbb C}}
\def\DD{{\mathbb D}}
\def\EE{{\mathbb E}}
\def\FF{{\mathbb F}}
\def\GG{{\mathbb G}}
\def\HH{{\mathbb H}}
\def\II{{\mathbb I}}
\def\JJ{{\mathbb J}}
\def\KK{{\mathbb K}}
\def\LL{{\mathbb L}}
\def\MM{{\mathbb M}}
\def\NN{{\mathbb N}}
\def\OO{{\mathbb O}}
\def\PP{{\mathbb P}}
\def\QQ{{\mathbb Q}}
\def\RR{{\mathbb R}}
\def\SS{{\mathbb S}}
\def\TT{{\mathbb T}}
\def\UU{{\mathbb U}}
\def\VV{{\mathbb V}}
\def\WW{{\mathbb W}}
\def\XX{{\mathbb X}}
\def\YY{{\mathbb Y}}
\def\ZZ{{\mathbb Z}}

\def\bfa{{\mathbf a}}    \def\bfA{{\mathbf A}}
\def\bfb{{\mathbf b}}    \def\bfB{{\mathbf B}}
\def\bfc{{\mathbf c}}    \def\bfC{{\mathbf C}}
\def\bfd{{\mathbf d}}    \def\bfD{{\mathbf D}}
\def\bfe{{\mathbf e}}    \def\bfE{{\mathbf E}}
\def\bff{{\mathbf f}}    \def\bfF{{\mathbf F}}
\def\bfg{{\mathbf g}}    \def\bfG{{\mathbf G}}
\def\bfh{{\mathbf h}}    \def\bfH{{\mathbf H}}
\def\bfi{{\mathbf i}}    \def\bfI{{\mathbf I}}
\def\bfj{{\mathbf j}}    \def\bfJ{{\mathbf J}}
\def\bfk{{\mathbf k}}    \def\bfK{{\mathbf K}}
\def\bfl{{\mathbf l}}    \def\bfL{{\mathbf L}}
\def\bfm{{\mathbf m}}    \def\bfM{{\mathbf M}}
\def\bfn{{\mathbf n}}    \def\bfN{{\mathbf N}}
\def\bfo{{\mathbf o}}    \def\bfO{{\mathbf O}}
\def\bfp{{\mathbf p}}    \def\bfP{{\mathbf P}}
\def\bfq{{\mathbf q}}    \def\bfQ{{\mathbf Q}}
\def\bfr{{\mathbf r}}    \def\bfR{{\mathbf R}}
\def\bfs{{\mathbf s}}    \def\bfS{{\mathbf S}}
\def\bft{{\mathbf t}}    \def\bfT{{\mathbf T}}
\def\bfu{{\mathbf u}}    \def\bfU{{\mathbf U}}
\def\bfv{{\mathbf v}}    \def\bfV{{\mathbf V}}
\def\bfw{{\mathbf w}}    \def\bfW{{\mathbf W}}
\def\bfx{{\mathbf x}}    \def\bfX{{\mathbf X}}
\def\bfy{{\mathbf y}}    \def\bfY{{\mathbf Y}}
\def\bfz{{\mathbf z}}    \def\bfZ{{\mathbf Z}}

\def\scra{{\mathscr A}}
\def\scrb{{\mathscr B}}
\def\scrc{{\mathscr C}}
\def\scrd{{\mathscr D}}
\def\scre{{\mathscr E}}
\def\scrf{{\mathscr F}}
\def\scrg{{\mathscr G}}
\def\scrh{{\mathscr H}}
\def\scri{{\mathscr I}}
\def\scrj{{\mathscr J}}
\def\scrk{{\mathscr K}}
\def\scrl{{\mathscr L}}
\def\scrm{{\mathscr M}}
\def\scrn{{\mathscr N}}
\def\scro{{\mathscr O}}
\def\scrp{{\mathscr P}}
\def\scrq{{\mathscr Q}}
\def\scrr{{\mathscr R}}
\def\scrs{{\mathscr S}}
\def\scrt{{\mathscr T}}
\def\scru{{\mathscr U}}
\def\scrv{{\mathscr V}}
\def\scrw{{\mathscr W}}
\def\scrx{{\mathscr X}}
\def\scry{{\mathscr Y}}
\def\scrz{{\mathscr Z}}

\def\phm{\phantom}
\def\smallstrut{\vphantom{\vrule height 3pt }}
\def\bdm #1#2#3#4{\left(
\begin{array} {c|c}{\ds{#1}}
 & {\ds{#2}} \\ \hline
{\ds{#3}\vphantom{\ds{#3}^1}} &  {\ds{#4}}
\end{array}
\right)}

\def\GL{\mathrm{GL}}
\def\PGL{\mathrm{PGL}}
\def\SL{\mathrm{SL}}
\def\Mp{\mathrm{Mp}}
\def\SU{\mathrm{SU}}
\def\SO{\mathrm{SO}}
\def\U{\mathrm{U}}
\def\Mat{\mathrm{M}}
\def\Tr{\mathrm{Tr}}
\def\tr{\mathrm{tr}}
\def\Ad{\mathrm{Ad}}
\def\Paf{\mathrm{Paf}}
\def\new{\mathrm{new}}
\def\Wh{\mathrm{Wh}}
\def\FJ{\mathrm{FJ}}
\def\Fj{\mathrm{Fj}}
\def\Sym{\mathrm{Sym}}
\def\sym{\mathrm{sym}}
\def\supp{\mathrm{supp}}
\def\proj{\mathrm{proj}}
\def\Hom{\mathrm{Hom}}
\def\End{\mathrm{End}}
\def\Ker{\mathrm{Ker}}
\def\Res{\mathrm{Res}}
\def\res{\mathrm{res}}
\def\cusp{\mathrm{cusp}}
\def\Irr{\mathrm{Irr}}
\def\rank{\mathrm{rank}}
\def\sgn{\mathrm{sgn}}
\def\diag{\mathrm{diag}}
\def\Wd{\mathrm{Wd}}
\def\nd{\mathrm{nd}}
\def\d{\mathrm{d}}
\def\La{\langle}
\def\Ra{\rangle}

\def\trs{\,^t\!}
\def\tri{\,^\iot\!}
\def\iu{\sqrt{-1}}
\def\oo{\hbox{\bf 0}}
\def\ono{\hbox{\bf 1}}
\def\smallcirc{\lower .3em \hbox{\rm\char'27}\!}
\def\AAf{\AA_\bff}
\def\thalf{\tfrac{1}{2}}
\def\bsl{\backslash}
\def\wtl{\widetilde}
\def\til{\tilde}
\def\Ind{\operatorname{Ind}}
\def\ind{\operatorname{ind}}
\def\cind{\operatorname{c-ind}}
\def\ord{\operatorname{ord}}
\def\beq{\begin{equation}}
\def\eeq{\end{equation}}
\def\d{\mathrm{d}}

\newcounter{one}
\setcounter{one}{1}
\newcounter{two}
\setcounter{two}{2}
\newcounter{thr}
\setcounter{thr}{3}
\newcounter{fou}
\setcounter{fou}{4}
\newcounter{fiv}
\setcounter{fiv}{5}
\newcounter{six}
\setcounter{six}{6}
\newcounter{sev}
\setcounter{sev}{7}

\newcommand{\shp}{\rm\char'43}

\def\lddots{\mathinner{\mskip1mu\raise1pt\vbox{\kern7pt\hbox{.}}\mskip2mu\raise4pt\hbox{.}\mskip2mu\raise7pt\hbox{.}\mskip1mu}}
\newcommand{\1}{1\hspace{-0.25em}{\rm{l}}}

\makeatletter
\def\varddots{\mathinner{\mkern1mu
    \raise\p@\hbox{.}\mkern2mu\raise4\p@\hbox{.}\mkern2mu
    \raise7\p@\vbox{\kern7\p@\hbox{.}}\mkern1mu}}
\makeatother

\def\today{\ifcase\month\or
 January\or February\or March\or April\or May\or June\or
 July\or August\or September\or October\or November\or December\fi
 \space\number\day, \number\year}

\makeatletter
\def\varddots{\mathinner{\mkern1mu
    \raise\p@\hbox{.}\mkern2mu\raise4\p@\hbox{.}\mkern2mu
    \raise7\p@\vbox{\kern7\p@\hbox{.}}\mkern1mu}}
\makeatother

\def\today{\ifcase\month\or
 January\or February\or March\or April\or May\or June\or
 July\or August\or September\or October\or November\or December\fi
 \space\number\day, \number\year}

\makeatletter
\@addtoreset{equation}{section}
\def\theequation{\thesection.\arabic{equation}}

\theoremstyle{plain}
\newtheorem{theorem}{Theorem}[section]
\newtheorem*{theorem_a}{Theorem A}
\newtheorem*{theorem_b}{Theorem B}
\newtheorem*{theorem_c}{Theorem C}
\newtheorem*{theorem_e}{Theorem}
\newtheorem*{corollary_a}{Corollary A}
\newtheorem{lemma}[theorem]{Lemma}
\newtheorem{proposition}[theorem]{Proposition}
\theoremstyle{definition}
\newtheorem{definition}[theorem]{Definition}
\newtheorem{conjecture}[theorem]{Conjecture}
\theoremstyle{remark}
\newtheorem{remark}[theorem]{Remark}
\newtheorem*{main_remark}{Remark}
\newtheorem{corollary}[theorem]{Corollary}

\renewcommand{\thepart}{\Roman{part}}
\setcounter{tocdepth}{1}
\setcounter{section}{0} 

%%%%%%%%%%%%%%%%%%%%%%%%%%%%%%%%%%%%%%%%%%%%%%%%%%%%%%%%%%%%%%%%%%%%%%%%%%%%%%%%

\title[]{\bf On the lifting of Hilbert cusp forms to Hilbert-Siegel cusp forms}  
\author{Tamotsu Ikeda}
\address{Department of mathematics, Kyoto University, Kitashirakawa Oiwake-cho, Sakyo-ku, Kyoto, 606-8502, Japan}
\email{ikeda@math.kyoto-u.ac.jp}
\author{Shunsuke Yamana}
\address{Department of mathematics, Kyoto University, Kitashirakawa Oiwake-cho, Sakyo-ku, Kyoto, 606-8502, Japan/Hakubi Center, Yoshida-Ushinomiya-cho, Sakyo-ku, Kyoto, 606-8501, Japan}
\email{yamana07@math.kyoto-u.ac.jp}
\begin{abstract}
Starting from a Hilbert cusp form of weight $2\kap$, we will construct a Hilbert-Siegel cusp form of weight $\kap+\frac{m}{2}$ and degree $m$ and its transfer to inner forms of symplectic groups. 
%This is a generalization of the Saito-Kurokawa lifting and its transfer to the higher dimensional case. 
%We construct a lifting that associates to a Hilbert cusp form a holomorphic cusp form on a quaternionic unitary group of higher degree. This refines and generalizes the lifting constructed by the first author. 
\end{abstract}
\maketitle

%%%%%%%%%%%%%%%%%%%%%%%%%%%%%%%%%%%%%%%%%%%%%%%%%%%%%%%%%%%%%%%%%%%%%%%%%%%%%%%%
\tableofcontents

%%%%%%%%%%%%%%%%%%%%%%%%%%%%%%%%%%%%%%%%%%%%%%%%%%%%%%%%%%%%%%%%%%%%%%%%%%%%%%%%
\section{Introduction}\label{sec:1}

The present investigation deals with the following problem: 
starting from simple automorphic data such as cusp forms on $\GL_2$, %or Dirichlet characters, 
construct more complicated automorphic forms on groups of higher degree. 
Toward this problem, Ikeda \cite{I2} has constructed a lifting associating to an elliptic cusp form a Siegel cusp form of even genus. 
This paper generalizes it to Hilbert cusp forms. 

To illustrate our results, let $F$ be a totally real number field of degree $d$ with ad\`{e}le ring $\AA$. 
%The tensor product of $D$ with the real number field over the rational number field may be identified with the product of $d$ copies of the total matrix algebra $\Mat_2(\RR)$ of degree $2$ over $\RR$. 
We write $\AAf$ and $\AA_\infty$ for the finite part and the infinite part of the ad\`{e}le ring. 
We denote the set of $d$ real primes of $F$ by $\frkS_\infty$. 
Let $\Sym_m=\{z\in\Mat_m\;|\trs z=z\}$ be the space of symmetric matrix of size $m$ and $W_m$ a symplectic vector space of dimension $2m$. 
We take matrix representation 
\[Sp_m=\left\{g\in\GL_{2m}\;\biggl|\;g\begin{pmatrix} 0 & -\ono_m \\ \ono_m & 0\end{pmatrix}\trs g=\begin{pmatrix} 0 & -\ono_m \\ \ono_m & 0\end{pmatrix}\right\} \]
of the associated symplectic group $Sp(W_m)$ by choosing a Witt basis of $W_m$. 
We define homomorphisms $\frkm:\GL_m\to Sp_m$ and $\frkn:\Sym_m\to Sp_m$ by 
\begin{align*}
\frkm(a)&=\begin{pmatrix} a & 0 \\ 0 & \trs a^{-1}\end{pmatrix}, & 
\frkn(b)&=\begin{pmatrix} \ono_m & b \\ 0 & \ono_m\end{pmatrix}. 
\end{align*}
Let $\Mp(W_m)_\AA\twoheadrightarrow Sp_m(\AA)$ be the metaplectic double cover. 
We here denote the inverse images of $Sp_m(\AA_\infty)$, $Sp_m(\AAf)$ and $Sp_m(F_v)$ by $\Mp(W_m)_\infty$, $\Mp(W_m)_\bff$ and $\Mp(W_m)_v$, respectively.  

Define the character $\bfe_\infty:\CC^d\to\CC^\times$ by $\bfe_\infty(\calz)=\prod_{v\in\frkS_\infty}e^{2\pi\iu \calz_v}$. 
Let $\psi=\otimes_v\psi_v$ be the additive character of $\AA/F$ whose restriction to $\AA_\infty$ is $\bfe_\infty|_{\RR^d}$. 
For $\xi\in\Sym_m(F)$ we define the character $\psi^\xi_\bff=\otimes_{v\notin\frkS_\infty}\psi_v^\xi:\Sym_m(\AAf)\to\CC^\times$ by $\psi^\xi_\bff(z)=\prod_{v\notin\frkS_\infty}\psi_v(\tr(\xi z_v))$. 
For $t\in F^\times_v$ there is an $8^\mathrm{th}$ root of unity $\gam(\psi_v^t)$ such that for all Schwartz functions $\phi$ on $F_v$ 
\[\int_{F_v}\phi(x_v)\psi_v(tx_v^2)\,\d x_v=\gam(\psi_v^t)|2t|_v^{-1/2}\int_{F_v}\calf\phi(x_v)\psi_v\left(-\frac{x_v^2}{4t}\right)\,\d x_v, \]
where $\d x_v$ is the self-dual Haar measure on $F_v$ with respect to the Fourier transform $\calf\phi(y)=\int_{F_v}\phi(x_v)\psi_v(x_vy)\,\d x_v$. 
Set $\gam^{\psi_v}(t)=\gam(\psi_v)/\gam(\psi^t_v)$. 
We denote the set of totally positive elements of $F$ by $F^\times_+$, the set of positive definite symmetric matrices of rank $m$ over $\RR$ by $\Sym_m(\RR)^+$ and the set of totally positive definite symmetric matrices of rank $m$ over $F$ by $\Sym_m^+$. 
Let 
\[\calh_m=\{\calz\in\Sym_m(\CC)\;|\;\Im\calz\in\Sym_m(\RR)^+\}\] 
be the Siegel upper half space  of degree $m$. 
For $t\in F^\times$ we write $\hat\chi^t=\otimes^{}_v\hat\chi^t_v$ for the quadratic character of $\AA^\times/F^\times$ associated to the extension $F(\sqrt{t})/F$ and denote its restriction to the finite id\`{e}le group $\AAf^\times$ by $\hat\chi_\bff^t$. 
For $\ell\in\RR^d$ we will set $|t|^\ell=\prod_{v\in\frkS_\infty}|t|_v^{\ell_v}$. 

The Lie group $\Mp(W_m)_v$ acts on $\calh_m$ through $Sp_m(F_v)$ for $v\in\frkS_\infty$. 
There is a unique factor of automorphy $\jmath:\Mp(W_m)_v\times\calh_m\to\CC^\times$ satisfying $\jmath(\til g_v,\calz_v)^2=\det(C_v\calz_v+D_v)$. 
We here write the projection of $\til g_v$ to $Sp_m(F_v)$ as $\begin{pmatrix} * & * \\ C_v & D_v\end{pmatrix}$. 
For each tuple $\kap$ of $d$ integers we set $J_{\kap/2}(\til g,\calz)=\prod_{v\in\frkS_\infty}\jmath(\til g_v,\calz_v)^{\kap_v}$ for $\til g\in\Mp(W_m)_\infty$ and $\calz\in\calh_m^d$. 

For $\ell\in\frac{1}{2}\ZZ^d$, $\til g\in\Mp(W_m)_\infty$ and a function $\calf$ on $\calh^d_m$ we define another function $\calf|_\ell\til g$ on $\calh_m^d$ by $\calf|_\ell\til g(\calz)=\calf(\til g\calz)J_\ell(\til g,\calz)^{-1}$. 
Define the origin of $\calh_m^d$ and the standard maximal compact subgroup of $\Mp(W_m)_\infty$ by 
\begin{align*}
\bfi_m&=(\iu\ono_m,\dots,\iu\ono_m)\in\calh^d_m, &
\til\calk_m&=\{\til g\in\Mp(W_m)_\infty\;|\;\til g(\bfi_m)=\bfi_m\}. 
\end{align*} 

A Hilbert-Siegel cusp form $\calf$ of degree $m$ and weight $\ell$ is a smooth function on $\Mp(W_m)_\AA$ which is left invariant under $Sp_m(F)$ and transforms on the right by the character $\til k\mapsto J_\ell(\til k,\bfi_m)^{-1}$ of $\til\calk_m$ and such that $\calf_{\til\Del}$ is a holomorphic function on $\calh_m^d$ having a Fourier expansion of the form 
\beq
\calf_{\til\Del}(\calz)=\sum_{\xi\in \Sym^+_m}|\det\xi|^{\ell/2}\bfw_\xi(\til\Del,\calf)\bfe_\infty(\tr(\xi\calz)) \label{tag:11}
\eeq 
for each $\til\Del\in\Mp(W_m)_\bff$, where $\bfw_\xi(\calf)$ is a function on $\Mp(W_m)_\bff$ and the function $\calf_{\til\Del}:\calh^d_m\to\CC$ is defined by 
\begin{align*}
\calf_{\til\Del}|_\ell\til g(\bfi_m)&=\calf(\til g\til\Del), &\til g&\in \Mp(W_m)_\infty. 
\end{align*}
Let $\frkC^{(m)}_\ell$ denote the space of Hilbert-Siegel cusp forms of degree $m$ and weight $\ell$. 
The group $\Mp(W_m)_\bff$ acts on the space $\frkC^{(m)}_\ell$ by right translation. 
It is important to know which representations appear in this space. 

Let $\pi_\bff\simeq\otimes'_{v\notin\frkS_\infty}\pi_v$ be an irreducible admissible unitary generic representation of $\PGL_2(\AA_\bff)$. 
For a technical reason we suppose that none of $\pi_v$ is supercuspidal, i.e., there is a collection of continuous characters $\mu_v$ of the multiplicative groups of nonarchimedean local fields $F_v$ such that $\pi_\bff$ is equivalent to the unique irreducible submodule of the principal series representation $\otimes_{v\notin\frkS_\infty}'I(\mu^{}_v,\mu_v^{-1})$, where $\mu_v$ is unramified for almost all $v$.   
We form the restricted tensor product $I_m^{\psi_\bff}(\mu_\bff)=\otimes'_{v\notin\frkS_\infty}I_m^{\psi_v}(\mu_v)$, where $I_m^{\psi_v}(\mu_v)$ is the representation of $\Mp(W_m)_v$ on the space of smooth functions $h_v$ on the local metaplectic group transforming on the left according to
\[h_v((\frkm(a)\frkn(b),\zet)\til g)=\zet^m\gam^{\psi_v}(\det a)^m\mu_v(\det a)|\det a|_v^{(m+1)/2}h_v(\til g)\]
for all $\zet\in\{\pm 1\}$, $a\in\GL_m(F_v)$, $z\in\Sym_m(F_v)$ and $\til g\in\Mp(W_m)_v$. 
This representation has a unique irreducible submodule $A^{\psi_\bff}_m(\mu_\bff)$, which is unitary. 

Put $\calk_1=\{g\in\SL_2(\AA_\infty)\;|\;g(\bfi_1)=\bfi_1\}$. 
Given  
\[g\in\GL_2(\AA_\infty)^+:=\{(g_v)\in\GL_2(\AA_\infty)\;|\;\det g_v>0\text{ for all }v\in\frkS_\infty\} \]
 and a function $\calf$ on $\calh^d_1$, we define a function $\calf|_\kap g$ on $\calh^d_1$ by 
\begin{align*}
\calf|_\kap g(\calz)&=\calf(g\calz)J_\kap(g,\calz)^{-1}, &
J_\kap(g,\calz)&=|\det g|^{-\kap/2}\prod_{v\in\frkS_\infty}(c_v\calz_v+d_v)^{\kap_v},  
\end{align*} 
where $g_v=\begin{pmatrix} * & * \\ c_v & d_v\end{pmatrix}$. 
A Hilbert cusp form $\calf$ on $\PGL_2$ of weight $2\kap$ is a smooth function on $\GL_2(\AA)$ satisfying 
\begin{align*}
\calf(z\gam gk)&=\calf(g) J_{2\kap}(k,\bfi_1)^{-1} &
(z&\in \AA^\times,\; \gam\in\GL_2(F),\; g\in\GL_2(\AA),\; k\in \calk_1)
\end{align*} 
and having a Fourier expansion of the form 
\[\calf_\Del(\calz)=\sum_{t\in F^\times_+}|t|^\kap\bfw_t(\Del,\calf)\bfe_\infty(t\calz) \] 
for each $\Del\in\GL_2(\AAf)$, where $\bfw_t(\calf)$ is a function on $\GL_2(\AAf)$ and the function $\calf_\Del:\calh_1^d\to\CC$ is defined by $\calf_\Del|_{2\kap}g(\bfi_1)=\calf(g\Del)$ for $g\in\GL_2(\AA_\infty)^+$. 
We write $\frkC_{2\kap}$ for the space of Hilbert cusp forms on $\PGL_2$ of weight $2\kap$. 

%It is interesting to ask when the representation $A_m^{\psi_\bff}(\mu_\bff)$ occurs in $\frkC^{(m)}_\ell$. 

\begin{theorem}\label{thm:11}
Notation being as above, $A^{\psi_\bff}_m(\mu_\bff)$ appears in $\frkC^{(m)}_{(2\kap+m)/2}$ if and only if $\pi_\bff$ appears in $\frkC_{2\kap}$ and $(-1)^{\sum_{v\in\frkS_\infty}\kap_v}\prod_{v\notin\frkS_\infty}\mu_v(-1)=1$.  
\end{theorem}

The representation $A^{\psi_\bff}_1(\mu_\bff)$ is the Shimura correspondence and $A^{\psi_\bff}_2(\mu_\bff)$ is the Saito-Kurokawa lifting. 
Both are theta liftings (cf. \S\S \ref{ssec:81}, \ref{ssec:82} and \ref{ssec:94}). 

More importantly, we can describe how the representation $A^{\psi_\bff}_m(\mu_\bff)$ is embedded in $\frkC^{(m)}_{(2\kap+m)/2}$ quite explicitly. 
Fix a Haar measure $\d b=\otimes_v\d b_v$ on  $\Sym_m(\AAf)$. 
Then we can associate to each $\xi\in\Sym_m^+$ a basis vector $w^{\mu_\bff}_\xi$ of the one-dimensional vector space $\Hom_{\Sym_m(\AAf)}(I^{\psi_\bff}_m(\mu_\bff)\circ\frkn,\psi^\xi_\bff)$ by 
\[w^{\mu_\bff}_\xi(\otimes_{v\notin\frkS_\infty}h_v)=\prod_{v\notin\frkS_\infty}w^{\mu_v}_\xi(h_v), \]
where $w^{\mu_v}_\xi\in\Hom_{\Sym_m(F_v)}(I_m^{\psi_v}(\mu_v)\circ\frkn,\psi^\xi_v)$ is defined by 
\begin{multline*}
w^{\mu_v}_\xi(h_v)=\int_{\Sym_m(F_v)}h_v\left(\left(\begin{pmatrix} 0 & \ono_m \\ -\ono_m & 0 \end{pmatrix}\frkn(b_v),1\right)\right)\overline{\psi_v^\xi(b_v)}\,\d b_v\\
\times\frac{|\det\xi|_v^{(m+1)/4}}{L\bigl(\frac{1}{2},\mu_v\hat\chi_v^{\det\xi}\bigl)}\prod_{j=1}^{[(m+1)/2]}L(2j-1,\mu_v^2)\times\begin{cases}
1 &\text{if $2\nmid m$, }\\
L\left(\frac{m+1}{2},\mu_v\hat\chi_v^{(-1)^{m/2}}\right) &\text{if $2|m$. } 
\end{cases}
\end{multline*}
The integral diverges but makes sense as it stabilizes.  
%\[\int_{\frkp_v^{-N}S_n(\frko_v)}f_v\left(\begin{pmatrix} 0 & \ono_n \\ \ono_n & 0 \end{pmatrix}\bfn(z_v)\right)\overline{\psi_v(\tau(Bz_v))}\,\d z_v\] is independent of $N$ if $N$ is sufficiently large. 
One can check that $w^{\mu_v}_\xi(h_v)=1$ for almost all $v$. 
We write $\vrh$ for the right regular action of $\Mp(W_m)_\bff$ on the space of fuctions on $\Mp(W_m)_\bff$. 

\begin{theorem}\label{thm:12}
If $\pi_\bff$ appears in $\frkC_{2\kap}$ and $(-1)^{\sum_{v\in\frkS_\infty}\kap_v}\prod_{v\notin\frkS_\infty}\mu_v(-1)=1$, then $A^{\psi_\bff}_m(\mu_\bff)$ appears in the decomposition of $\frkC^{(m)}_{(2\kap+m)/2}$ with multiplicity one, and there is a set $\{c_t\}_{t\in F^\times_+}$ of complex numbers such that the $\Mp(W_m)_\bff$-intertwining embeddings $i_m^\eta:A^{\psi_\bff}_m(\mu^{}_\bff\hat\chi^\eta_\bff)\hookrightarrow\frkC^{(m)}_{(2\kap+m)/2}$ are given for all $m$ and $\eta\in F^\times_+$ by 
\[i^\eta_m(h)_{\til\Del}(\calz)=\sum_{\xi\in\Sym^+_m}|\det\xi|^{(2\kap+m)/4}c_{\eta\det\xi}\bfe_\infty(\tr(\xi\calz))w_\xi^{\mu^{}_\bff\hat\chi^\eta_\bff}(\vrh(\til\Del)h). \] 
\end{theorem}

The constant $c_t$ is a mysterious part of the $t^{\mathrm{th}}$ Fourier coefficient of a Hilbert cusp form of weight $\kap+\frac{1}{2}$ whose square is related to the central value $L\bigl(\frac{1}{2},\pi^{}_\bff\otimes\hat\chi^t_\bff\bigl)$ (cf. Theorem 12.3 of \cite{HI}). 
The formula of Fourier coefficients looks like the classical Maass relation. 
Theorem \ref{thm:61} constructs analogous liftings of $\pi_\bff$ to inner forms of symplectic groups of even rank, which are given by similar Fourier series with the same coefficients $\{c_t\}$. 
The series naturally extends to a cusp form on similitude groups for even $m$ (see Remark \ref{rem:61}(\ref{rem:611})). 

Theorem \ref{thm:12} is a generalization of the lifting constructed by Ikeda \cite{I2}, where he discussed the case in which $F=\QQ$, $m$ is even and $\mu_p\hat\chi^{(-1)^{m/2}}_p$ is an unramified unitary character of $\QQ_p^\times$ for all rational primes $p$. 
The proof in \cite{I2} uses the algebraic independence of the $p^{-s}$ and works only over $\QQ$. 
Furthermore, this method cannot apply to nonsplit inner forms of symplectic groups even when $F=\QQ$. 
Subsequently, Ikeda invented a more general approach and proved in his unpublished preprint that the representation occurs in the space $\frkC^{(m)}_{(2\kap+m)/2}$ with multiplicity one. 
Later, Yamana refined and generalized this new approach, giving the explicit Fourier expansions.  
The present article was written finally by combining Yamana's manuscript with Ikeda's original preprint. 

The method is applicable to Maass forms in principle. 
Indeed, the proof should require only suitable estimates of real analytic degenerate Whittaker functions, which guarantees convergence of the Fourier series, and its inductive structure analogous to Lemma \ref{lem:74}. 
%The approaches taken here have distinct advantages. The proof in \cite{I2} uses the Fourier coefficient formula of the Siegel Eisenstein series, but it does not appear in this paper. 

\subsection*{Acknowledgement}
Ikeda is partially supported by the JSPS KAKENHI Grant Number 26610005.
Yamana is partially supported by JSPS Grant-in-Aid for Young Scientists (B) 26800017. 
This paper was partly written while Yamana was visiting 
%the Erwin Schr\"{o}dinger International Institute for Mathematical Physics and 
University of Rijeka, and he thanks Neven Grbac for his hospitality and encouragement. 
We thank Marcela Hanzer for suggesting the unitarity of $A^{\psi_\bff}_m(\mu_\bff)$. 

%%%%%%%%%%%%%%%%%%%%%%%%%%%%%%%%%%%%%%%%%%%%%%%%%%%%%%%%%%%%%%%%%%%%%%%%%%%%%%%%
\section{Preliminaries}\label{sec:2}

%%%%%%%%%%%%%%%%%%%%%%%%%%%%%%%%%%%%%%%%%%%%%%%%%%%%%%%%%%%%%%%%%%%%%%%%%%%%%%%%
\subsection{Notation}\label{ssec:21}

For an associative ring $\calo$ with identity element we denote by $\calo^{\times}$ the group of all its invertible elements and by $\Mat^m_n(\calo)$ the $\calo$-module of all $m\times n$ matrices with entries in $\calo$. 
Put $\calo^m=\Mat^m_1(\calo)$, $\Mat_n(\calo)=\Mat^n_n(\calo)$ and $\GL_n(\calo)=\Mat_n(\calo)^\times$. 
The zero element of $\Mat^m_n(\calo)$ is denoted by $0$ and the identity element of the ring $\Mat_n(\calo)$ is denoted by $\ono_n$. 
The transpose of a matrix $x$ is denoted by $\trs x$. 
If $x_1,\dots,x_k$ are square matrices, then $\diag[x_1,\dots,x_k]$ denotes the matrix with $x_1,\dots,x_k$ in the diagonal blocks and $0$ in all other blocks. 
Assume that $\calo$ has an involution $a\mapsto a^\iot$. 
For a matrix $x$ over $\calo$, let $\trs x$ be the transpose of $x$ and $x^*=\trs x^\iot$ the conjugate transpose of $x$. 
Given $\eps\in\{\pm 1\}$, we let $S_n^\eps=\{z\in\Mat_m(\calo)\;|\;z^*=\eps z\}$ be the space of $\eps$-hermitian matrices of size $m$. 
Set $z[x]=x^*zx$ for matrices $z\in S_m^\eps$ and $x\in\Mat^m_n$. 
Given $\eps$-hermitian matrices $B\in S^\eps_j$ and $\vXi\in S^\eps_k$, we sometimes write $B\oplus \vXi$ instead of $\diag[B,\vXi]\in S^\eps_{j+k}$, particularly when we view them as $\eps$-hermitian forms. 
We say that $\vXi$ is represented by $B$ if there is a matrix $x\in\Mat^j_k(\calo)$ such that $B[x]=\vXi$. 

We denote by $\NN$, $\ZZ$, $\QQ$, $\RR$, $\CC$, $\RR^\times_+$, $\SS$ and $\mu_k$ the set of strictly positive rational integers, the ring of rational integers, the fields of rational, real, complex numbers, the groups of strictly positive real numbers, complex numbers of absolute value $1$ and $k$th roots of unity. 
We define the sign character $\sgn:\RR^\times\to \mu_2$ by $\sgn(x)=x/|x|$. 
When $X$ is a totally disconnected locally compact topological space or a smooth real manifold, we write $\cals(X)$ for the space of Schwartz-Bruhat functions on $X$. 

%%%%%%%%%%%%%%%%%%%%%%%%%%%%%%%%%%%%%%%%%%%%%%%%%%%%%%%%%%%%%%%%%%%%%%%%%%%%%%%%
\subsection{Quaternionic unitary groups}\label{ssec:22}

Let $F$ for the moment be an arbitrary field and $D$ a quaternion algebra over $F$, by which we understand a central simple algebra over $F$ such that $[D:F]=4$. 
We frequently regard $D$ as an algebra variety over $F$. 
We denote by $^\iot$ the main involution of $D$,  by $x^*=\trs x^\iot$ the conjugate transpose of a matrix $x\in\Mat_n(D)$, by $\nu:\GL_n(D)\to\GG_m$ the reduced norm and by $\tau:\Mat_n(D)\to\GG_{a}$ the reduced trace, where $\GG_m=\GL_1$ and $\GG_a=\Mat_1$ are the multiplicative and additive groups in one variable over $F$.  
If $n=1$, then $\nu(x)=xx^{\iot}$ and $\tau(x)=x+x^{\iot}$ for $x\in D$. 
Put 
\begin{align*}
S_n&=\{B\in\Mat_n(D)\;|\;B^* =-B\}, & 
S^\nd_n&=S_n\cap\GL_n(D).  
\end{align*}
When $n=1$, we simply write $D_-^{}=S_1^{}$ and $D^\nd_{-}=S^\nd_1$. 
We identify $S_n$ with the space of $D$-valued skew hermitian forms on the right $D$-module $D^n$, by which we understand an $F$-linear map $B:D^n\times D^n\to D$ such that 
\begin{align*}
B(x,y)^\iot&=-B(y,x), & 
%B(x,y+z)&=B(x,y)+B(x,z), &
B(xa,yb)&=a^\iot B(x,y)b &
(a,b&\in D;\; x,y\in D^n). 
\end{align*}

We define the algebraic group $\calg_n$ by
\[\calg_n=\{g\in\GL_{2n}(D)\;|\;gJ_ng^*=\lam_n(g)J_n\text{ with }\lam_n(g)\in \GG_{m}\}, \]
where 
\[J_n=\begin{pmatrix} 0 & \ono_n \\ \ono_n & 0\end{pmatrix}\in\GL_{2n}(F). \]
%Note that $\calg_n$ is disconnected as an algebraic group. We denote the identity component of $\calg_n$ by $\calg_n^\circ$. 
We call $\lam_n:\calg_n\to\GG_{m}$ the similitude character. 
We are interested in its kernel $G_n=\{g\in\calg_n\;|\;\lam_n(g)=1\}$. 
For $A\in\GL_n(D)$, $z\in S_n$ and $t\in\GG_m$ we define matrices in $\GL_{2n}(D)$ by 
\begin{align*}
\bfm(A)&=\begin{pmatrix} A & 0 \\ 0 & (A^{-1})^{*}\end{pmatrix}, & 
\bfn(z)&=\begin{pmatrix} \ono_n & z \\ 0 & \ono_{n}\end{pmatrix}, &
\bfd(t)&=\begin{pmatrix} t\cdot\ono_n & 0 \\ 0 & \ono_n\end{pmatrix}. 
\end{align*}
Let $P_n$ be the parabolic subgroup of $G_n$ which has a Levi factor $M_n=\{\bfm(A)\;|\;A\in\GL_n(D)\}$ and the unipotent radical $N_n=\{\bfn(z)\;|\;z\in S_n\}$. 

%%%%%%%%%%%%%%%%%%%%%%%%%%%%%%%%%%%%%%%%%%%%%%%%%%%%%%%%%%%%%%%%%%%%%%%%%%%%%%%%
\subsection{The split case}\label{ssec:23}

We include the case in which $D$ is the matrix algebra $\Mat_2(F)$ of degree $2$ over $F$. Let us now take this case. 
We often identify $\Mat_m(D)$ with $\Mat_{2m}(F)$ by viewing an element $(x_{ij})$ of $\Mat_m(D)$ as a matrix of size $2m$ whose $(i,j)$-block of size $2$ is $x_{ij}$. Put \begin{align*}J&=\begin{pmatrix} 0 & 1 \\ -1 & 0\end{pmatrix}\in\GL_2(F), &B_m&=\diag[\underbrace{J,\dots,J}_m]\in\GL_{2m}(F). \end{align*}
Then we easily see that $X^*=B_m^{-1}\trs XB_m$ for $X\in\Mat_m(D)$, where $\trs X$ denotes the transpose of $X$ as a matrix of size $2m$. 
%We define the split group by \[\scrg_n=\left\{g\in\GL_{4n}\;\biggl|\;g\begin{pmatrix} 0 & \ono_{2n} \\ -\ono_{2n} & 0\end{pmatrix}\trs g=\begin{pmatrix} 0 & \ono_{2n} \\ -\ono_{2n} & 0\end{pmatrix}\right\}. \]
We are led to \begin{align*}
\sig_nG_n\sig_n^{-1}&=Sp_{2n}, & B_nS_n&=\Sym_{2n}, 
\end{align*} where $\sig_n=\diag[\ono_{2n},B_n]$. 
Thus $G_n$ is an inner form of $Sp_{2n}$. 

%%%%%%%%%%%%%%%%%%%%%%%%%%%%%%%%%%%%%%%%%%%%%%%%%%%%%%%%%%%%%%%%%%%%%%%%%%%%%%%%
\subsection{The case $n=1$}\label{ssec:24}

When $\calg$ is an algebraic group over a field $F$ and $Z$ is its center, we write $P\calg$ for the adjoint group $\calg/Z$. 
%We write $Z_k$ for the center of $\GL_k(D)$. 
It is important to note that the group $P\calg_1$ is isomorphic to a certain orthogonal group. 
To see this relation, we recall some well-known facts on Clifford algebras. 
The basic reference is \cite{Sh4}. 
For the time being, we will take $V$ to be a finite dimensional vector space over a field $F$ of characteristic different from $2$, and let $q_V:V\to F$ be a nondegenerate symmetric $F$-bilinear form. 

A Clifford algebra of $(V,q_V)$ is an $F$-algebra $A$ with an $F$-linear map $p: V\to A$ satisfying the following properties:
\begin{enumerate}
\item[$\bullet$] $A$ has an identity element, which we denote by $\ono_A$; 
\item[$\bullet$] $A$ as an $F$-algebra is generated by $p(V)$ and $\ono_A$; 
\item[$\bullet$] $p(x)^2=q_V(x)\ono_A$ for every $x\in V$;
\item[$\bullet$] $A$ has dimension $2^\ell$ over $F$, where $\ell=\dim V$. 
\end{enumerate}
It is known that such a pair $(A,p)$ is unique up to isomorphism. 
Moreover, $p$ is injective, and as such, $V$ can be viewed as a subspace of $A$ via $p$. 
We denote this algebra $A$ by $A(V)$. 
The basic equalities are $xy+yx=q_V(x+y)-q_V(x)-q_V(y)$ for $x,y\in V$. 

There is an automorphism $\bet\mapsto\bet'$ of $A(V)$ such that $v'=-v$ for every $v\in V$. 
Similarly, there is an anti-automorphism $\bet\mapsto\bet^\rho$ of $A(V)$ such that $v^\rho=v$ for every $v\in V$. 
Let us put 
\begin{align*}
A^+(V)&=\{\bet\in A(V)\;|\;\bet'=\bet\}, \\
G^+(V)&=\{\bet\in A^+(V)^\times\;|\; \bet V\bet^{-1}=V\}. 
\end{align*}
Put $\mu_1(\bet)=\bet\bet^\rho$ for $\bet\in G^+(V)$. 
The map $\mu_1$ gives a homomorphism of $G^+(V)$ to $F^\times$. 
%and the spin group $\Spin^\vph(F)$ of $\vph$ is defined by \[\Spin^\vph(F)=\{\bet\in G^+(V)\;|\; \lam(\bet)=1\}. \]
For $\bet\in G^+(V)$ we can define $\vth(\bet)\in\GL(V)$ by $\vth(\bet)v=\bet v\bet^{-1}$ $(v\in V)$. 
Then it is well-known that $\vth$ gives an isomorphism of $G^+(V)/F^\times$ onto the special orthogonal group 
\[\SO(V)=\{g\in\SL(V)\;|\;(gx,gy)=(x,y)\text{ for all }x,y\in V\}. \]
%Recall that the spinor norm $\sig_F:\SO^\vph(F)\to F^\times/F^{\times2}$ is defined as follows. Given $g\in\SO^\vph(F)$, take an element $\bet\in G^+(V)$ so that $\vth(\bet)=g$. Then $\lam(\bet)$ is a well-defined element of $F^\times/F^{\times2}$ and is denoted by $\sig_F(g)$. Clearly, \[\vth(\Spin^\vph(F))=\{g\in\SO^\vph(F)\;|\;\sig_F(g)=1\}. \] 

By restricting the symmetric bilinear form on $D^2$ given by $(x,y)\mapsto \frac{1}{2}\tau(xy^\iot)$, we obtain a three dimensional quadratic space $V_D=(D_-,q_{D_-})$ of discriminant $1$. 
%When $D=\Mat_2(F)$, we write $V^{+}=V_D$.  
In what follows we take $V=Fe\oplus V_D\oplus Fe'$ and define the quadratic form $q_V$ by 
\begin{align*}
q_{D_-}(x)&=-\nu(x), &
q_V(re+x+r'e')&=rr'+q_{D_-}(x) 
\end{align*}
for $r,r'\in F$ and $x\in D_-$. 
%Let $Q_1$ be the parabolic subgroup of $\SO(V)$ which stabilizes the isotropic line spanned by $e$.  

\begin{lemma}\label{lem:21}
Notation and assumption being as above, there is an $F$-linear ring homomorphism $\vPs:A(V)\to\Mat_2(D)$ such that 
\[\vPs(\bet^\rho)=\begin{pmatrix} 0 & 1 \\ 1 & 0 \end{pmatrix}\vPs(\bet)^{*}\begin{pmatrix} 0 & 1 \\ 1 & 0 \end{pmatrix}\]
for all $\bet\in A(V)$ and whose restriction gives isomorphisms 
\begin{align*}
A^{+}(V)&\simeq\Mat_{2}(D), & 
G^+(V)&\simeq \calg_1. 
\end{align*}
Furthermore, for given $t\in F^\times$, $a\in D^\times$ and $z\in D_-$, 
\begin{align*}
&(\vth\circ\vPs^{-1})(\bfd(t)\bfm(A))(re+x+r'e')=t\nu(A)re+AxA^{-1}+(t\nu(A))^{-1}r'e', \\ 
&(\vth\circ\vPs^{-1})(\bfn(z))(re+x+r'e')=(r-\tau(zx)+r'\nu(z))e+(x+r'z)+r'e', \\
&(\vth\circ\vPs^{-1})\left(\begin{pmatrix} 0 & 1 \\ 1 & 0 \end{pmatrix}\right)(re+x+r'e')=r'e-x+re'.  
\end{align*}
\end{lemma}

\begin{proof}
This isomorphism is explained in \S\S A4.2 and A4.3 of \cite{Sh4}, to which we refer for additional explanation. 
We shall give a brief account for the convenience of the reader. 
Define a map $p:V\to\Mat_2(D)\oplus \Mat_2(D)$ by $p(v)=(p'(v),-p'(v))$, where $p':V\to\Mat_2(D)$ is defined by 
\begin{align*}
p'(re+x+r'e')&=\begin{pmatrix} x & r \\ r' & -x \end{pmatrix} 
&(r,r'&\in F,\; x\in D_{-}).
\end{align*}
Notice that $p(v)^2=(q_V(v),q_V(v))$ for every $v\in V$. 
Since the elements $p(v)$ for all $v\in V$ generate $\Mat_2(D)\oplus\Mat_2(D)$ (see \S 4.2 of \cite{Sh4}), the pair $(\Mat_2(D)\oplus\Mat_2(D),p)$ is a Clifford algebra of $(V,q_V)$. 
It can clearly be seen that 
\[A^+(V)=\{(\bet,\bet)\;|\;\bet\in\Mat_2(D)\}. \]
We define $\vPs$ as the projection onto the first factor. 
The canonical involution of $A^+(V)$ can be given as above, as this is so for every elements of $p(V)$. 
Hence we know that $\vPs(G^{+}(V))\subset \calg_1$ and conclude that $\mu_1$ coincides with the similitude character $\lam_1$ of $\calg_1$. 
To prove the reverse inclusion, we first observe that 
\[p'(V)\begin{pmatrix} 0 & 1 \\ 1 & 0 \end{pmatrix}=\{\bet\in\Mat_2(D)\;|\;\bet^*=\bet\}. \]
If $g\in\calg_1$, then 
\[gp'(V)g^{-1}=\lam_1(g)^{-1}gp'(V)\begin{pmatrix} 0 & 1 \\ 1 & 0 \end{pmatrix}g^{*}\begin{pmatrix} 0 & 1 \\ 1 & 0 \end{pmatrix}=p(V)\]
and hence $\Psi^{-1}(g)=(g,g)\in G^+(V)$. 
The last assertion can be verified by a simple calculation. 
\end{proof} 

%%%%%%%%%%%%%%%%%%%%%%%%%%%%%%%%%%%%%%%%%%%%%%%%%%%%%%%%%%%%%%%%%%%%%%%%%%%%%%%%
\section{Degenerate Whittaker functions}\label{sec:3}

The ground field $F$ is a totally real number field or its completion. 
Excluding the case of the real field, we let $\frko$ be the maximal order of $F$ and fix a maximal order $\calo$ of $D$. 
In the real case we set $\psi=\bfe|_\RR$. 
When $F$ is an extension of $\QQ_p$, we define the character $\psi$ of $F$ by $\psi(x)=\bfe(-y)$ with $y\in\QQ$ such that $\Tr_{F/\QQ_p}(x)-y\in\ZZ_p$. 
In the global case we put $\psi_\infty(x)=\prod_{v\in\frkS_\infty}\bfe(x_v)$, $\psi_\bff(x)=\prod_{v\notin\frkS_\infty}\psi_v(x_v)$ and $\psi(x)=\psi_\infty(x_\infty)\psi_\bff(x_\bff)$ for $x\in\AA$. 

%%%%%%%%%%%%%%%%%%%%%%%%%%%%%%%%%%%%%%%%%%%%%%%%%%%%%%%%%%%%%%%%%%%%%%%%%%%%%%%%
\subsection{Degenerate principal series}\label{ssec:31}

In this and the next subsections $F$ is a completion at a nonarchimedean prime. 
We denote by $q$ the order of the residue field of the valuation ring $\frko$,   by $\alp(t)=|t|$ the normalized absolute value of $t\in F^\times$ and by $\hat\chi^t$ the quadratic character of $F^\times$ associated to $F(\sqrt{t})/F$ via class field theory. 
For $B\in S_n^\nd$ we set $\hat\chi^B=\hat\chi^{(-1)^n\nu(B)}$.  
We write $\Ome(F^\times)$ for the group of all continuous homomorphisms from $F^\times$ to $\CC^\times$. 
Continuous homomorphisms of the form $\alp^s$ for some $s\in\CC$ are called unramified. 
Define $\sig(\mu)$ as the unique real number such that $\mu\alp^{-\sig(\mu)}$ is unitary. 

For $\mu\in\Ome(F^\times)$ the normalized induced representation $J_n(\mu)$ is realized on the space of smooth functions $f:\calg_n\to\CC$ satisfying  
\[f(\bfd(t)\bfm(A)\bfn(z)g)=\mu(t^n\nu(A))|t^n\nu(A)|^{(2n+1)/2}f(g)\]
for all $t\in F^\times$, $A\in\GL_n(D)$, $z\in S_n$ and $g\in\calg_n$. 
We denote its restriction to $G_n$ by $I_n(\mu)$. 
Since $\calg_n=\{\bfd(t)\;|\;t\in F^\times\}\ltimes G_n$, these representations can also be realized on the space of functions $f:G_n\to\CC$ satisfying  
\[f(\bfm(A)\bfn(z)g)=\mu(\nu(A))|\nu(A)|^{(2n+1)/2}f(g)\]
for all $A\in\GL_n(D)$, $z\in S_n$ and $g\in G_n$. 

Let $B\in S_n$. 
We define the character $\psi^B:S_n\to\SS$ by $\psi^B(z)=\psi(\tau(B z))$ for $z\in S_n$. 
For any smooth representation $\vPi$ of $G_n$ we put 
\[\Wh_B(\vPi)=\Hom_{S_n}(\vPi\circ\bfn,\psi^B). \]

\begin{proposition}[{\cite{Ka,KR,Y2}}]\label{prop:31}
\begin{enumerate}
\renewcommand\labelenumi{(\theenumi)}
\item\label{prop:311} If $-\frac{1}{2}<\sig(\mu)<\frac{1}{2}$, then $I_n(\mu)$ is irreducible and unitary. %, self-contragredient and $I_n(\mu)\simeq I_n(\mu^{-1})$. 
\item\label{prop:312} $\dim\Wh_B(I_n(\mu))=1$ for all $\mu\in\Ome(F^\times)$ and $B\in S_n^\nd$. 
\item\label{prop:313} Assume that $\mu^2=\alp$. 
Then $I_n(\mu)$ has a unique irreducible subrepresentation $A_n(\mu)$, which is unitary.  
Moreover, $A_n(\mu)$ is the unique irreducible subrepresentation of $J_n(\mu)$. 
Furthermore, $\Wh_B(A_n(\mu))$ is nonzero if and only if $\hat\chi^B\neq \mu\alp^{-1/2}$. 
\end{enumerate}
\end{proposition}

\begin{proof}
The module structure of $I_n(\mu)$ is determined by Kudla-Rallis \cite{KR} in the symplectic case and by the author \cite{Y2} in the quaternion case. 
The unitarity follows from the general fact on irreducible subquotients of ends of complementary series explained in Section 3 of \cite{T}.  
The second part is proved in \cite{Ka}. 
We will prove (\ref{prop:313}).  
Since $J_n(\bfd(t),\mu)A_n(\mu)$ is an irreducible submodule of $I_n(\mu)$, we know that $J_n(\bfd(t),\mu)A_n(\mu)=A_n(\mu)$ for all $t\in F^\times$.  
It is known that  
\[I_n(\mu)/A_n(\mu)\simeq\sum_{\hat\chi^{B'}=\mu\alp^{-1/2}}R^\psi_n(B'), \] 
where $B'$ extends over all equivalence classes of nondegenerate skew hermitian matrices of size $n$ with character $\mu\alp^{-1/2}$ and $R^\psi_n(B')$ is an irreducible representation arising via the Weil representation of $G_n\times \U(B')$ associated to $B'$, where $\U(B')=\{g\in\GL_n(D)\;|\;B'[g]=B'\}$.  
In the symplectic case $\Wh_B(R^\psi_n(B'))$ is nonzero if and only if $B$ is equivalent to $B'$ by Lemma 3.5 of \cite{KR}. 
One can see that this result is valid in the quaternion case by a basic calculation based on Lemme on p.~73 of \cite{MVW}. 
The claimed fact derives from the exactness of the Jacquet functor combined with these observations.  
\end{proof}

%%%%%%%%%%%%%%%%%%%%%%%%%%%%%%%%%%%%%%%%%%%%%%%%%%%%%%%%%%%%%%%%%%%%%%%%%%%%%%%%
\subsection{Jacquet integrals}\label{ssec:32}

Put
\begin{align*}
\calr_n&=S_n\cap\Mat_n(\calo), & 
\scrr_n&=\{B\in S_n\;|\;\tau(B\calr_n)\subset\frko\}, &
\scrr_n^\nd&=\scrr_n\cap S_n^\nd.    
\end{align*}
For an ideal $\frkc$ of $\frko$ we put 
\[ \Gam_n[\frkc]=\left\{\begin{pmatrix} a & b \\ c & d \end{pmatrix}\in G_n\;\biggl|\; a,d\in\Mat_n(\calo),\;b\in\frkc^{-1}\Mat_n(\calo),\;c\in\frkc\Mat_n(\calo)\right\}. \]
Define a Haar measure $\d z$ on $S_n$ so that the measure of $\calr_n$ is $1$.  
For $g\in G_n$ the quantity $\vep_\frkc(g)$ is defined by writing $g=pk$ with $p=\bfm(A)\bfn(z)\in P_n$, $k\in\Gam_n[\frkc]$, and setting $\vep_\frkc(g)=|\nu(A)|$. 
For $\mu\in\Ome(F^\times)$, $f\in I_n(\mu)$ and $s\in\CC$ we define $f^{(s)}\in I_n(\mu\alp^s)$ by $f^{(s)}(g)=f(g)\vep_\frko(g)s$. 
For $B\in S_n^\nd$ the integral 
\[\bfw^{\mu\alp^s}_B(f^{(s)})=\int_{S_n}f^{(s)}(J_n\bfn(z))\overline{\psi^B(z)}\,\d z\]
defines a formal Dirichlet series in the variable $s$, which is absolutely convergent for $\Re s>n+\frac{1}{2}-\sig(\mu)$. 
Actually, the integral stabilizes and consequently, it is a polynomial of $q^{-s}$, from which we can evaluate $\bfw^{\mu\alp^s}_B(f^{(s)})$ at $s=0$ so as to get a basis vector $\bfw^\mu_B\in\Wh_B(I_n(\mu))$. 
From now on we assume that $\sig(\mu)>-\frac{1}{2}$ and set
\[w^\mu_B(f)=|\nu(B)|^{(2n+1)/4}\bfw^\mu_B(f)\frac{L\left(n+\frac{1}{2},\mu\right)}{L\left(\frac{1}{2},\mu\hat\chi^B\right)}\prod_{j=1}^nL(2j-1,\mu^2). \]
We write $\vrh$ (resp. $\wp$) for the right regular action of $G_n$ (resp. $\calg_n$) on the space of fuctions on $G_n$ (resp. $\calg_n$). 

\begin{lemma}\label{lem:31}
Let $B\in S_n^\nd$ and $\mu\in\Ome(F^\times)$. 
Assume that $\sig(\mu)>-\frac{1}{2}$. 
\begin{enumerate}
\renewcommand\labelenumi{(\theenumi)}
\item\label{lem:311} $0\neq w^\mu_B\in\Wh_B(I_n(\mu))$. 
\item\label{lem:312} When $\mu^2=\alp$, the restriction of $w^\mu_B$ to $A_n(\mu)$ is nonzero if and only if $\hat\chi^B\neq\mu\alp^{-1/2}$. 
\item\label{lem:313} If $t\in F^\times$ and $A\in\GL_n(D)$, then 
\[w^\mu_B\circ\wp(\bfd(t)\bfm(A))=\mu(t^n\nu(A))^{-1}w^\mu_{tB[A]}. \]
\end{enumerate}
\end{lemma}

\begin{proof}
The first part is clear. 
The second part is evident from Proposition \ref{prop:31}. 
The third part can be verified by obvious changes of variables. 
\end{proof}

The following result is derived in \cite{Y4}. 

\begin{lemma}\label{lem:33}
Let $f\in I_n(\mu)$. 
For any compact subset $C$ of $G_n$ there are positive constants $c$ and $M$ such that for all $\Del\in C$ 
\[|w^\mu_B(\vrh(\Del)f)|\leq c|\nu(B)|^{-M}. \]
\end{lemma}

The Siegel series associated to $B\in \scrr_n$ is defined by 
\[b(B,s)=\sum_{z\in S_n/\calr_n} \psi'(-\tau(Bz))\nu[z]^{-s}, \]
where $\nu[z]=[z\calo^n+\calo^n:\calo^n]^{1/2}$ and $\psi'$ is an arbitrarily fixed additive character on $F$ of order zero. 
We define the function $\gam(B,s)$ by 
\[\gam(B,s)=\zet(s)^{-1}L(s-n,\hat\chi^B)\times\begin{cases}
\prod_{j=1}^n\zet(2s-2j)^{-1} &\text{if $D\simeq\Mat_2(F)$, }  \\
\prod_{j=1}^{\strut{[n/2]}}\zet(2s-4j)^{-1} &\text{otherwise. }
\end{cases}\]
Put $F(B,q^{-s})=b(B,s)\gam(B,s)^{-1}$.
Then $F(B,X)$ is a polynomial of degree $f^B$ with constant term $1$ by \cite{I5,Y4} (see Section \ref{sec:10} for the definition of $f^B$). 

We denote the different of $F/\QQ_p$ by $\frkd$. 

\begin{lemma}\label{lem:34}
If $B\in\scrr_n^\nd$, then there is a constant $c_s$ independent of $B$ such that $w_B^{\alp^s}(\vep_\frkd^{s+(2n+1)/2})=c_s|\nu(B)|^{(2n+1)/4}F(B,q^{-s-(2n+1)/2})$. 
If $\frkd=\frko$ and $D$ is not a division quaternion algebra, then $c_s=1$ and there is a positive constant $M$, depending only on $n$, such that $|w^{\alp^s}_B(\vep_\frkd^{s+(2n+1)/2})|\leq |\nu(B)|^{-M}$ for all $B\in S_n^\nd$ and $-\frac{1}{2}<\Re s<\frac{1}{2}$. 
\end{lemma}

\begin{proof}
The first part is Lemma 4.5 of \cite{Y4}. 
Ikeda shows that when $\Re s>0$, 
\[|F(B,q^{-s})|\leq |\nu(B)|^{-(13n^2+13n+4)/2}\]
in the proof of Lemma 4.1 of \cite{I2}, from which we obtain the desired estimate. 
\end{proof}

%%%%%%%%%%%%%%%%%%%%%%%%%%%%%%%%%%%%%%%%%%%%%%%%%%%%%%%%%%%%%%%%%%%%%%%%%%%%%%%%
\subsection{Degenerate Whittaker functions: the archimedean case}\label{ssec:33}

We discuss the case in which $F=\RR$ and $D=\Mat_2(\RR)$. 
We define the additive character of $\CC$ by $\bfe(z)=e^{2\pi\iu z}$ for $z\in\CC$. 
%When $S\in\Mat_{2n}(\RR)$ is a symmetric matrix, we write $S>0$ if $S[x]>0$ for all $0\neq x\in \RR^{2n}$. We call a skew hermitian matrix $B$ over $\Mat_2(\RR)$ positive definite if $BB_n\in\Sym_{2n}^+$. 
Put 
\begin{align*}
S_n^+&=\{B\in S_n\;|\;BB_n\in\Sym_{2n}(\RR)^+\}, & 
\calg^+_n&=\{g\in\calg_n\;|\;\lam_n(g)>0\}. 
\end{align*}
%Let $S_n^+$ be the set of positive definite skew hermitian matrices in $S_n$. Let $\calg^+_n=\{g\in\calg_n\;|\;\lam_n(g)>0\}$ be the identity component of the real reductive group $\calg_n$. 
We can define the action of $\calg^+_n$ on the space 
\[\frkH_n=\{Z\in\Mat_{2n}(\CC)\;|\;\trs(ZB_n^{-1})=ZB_n^{-1},\;\Im(ZB_n^{-1})\in\Sym_{2n}(\RR)^+\} \] 
and the automorphy factor on $\calg^+_n\times\frkH_n$ by 
\begin{align*}
gZ&=(AZ+B)(CZ+D)^{-1}, & 
j(g,Z)&=\nu(g)^{-1/2}\nu(CZ+D)
\end{align*}
for $Z\in\frkH_n$ and $g=\begin{pmatrix} A & B \\ C & D\end{pmatrix}\in\calg^+_n$ with matrices $A,B,C,D$ of size $n$ over $\Mat_2(\RR)$. 
There is a biholomorphic isomorphism from $\frkH_n$ onto the Siegel upper half space $\calh_{2n}$ (cf. \S \ref{ssec:23}). 
Define the origin of $\frkH_n$ and the standard maximal compact subgroup of $G_n$ by 
\begin{align*}
\bfi&=\iu B_n\in\frkH_n, & 
K_n&=\{g\in G_n\;|\;g(\bfi)=\bfi\}. 
\end{align*}
For $\ell\in\NN$ and $B\in S_n^+$ we define a function $W^{(\ell)}_B:\calg^+_n\to\CC$ by 
\[W^{(\ell)}_{B}(g)=\nu(B)^{\ell/2}\bfe(\tau(Bg(\bfi)))j(g,\bfi)^{-\ell}. \]
Clearly, 
\beq
W^{(\ell)}_B(\bfn(z)\bfd(t)\bfm(A)gk)=\bfe(\tau(Bz))\sgn(\nu(A))^\ell W^{(\ell)}_{tB[A]}(g)j(k,\bfi)^{-\ell} \label{tag:31}
\eeq 
for $z\in S_n$, $A\in\GL_{2n}(\RR)$, $t\in\RR^\times_+$, $g\in\calg^+_n$ and $k\in K_n$. 

%%%%%%%%%%%%%%%%%%%%%%%%%%%%%%%%%%%%%%%%%%%%%%%%%%%%%%%%%%%%%%%%%%%%%%%%%%%%%%%%
\subsection{Degenerate Whittaker functions: the global case}\label{ssec:34}

Until the end of the next section $D$ is a totally indefinite quaternion algebra over a totally real number field $F$. 
%We say that $v$ is ramified or unramified in $D$ according as $D_v$ is division or $D_v\simeq\Mat_2(F_v)$. 
For each place $v$ of $F$ and an algebraic group $\calv$ defined over $F$, let $F_v$ be the $v$-completion of $F$ and put $\calv_v=\calv(F_v)$ to make our exposition simpler. The ad\`{e}le group, the finite part of the ad\`{e}le group, the infinite part of the ad\`{e}le group and its connected component of the identity are denoted by $\calv(\AA)$, $\calv(\AAf)$, $\calv(\AA_\infty)$ and $\calv(\AA_\infty)^+$, respectively. 
For an ad\`{e}le point $x\in\calv(\AA)$ we denote its projections to $\calv(\AAf)$, $\calv(\AA_\infty)$ and $\calv_v$ by $x_\bff$, $x_\infty$ and $x_v$, respectively. 

We will denote the group of totally positive elements of $F$ by $F^\times_+$. 
Put 
\[S_n^+=\{B\in S_n(F)\;|\;B\in S_n(F_v)^+\text{ for all }v\in\frkS_\infty\}. \]
When $n=1$, we write $D_-^+=S_1^+$. 
%For $\ell=(\ell_v)\in\NN^d$ we define a character $\sgn_\ell:\GL_n(D_\infty)\to\mu_2$ and a function $j_\ell:\calg^+_{n,\infty}\times\frkH_{2n,\eps}\to\CC^\times$ by \begin{align*} \sgn_\ell(A)&=\prod_{v\in\frkS_\infty}\sgn(\nu(A_v))^{\ell_v}, & j_\ell(g,Z)&=\prod_{v\in\frkS_\infty}j(g_v,Z_v)^{\ell_v} \end{align*}
For $B\in S_n(F)$ we define the characters $\psi^B:S_n(\AA)\to\SS$ by $\psi^B(z)=\psi(\tau(B z))$ whose restriction to $S_n(\AAf)$ is denoted by $\psi^B_\bff$. 
%Define the character $\bfe^B_\infty:S_n(\AA_\infty)\otimes_\RR\CC\to\CC^\times$ by $\bfe^B_\infty(Z)=\prod_{v\in\frkS_\infty}\bfe(\tau(B Z_v))$. 
For any smooth representation $\vPi$ of $G_n(\AAf)$ we put 
\[\Wh_B(\vPi)=\Hom_{S_n(\AAf)}(\vPi\circ\bfn,\psi^B_\bff). \]

For $v\notin\frkS_\infty$ and $\mu_v,\lam_v\in\Ome(F_v^\times)$ let $I(\mu_v,\lam_v)$ denote the representation of $\PGL_2(F_v)$ on the space of all smooth functions $f$ on $\GL_2(F_v)$ satisfying 
\[f\left(\begin{pmatrix} a & b \\ 0 & d\end{pmatrix}g\right)=\mu_v(a)\lam_v(d)|ad^{-1}|^{1/2}f(g)\]
for all $a,d\in F_v^\times$; $b\in F_v$ and $g\in\GL_2(F_v)$. 
This representation $I(\mu_v,\lam_v)$ is irreducible unless $\mu_v\lam_v^{-1}\in\{\alp,\alp^{-1}\}$. 
If $\mu_v\lam_v^{-1}=\alp$, then $I(\mu_v,\lam_v)$ has a unique irreducible submodule, which we denote by $A(\mu_v,\lam_v)$. 

In what follows we fix, once and for all, an irreducible admissible unitary generic representation $\pi_\bff$ of $\PGL_2(\AAf)$ whose local components are not supercuspidal. 
Then $\pi_\bff$ is equivalent to the unique irreducible submodule $A(\mu_\bff^{},\mu_\bff^{-1})$ of $I(\mu_\bff^{},\mu_\bff^{-1})=\otimes_{v\notin\frkS_\infty}'I(\mu_v^{},\mu_v^{-1})$ for some character $\mu_\bff$ of the finite idele group $\AAf^\times$ whose restriction to $F_v^\times$ is denoted by $\mu_v$ and which fulfills the following conditions: 
\begin{itemize}
\item $-\frac{1}{2}<\sig(\mu_v)\leq\frac{1}{2}$ for all $v$; 
\item $\mu_v^2=\alp$ whenever $\sig(\mu_v)=\frac{1}{2}$; 
\item $\mu_v$ is unramified and $\sig(\mu_v)<\frac{1}{2}$ for almost all $v$. 
\end{itemize}
%We write $\Ome^{}_\bff(\AAf^\times)$ for the set of those characters $\mu_\bff$. 
Let $\frkS_{\pi_\bff}$ be the set of nonarchimedean primes $v$ such that $\sig(\mu_v)=\frac{1}{2}$.

Let $I_n(\mu_\bff)$ and $J_n(\mu_\bff)$ be the degenerate principal series representations of $G_n(\AAf)$ and $\calg_n(\AAf)$ induced from the character $\bfd(t)\bfm(A)\mapsto\mu_\bff(t^n\nu(A))$. 
They have factorizations $I_n(\mu_\bff)\simeq\otimes_{v\notin\frkS_\infty}'I_n(\mu_v)$ and $J_n(\mu_\bff)\simeq\otimes_{v\notin\frkS_\infty}'J_n(\mu_v)$. 
For $B\in S^+_n$ Lemma \ref{lem:34} defines a nonzero vector $w^{\mu_\bff}_B\in\Wh_B(I_n(\mu_\bff))$ by 
\[w^{\mu_\bff}_B(f)=\prod_{v\notin\frkS_\infty}w^{\mu_v}_B(f_v), \]
provided that $f=\otimes_{v\notin\frkS_\infty}f_v$ is factorizable. 
Let us set 
\[A_n(\mu_\bff)=(\otimes_{v\in\frkS_{\pi_\bff}}A_n(\mu_v))\otimes(\otimes'_{v\notin\frkS_\infty\cup\frkS_{\pi_\bff}} I_n(\mu_v)). \]
We can view $A_n(\mu_\bff)$ as the unique irreducible subrepresentation of both $I_n(\mu_\bff)$ and $J_n(\mu_\bff)$. 
Put 
\[S^{\pi_\bff}_n=\{B\in S_n^+\;|\;\hat\chi^B_v\neq\mu_v\alp^{-1/2}\text{ for all }v\in\frkS_{\pi_\bff}\}. \]
When $n=1$, we will sometimes write $D^{\pi_\bff}_-=S^{\pi_\bff}_1$. 
Proposition \ref{prop:31} and Lemma \ref{lem:31} give the following result:  

\begin{lemma}\label{lem:35}
Let $B\in S^+_n$.  
Then $\Wh_B(A_n(\mu_\bff))$ is nonzero if and only if the restriction of $w^{\mu_\bff}_B$ to $A_n(\mu_\bff)$ is nonzero if and only if $B\in S^{\pi_\bff}_n$. 
\end{lemma}

%%%%%%%%%%%%%%%%%%%%%%%%%%%%%%%%%%%%%%%%%%%%%%%%%%%%%%%%%%%%%%%%%%%%%%%%%%%%%%%%
\section{Holomorphic cusp forms on quaternion unitary groups}\label{sec:4}

When $\calf$ is a smooth function on $N_n(F)\bsl G_n(\AA)$ and $B\in S_n(F)$, let\[W_B(g,\calf)=\int_{S_n(F)\bsl S_n(\AA)}\calf(\bfn(z)g)\overline{\psi^B(z)}\,\d z\]
be the $B^{\mathrm{th}}$ Fourier coefficient of $\calf$. 
For $\ell\in\ZZ^d$ and $B\in S^+_n$ we define a function $W_B^{(\ell)}:\calg_n(\AA_\infty)^+\to\CC$ by 
\[W_B^{(\ell)}(g)=\prod_{v\in\frkS_\infty}W^{(\ell_v)}_B(g_v). \]

\begin{definition}\label{def:41}
The symbol $\frkG^n_\ell$ (resp. $\frkT^n_\ell$, $\frkC^n_\ell$) denotes the space of all smooth functions $\calf$ on $\calg_n(F)\bsl\calg_n(\AA)$ (resp. $P_n(F)\bsl G_n(\AA)$, $G_n(F)\bsl G_n(\AA)$) that admit Fourier expansions of the form 
\[\calf(g)=\sum_{B\in S_n^+}\bfw_B(g_\bff,\calf)W^{(\ell)}_B(g_\infty) \]
which is absolutely and uniformly convergent on any compact neighborhood of $g=g_\infty g_\bff\in \calg_n(\AA_\infty)^+\calg_n(\AAf)$ (resp. $g=g_\infty g_\bff\in G_n(\AA)$). 
\end{definition}

\begin{remark}\label{rem:41}
Put 
\begin{align*}
\calp_n&=\{\bfd(t)p\;|\;t\in\GG_m,\;p\in P_n\}, & 
\calp^+_n&=\{\bfd(t)p\;|\;t\in F^\times_+,\;p\in P_n(F)\}. 
\end{align*} 
Since $\calg_n(\AA)=\calp_n(F)\calg_n(\AA_\infty)^+\calg_n(\AAf)$, smooth functions on $\calp_n(F)\bsl\calg_n(\AA)$ can naturally be identified with smooth functions on $\calp^+_n\bsl\calg_n(\AA_\infty)^+\calg_n(\AAf)$. 
\end{remark} 

\begin{lemma}\label{lem:41}
If $i:A_n(\mu_\bff)\to\frkT^n_\ell$ is a $G_n(\AAf)$-intertwining map, then there are complex numbers $c_B$ such that 
\[i(g,f)=\sum_{B\in S^{\pi_\bff}_n}c^{}_BW^{(\ell)}_B(g_\infty)w_B^{\mu_\bff}(\vrh(g_\bff)f)\]
for all $f\in A_n(\mu_\bff)$, $g=g_\infty g_\bff$, $g_\infty\in G_n(\AA_\infty)$ and $g_\bff\in G_n(\AAf)$. 
\end{lemma}

\begin{proof}
Notice that the coefficient $\bfw_B(g_\bff,\calf)$ is given by 
\[\bfw_B(g_\bff,\calf)=|\nu(B)|^{-\ell/2}\bfe_\infty(\tau(B(\bfi,\dots,\bfi)))W_B(g_\bff,\calf). \]
Recall that $|\nu(B)|^{\ell/2}=\prod_{v\in\frkS_\infty}|\nu(B)|_v^{\ell_v/2}$. 
In particular, for $z\in S_n(\AAf)$ 
\[\bfw_B(\bfn(z)g_\bff,\calf)=\psi^B_\bff(z)\bfw_B(g_\bff,\calf). \]
Therefore the $\CC$-linear functional $f\mapsto\bfw_B(\ono_{2n},i(f))$ belongs to the space $\Wh_B(A_n(\mu_\bff))$. 
There is a complex number $c_B$ such that $\bfw^{}_B(\ono_{2n},i(f))=c^{}_Bw^{\mu_\bff}_B(f)$ for all $f\in A_n(\mu_\bff)$ in view of Lemma \ref{lem:35}. 
\end{proof}

We associate to $f\in A_n(\mu_\bff)$, $\ell=(\ell_v)_{v\in\frkS_\infty}\in\ZZ^d$ and complex numbers $c_B$ indexed by $B\in S^{\pi_\bff}_n$ the Fourier series 
\begin{align*}
\calf_\ell(g;f,\{c_B\})&=\sum_{B\in S^{\pi_\bff}_n}c^{}_BW^{(\ell)}_B(g_\infty)w_B^{\mu_\bff}(\vrh(g_\bff)f), &
g&=g_\infty g_\bff\in G_n(\AA), 
\end{align*}
assuming that the series is absolutely convergent. 

\begin{lemma}\label{lem:42}
Notation being as above, if for any lattice $L$ in $S_n(F)$ there are positive constants $C$ and $M$ such that $|c_B|\leq CN_{F/\QQ}(\nu(B))^M$ for all $B\in S_n^+\cap L$, then the series $\calf_\ell(g;f,\{c_B\})$ is absolutely and uniformly convergent on any compact subset of $G_n(\AA)$ for every $f\in A_n(\mu_\bff)$. 
\end{lemma}

\begin{proof}
The proof goes along the same lines of the arguments in Section 4 of \cite{I2}. 
It suffices to show that the series
\[\sum_{B\in S^{\pi_\bff}_n}c^{}_B|\nu(B)|^{\ell/2} w^{\mu_\bff}_B(f)\bfe^{}_\infty(\tau(BZ))\]
is absolutely and uniformly convergent on any compact subset of $\frkH^d_n$. 
Put $\calr_n=S_n(F)\cap\Mat_n(\calo)$. 
Take a natural number $N$ such that $w^{\mu_\bff}_B(f)=0$ unless $B\in N^{-1}\calr_n$. 
Lemmas \ref{lem:33} and \ref{lem:34} say that $w^{\mu_\bff}_B(f)\leq C'N_{F/\QQ}(\nu(B))^{M'}$ for all $B\in S^{\pi_\bff}_n$ with constants $C'$ and $M'$ depending only on $f$. 
Note that $N_{F/\QQ}(\nu(B))\leq (nd)^{-nd}(\Tr_{F/\QQ}\tau(B))^{nd}$. 
The number of $B\in N^{-1}\calr_n\cap S^+_n$ such that $\Tr_{F/\QQ}\tau(B)\leq T$ is $O(T^{dn(2n+1)})$. 
From these estimates the series converges absolutely and uniformly on 
\[\{Z\in\frkH_n^d\;|\;\Im(Z_vB_n^{-1})-\eps\ono_n\in\Sym_{2n}(F_v)^+\text{ for all }v\in\frkS_\infty\}\] 
for any positive constant $\eps$. 
\end{proof}

\begin{definition}\label{def:42}
Let $T^n_\ell(\mu_\bff)$ be the vector space which consists of sets $\{c_B\}_{B\in S^{\pi_\bff}_n}$ of complex numbers such that the series $\calf_\ell(g;f,\{c_B\})$ converges absolutely and uniformly on any compact subset of $G_n(\AA)$ for all $f\in A_n(\mu_\bff)$ and such that for all $B\in S^{\pi_\bff}_n$ and $A\in\GL_n(D(F))$
\[c_{B[A]}=c_B\mu_\bff(\nu(A))^{-1}\prod_{v\in\frkS_\infty}\sgn_v(\nu(A))^{\ell_v}. \]
Let $C^n_\ell(\mu_\bff)$ (resp. $G^n_\ell(\mu_\bff)$) denote the space of coefficients $\{c_B\}\in T^n_\ell(\mu_\bff)$ such $\calf_\ell(f,\{c_B\})\in\frkC^n_\ell$ (resp. $\frkG^n_\ell$) for all $f\in A_n(\mu_\bff)$. 
\end{definition}

\begin{lemma}\label{lem:43}
\begin{enumerate}
\renewcommand\labelenumi{(\theenumi)}
\item\label{lem:431} $\calf_\ell(f,\{c_B\})\in\frkT^n_\ell$ for $\{c_B\}\in T^n_\ell(\mu_\bff)$ and $f\in A_n(\mu_\bff)$. 
\item\label{lem:432} Let $\{c_B\}\in C^n_\ell(\mu_\bff)$. 
Then $\{c_B\}\in G^n_\ell(\mu_\bff)$ if and only if $c_{tB}=c_B\mu_\bff(t)^{-n}$ for all $t\in F^\times_+$. 
\end{enumerate}
\end{lemma}

\begin{proof}
Fix $f\in A_n(\mu_\bff)$ and put $w^{(\ell)}_B(g)=W^{(\ell)}_B(g_\infty)w^{\mu_\bff}_B(\vrh(g_\bff)f)$. 
Then Lemma \ref{lem:31}(\ref{lem:313}) and (\ref{tag:31}) say that 
\[w^{(\ell)}_B(\bfn(z)\bfd(t)\bfm(A)g)=\frac{\psi^B(z)}{\mu_\bff(t^n\nu(A))}w^{(\ell)}_{tB[A]}(g)\prod_{v\in\frkS_\infty}\sgn_v(\nu(A))^{\ell_v} \]
for $z\in S_n(\AA)$, $A\in\GL_n(D(F))$, $t\in F^\times_+$ and $g\in \calg_n(\AA_\infty)^+\calg_n(\AAf)$, from which $\calf_\ell(f,\{c_B\})$ is left invariant under $P_n(F)$. 

We can define the function $\calf'_\ell(g;f,\{c_B\}): \calg_n(\AA_\infty)^+\calg_n(\AAf)\to\CC$ by 
\begin{align*}
\calf'_\ell(g;f,\{c_B\})&=\sum_{B\in S^{\pi_\bff}_n}c^{}_BW^{(\ell)}_B(g_\infty)w_B^{\mu_\bff}(\wp(g_\bff)f), & 
f&\in A_n(\mu_\bff). 
\end{align*} 
Then $\calf'_\ell(f,\{c_B\})$ is left invariant under $\calp^+_n$ if and only if $c_{tB}=c_B\mu_\bff(t)^{-n}$ for all $t\in F^\times_+$, in which case $\calf'_\ell(f,\{c_B\})$ naturally defines a cusp form in $\frkG^n_\ell$ owing to Remark \ref{rem:41}. 
\end{proof}

%%%%%%%%%%%%%%%%%%%%%%%%%%%%%%%%%%%%%%%%%%%%%%%%%%%%%%%%%%%%%%%%%%%%%%%%%%%%%%%

\section{Hilbert-Siegel cusp forms of half-integral weight}\label{sec:5}

%%%%%%%%%%%%%%%%%%%%%%%%%%%%%%%%%%%%%%%%%%%%%%%%%%%%%%%%%%%%%%%%%%%%%%%%%%%%%%%%
\subsection{The metaplectic group}\label{ssec:51}

We first review the basic facts about the metaplectic double cover of $Sp_m(\AA)$. 
This is mostly a well-known material that can be found in e.g. \cite{I,Sh3,G,HI}. 
%Let $(W_m,\La\;,\;\Ra)$ be a $2m$ dimensional symplectic vector space over a field $F$. The associated symplectic group $Sp(W_m)$ is isomorphic to $Sp_m(F)$, and we may fix such an isomorphism by choosing a Witt basis $\{\vep^{}_i,\vep'_j\}$ of $W_m$ so as to have $\La\vep^{}_i,\vep'_j\Ra=\del_{i,j}$. 
Let $Sp(W_m)=Sp_m$ be the symplectic group of rank $m$ acting on the $2m$ dimensional symplectic vector space $W_m$ over a field $F$ on the left.  
For $a\in\GL_m$ and $b\in\Sym_m$ we define matrices of size $2m$ by 
\begin{align*}
\frkm(a)&=\begin{pmatrix} a & 0 \\ 0 & \trs a^{-1}\end{pmatrix}, & 
\frkn(b)&=\begin{pmatrix} \ono_m & b \\ 0 & \ono_m\end{pmatrix}. 
\end{align*}
%The group $Sp(W_m)$ acts on the space $W_m$ on the left. 
These matrices generate the parabolic subgroup $\PP_m$ of $Sp(W_m)$ with unipotent radical $\UU_n=\{\frkn(b)\;|\;b\in\Sym_m\}$. 
Put $\Sym_m^\nd=\Sym_m\cap\GL_m$. 

Let $F$ be a local field of characteristic zero. 
We exclude the complex case. 
The metaplectic group $\Mp(W_m)$ is the nontrivial central extension of $Sp(W_m)$ by $\mu_2$.  
Let $\mu_2$ inject into the center of $\Mp(W_m)$. 
We call functions on $\Mp(W_m)$ or representations of $\Mp(W_m)$ genuine if $-1\in\mu_2$ acts by multiplication by $-1$. 
We choose the section $\bfs:Sp(W_m)\to\Mp(W_m)$ in such a way that 
\begin{align*}
\zet\bfs(g)\cdot\zet'\bfs(g')&=\zet\zet' c(g,g')\bfs(gg') & 
(\zet,\zet'&\in\mu_2;\;g,g'\in Sp_m(F))
\end{align*}
where $c(g,g')$ is the Rao two cocycle on $Sp_m(F)$. 
%Recall that $c(ugu',g'u'')=c(g,u'g')$ for $g,g'\in\SL_2(F)$ and $u,u',u''\in N$. 
The restriction of $\bfs$ to $\UU_m$ is a group homomorphism, by which we view $\UU_m$ as a subgroup of $\Mp(W_m)$. 
We will write $\til\frkm=\bfs\circ\frkm$ and $\til\frkn=\bfs\circ\frkn$. 
Note that  
\begin{align*}
\til\frkm(a)\til\frkm(a')&=\hat\chi^{\det a}(\det a')\til\frkm(aa'), & 
\til\frkm(a)\til\frkn(b)\til\frkm(a)^{-1}&=\til\frkn(ab\trs a)
\end{align*} 
for $a,a'\in \GL_m(F)$ and $b\in\Sym_m(F)$. 
For a subgroup $H$ of $Sp_m(F)$ we denote the inverse image of $H$ in $\Mp(W_m)$ by $\til H$. 
%We take the self-dual Haar measure on $F$ with respect to the Fourier transform defined by\begin{align*}\hat\phi(x)&=\int_F\phi(y)\psi(xy)\,\d y, & \phi&\in\cals(F). \end{align*}For $t\in F^\times$ there is an $8^\mathrm{th}$ root of unity $\gam(\psi^t)$ such that for all $\phi\in\cals(F)$ \[\int_F\phi(x)\psi(tx^2)\,\d x=\gam(\psi^t)|2t|^{-1/2}\int_F\hat\phi(x)\psi\left(-\frac{x^2}{4t}\right)\,\d x. \] Put $\gam^\psi(t)=\gam(\psi)/\gam(\psi^t)$. 
As in Section \ref{sec:1} we define the function $\gam^\psi:F^{\times 2}\bsl F^\times\to\mu_4$, which possesses the following properties:  
\begin{align}
\gam^\psi(tt')&=\gam^\psi(t)\gam^\psi(t')\hat\chi^t(t'), &
\gam^{\psi^t}(t')&=\gam^\psi(t')\hat\chi^t(t').  \label{tag:51}
\end{align}

When $F$ is of odd residual characteristic, there is a unique splitting $Sp_m(\frko)\hookrightarrow\Mp(W_m)$, by which we regard $Sp_m(\frko)$ as a subgroup of $\Mp(W_m)$.  
In other words there is a continuous map $\zet:Sp_m(F)\to\mu_2$ such that $c(k,k')=\zet(k)\zet(k')\zet(kk')$ for $k,k'\in Sp_m(\frko)$. 
We shall set $\zet(g)=1$ in the real and dyadic cases to make our exposition uniform. 
We use a cocycle 
\[\bfc(g,g')=c(g,g')\zet(g)\zet(g')\zet(gg')\]
with global applications in view, namely, we identity $\Mp(W_m)$ with the product $Sp_m(F)\times\mu_2$ whose group law is given by $(g,\zet)(g',\zet')=(gg',\zet\zet'\bfc(g,g'))$. 
It should be remarked that the section $\bfs$ is now given by $\bfs(g)=(g,\zet(g))$. 

The real metaplectic group acts on the Siegel upper half-space $\calh_m$ through $Sp_m(\RR)$. 
There is a unique factor of automorphy $\jmath:\Mp(W_m)\times\calh_m\to\CC^\times$ satisfying $\jmath(\zet\bfs(g),\calz)^2=\det(C\calz+D)$ for $\zet\in\mu_2$ and $g=\begin{pmatrix} * & * \\ C & D \end{pmatrix}\in\Mp(W_m)$. 
For each $\ell\in\frac{1}{2}\ZZ$ we put $J_\ell(\til g,\calz)=\jmath(\til g,\calz)^{2\ell}$. 
For each positive definite $\xi\in\Sym_m(\RR)$ we define a function on the real metaplectic group by 
\[W^{(\ell)}_\xi(\til g)=(\det\xi)^{\ell/2}\bfe(\xi\til g(\iu\ono_m))J_\ell(\til g,\iu\ono_m)^{-1}. \]
If $\zet\in\mu_2$, $a\in\GL_m(\RR)$ and $\til g\in\Mp(W_m)$, then 
\beq
W^{(\ell)}_\xi(\zet \til\frkm(a)\til g)=\zet^{2\ell}\gam^\psi(\det a)^{2\ell}W^{(\ell)}_{\xi[a]}(\til g). \label{tag:52}
\eeq

%%%%%%%%%%%%%%%%%%%%%%%%%%%%%%%%%%%%%%%%%%%%%%%%%%%%%%%%%%%%%%%%%%%%%%%%%%%%%%%%
\subsection{Representations of the metaplectic group}\label{ssec:52}

Now we assume $F$ to be nonarchimedean. 
For $\xi\in\Sym_m^\nd$ and a smooth representation $\vPi$ of $\Mp(W_m)$ we set $\Wh_\xi(\vPi)=\Hom_{\Sym_m(F)}(\vPi\circ\til\frkn,\psi^\xi)$. 
When $\mu\in\Ome(F^\times)$ and $\sig(\mu)>-\frac{1}{2}$, we define the representation $I^\psi_m(\mu)$ of $\Mp(W_m)$ and the nonzero functional $w^\mu_\xi\in\Wh_\xi(I^\psi_m(\mu))$ as in Section \ref{sec:1}. 
%let $I^\psi_m(\mu)$ be the representation of $\Mp(W_m)$ on the space of smooth functions $h:\Mp(W_m)\to\CC$ satisfying \[h(\zet\til\frkm(a)\til\frkn(b)\til g)=\zet^m\gam^\psi(\det a)^m\mu(\det a)|\det a|^{(m+1)/2}h(\til g)\] for all $\zet\in\mu_2$, $a\in\GL_m(F)$, $b\in\Sym_m(F)$ and $\til g\in\Mp(W_m)$.  

\begin{proposition}[{\cite[Lemma 3]{W2}}, \cite{Sw}]\label{prop:51}
\begin{enumerate}
\renewcommand\labelenumi{(\theenumi)}
\item\label{prop:511} If $-\frac{1}{2}<\sig(\mu)<\frac{1}{2}$, then $I_m^\psi(\mu)$ is irreducible and unitary. %and $I_m^\psi(\mu)\simeq I_m^\psi(\mu^{-1})$. 
\item\label{prop:512} $\dim\Wh_\xi(I^\psi_m(\mu))=1$ for all $\xi\in\Sym_m^\nd$ and $\mu\in\Ome(F^\times)$. 
\item\label{prop:513} Assume that $\mu^2=\alp$. 
Then $I_m^\psi(\mu)$ has a unique irreducible subrepresentation $A_m^\psi(\mu)$, which is unitary. 
Furthermore, $\Wh_\xi(A_m^\psi(\mu))$ is nonzero if and only if the restriction of $w^\mu_\xi$ to $A_m^\psi(\mu)$ is nonzero if and only if $\hat\chi^{\det\xi}\neq \mu\alp^{-1/2}$. 
\item\label{prop:514} If $F$ is not dyadic, $\psi$ is of order $0$, $\mu$ is unramified, $\xi\in\Sym_m^\nd\cap\GL_m(\frko)$ and $h(k)=1$ for $k\in Sp_m(\frko)$, then $w^\mu_\xi(h)=1$.  
\item\label{prop:515} For all $\xi\in\Sym_m^\nd$ and $a\in\GL_m(F)$  
\[w^\mu_\xi\circ\vrh(\zet\til\frkm(a))=\zet^m\gam^\psi(\det a)^m\mu(\det a)^{-1}w^\mu_{\xi[a]}. \]
\end{enumerate}
\end{proposition}

\begin{proof}
When $m$ is even, all the results are included in Proposition \ref{prop:31} and Lemma \ref{lem:31}. 
The fourth part is Theorem 16.2 of \cite{Sh3}. 
The other assertions are included in \cite{Sw} or can be derived analogously. 
\end{proof}

%For $\til g\in\Mp(W_m)$ the quantity $|a(\til g)|$ is defined by writing $\til g=\til\frkm(a)\til\frkn(b)\til k$ with $a\in\GL_m(F)$, $b\in\Sym_m(F)$ and $\til k\in \til K_m$, and setting $|a(\til g)|=|\det a|$. For $h\in I^\psi_m(\mu)$ and $s\in\CC$ we define $h^{(s)}\in I^\psi_m(\mu\alp^s)$ by $h^{(s)}(\til g)=h(\til g)|a(\til g)|^s$. 
%For $\xi\in\Sym_m^\nd$ and $h\in I^\psi_m(\mu)$ the integral \begin{multline*} w^\mu_\xi(h)=\int_{\Sym_m(F)}h(\bfs(J_m)\til\frkn(b))\overline{\psi^\xi(b)}\,\d b\\\times\frac{|\det\xi|^{(m+1)/4}}{L\bigl(\frac{1}{2},\mu\hat\chi^{\det\xi}\bigl)}\prod_{j=1}^{[(m+1)/2]}L(2j-1,\mu^2)\times\begin{cases} 1 &\text{if $2\nmid m$, }\\L\left(\frac{m+1}{2},\mu\right) &\text{if $2|m$. } \end{cases}\end{multline*}converges absolutely for $\sig(\mu)>\frac{m+1}{2}$. Such an integral is continued to a holomorphic function for $\sig(\mu)>-\frac{1}{2}$ and gives a nonzero vector of $\Wh_\xi(I^\psi_m(\mu))$. 

%%%%%%%%%%%%%%%%%%%%%%%%%%%%%%%%%%%%%%%%%%%%%%%%%%%%%%%%%%%%%%%%%%%%%%%%%%%%%%%%
\subsection{Holomorphic cusp forms on $\Mp(W_m)$}\label{ssec:53}

Let $F$ be a totally real number field. 
For each place $v$ of $F$ we adopt the notation by adding a subscript $v$ for objects associated to $F_v$. 
We define the adelic cocycle 
\begin{align*}
&\bfc:Sp_m(\AA)\times Sp_m(\AA)\to\mu_2, & \bfc(g,g')&=\prod_v\bfc_v(g_v,g'_v)
\end{align*}
for $g,g'\in Sp_m(\AA)$. 
Recall that $\bfc_v(g_v, g'_v)=1$ for almost all $v$ by the definition of the local cocycle $\bfc_v$. 
We write $\Mp(W_m)_\AA$ for the central extension of $Sp_m(\AA)$ associated to $\bfc$. 
%We define a set theoretic section $\vka:\SL_2(\AA)\to\Mp(W)_\AA$ by $\vka(g)=(g,1)$ for $g\in \Mp(W)_\AA$. 

Let $\prod_v'\Mp(W_m)_v$ denote the restricted direct product with respect to the subgroups $\{Sp_m(\frko_v)\}_{v\nmid 2}$. 
There exist a canonical embedding $\Mp(W_m)_v\hookrightarrow\Mp(W_m)_\AA$ and a canonical surjection $\prod_v'\Mp(W_m)_v\twoheadrightarrow\Mp(W_m)_\AA$. 
The image of $(\til g_v)\in\prod_v'\Mp(W_m)_v$ is also denoted by $(\til g_v)$. 
Let $\Mp(W_m)_\infty$ (resp. $\Mp(W_m)_\bff$) denote the image of $\prod_{v\in\frkS_\infty}\Mp(W_m)_v$ (resp. $\prod_{v\notin\frkS_\infty}'\Mp(W_m)_v$). 
If we are given a collection of genuine admissible representations $\sig_v$ of $\Mp(W_m)_v$ such that the space of $Sp_m(\frko_v)$-invariant vectors in $\sig_v$ is one-dimensional for almost all $v$, then we can form a genuine admissible representation of $\Mp(W_m)_\AA$ by taking a restricted tensor product $\otimes_v'\sig_v$. 

It is well-known that $\Mp(W_m)_\AA$ splits over the subgroup of rational points $Sp_m(F)$. 
An explicit splitting is given by 
\begin{align*}
\bfs:Sp_m(F)&\to\Mp(W_m)_\AA, & 
\bfs(\gam)&=(\gam,\prod_v\zet_v(\gam)), 
\end{align*}
where if $\gam\in Sp_m(F)$, then $\zet_v(\gam)=1$ for almost all $v$. 
Though the expression $\prod_v\zet_v(g_v)$ does not make sense for all $g\in Sp_m(\AA)$, we will denote the element $(g,\prod_v\zet_v(g_v))$ by $\bfs(g)$ whenever it makes sense.  
For example, $\bfs(\frkm(a)\frkn(b))$ is defined for $a\in\GL_m(\AA)$ and $b\in\Sym_m(\AA)$. 
We will regard $Sp_m(F)$ and $\UU_m(\AA)$ as subgroups of $\Mp(W_m)_\AA$ via $\bfs$. 
For $\xi\in\Sym_m(F)$ and a smooth function $\calf:\UU_m(F)\bsl\Mp(W_m)_\AA\to\CC$ let 
\[W_\xi(\til g,\calf)=\int_{\Sym_m(F)\bsl\Sym_m(\AA)}\calf(\bfs(\frkn(b))\til g)\psi^\xi(-b)\, \d b\]
denote the $\xi^\mathrm{th}$ Fourier coefficient of $\calf$. 

For $\ell\in\frac{1}{2}\ZZ^d$ and $\xi\in\Sym^+_m$ we define a function $W_\xi^{(\ell)}$ on $\Mp(W_m)_\infty$ by 
\[W_\xi^{(\ell)}(\til g)=\prod_{v\in\frkS_\infty}W^{(\ell_v)}_\xi(\til g_v). \]

\begin{definition}\label{def:51}
The symbol $\frkC^{(m)}_\ell$ (resp. $\frkT^{(m)}_\ell$) denotes the space of all smooth functions $\calf$ on $Sp_m(F)\bsl\Mp(W_m)_\AA$ (resp. $\PP_m(F)\bsl \Mp(W_m)_\AA$) that admit Fourier expansions of the form 
\[\calf(\til g)=\sum_{\xi\in\Sym_m^+}\bfw_\xi(\til g_\bff,\calf)W^{(\ell)}_\xi(\til g_\infty) \]
which is absolutely and uniformly convergent on any compact neighborhood of $\til g=\til g_\infty\til g_\bff$, where $\til g_\infty\in\Mp(W_m)_\infty$ and $\til g_\bff\in\Mp(W_m)_\bff$. 
\end{definition}

Notation being as in \S \ref{ssec:34}, we form the restricted tensor product 
\[A^{\psi_\bff}_m(\mu_\bff)=\{\otimes_{v\in\frkS_{\pi_\bff}}A^{\psi_v}_m(\mu_v)\}\otimes\{\otimes_{v\notin\frkS_\infty\cup\frkS_{\pi_\bff}}'I^{\psi_v}_m(\mu_v)\}. \]
For $\xi\in\Sym_m^+$ and a smooth representation $\vPi$ of $\Mp(W_m)_\bff$ we put 
\[\Wh_\xi(\vPi)=\Hom_{\Sym_m(\AAf)}(\vPi\circ\bfs\circ\frkn,\psi^\xi_\bff). \]
We can define $w^{\mu_\bff}_\xi\in\Wh_\xi(A^{\psi_\bff}_m(\mu_\bff))$ by setting $w^{\mu_\bff}_\xi(h)=\prod_{v\notin\frkS_\infty}w^{\mu_v}_\xi(h_v)$ for factorizable vectors $h=\otimes_vh_v\in A^{\psi_\bff}_m(\mu_\bff)$. 
Put 
\begin{align*}
\Sym^{\pi_\bff}_m&=\{\xi\in \Sym^+_m\;|\;\hat\chi^{\det\xi}_v\neq\mu_v\alp^{-1/2}\text{ for all }v\in\frkS_{\pi_\bff}\}, &
F^\times_{\pi_\bff}&=\Sym^{\pi_\bff}_1. 
\end{align*}
Proposition \ref{prop:51} tells us that $\Wh_\xi(A^{\psi_\bff}_m(\mu_\bff))$ is nonzero if and only if the restriction of $w^{\mu_\bff}_\xi$ to $A^{\psi_\bff}_m(\mu_\bff)$ is nonzero if and only if $\xi\in\Sym_m^{\pi_\bff}$. 

\begin{definition}\label{def:52}
Assume that $2\ell_v-m$ is even for every $v\in\frkS_\infty$. 
Let $T^{(m)}_\ell(\mu_\bff)$ (resp. $C^{(m)}_\ell(\mu_\bff)$) denote the vector space which consists of sets $\{c_\xi\}$ of complex numbers indexed by $\xi\in\Sym_m^{\pi_\bff}$ such that the series 
\[\calf_\ell(\til g;h,\{c_\xi\})=\sum_{\xi\in\Sym_m^{\pi_\bff}}c_\xi W^{(\ell)}_\xi(\til g_\infty)w^{\mu_\bff}_\xi(\vrh(\til g_\bff)h)\] 
belongs to $\frkT^{(m)}_\ell$ (resp. $\frkC^{(m)}_\ell$) for every $h\in A_m^{\psi_\bff}(\mu_\bff)$. 
\end{definition}

\begin{lemma}\label{lem:51}
\begin{enumerate}
\renewcommand\labelenumi{(\theenumi)}
\item\label{lem:511} $C^{(m)}_\ell(\mu_\bff)\neq\{0\}$ if and only if there is a $\Mp(W_m)_\bff$ intertwining embedding $A^{\psi_\bff}_m(\mu_\bff)\hookrightarrow\frkC_\ell^{(m)}$. 
\item\label{lem:512} Assume that $\calf_\ell(h,\{c_\xi\})$ converges for all $h\in A_m^{\psi_\bff}(\mu_\bff)$. 
Then $\{c_\xi\}\in T^{(m)}_\ell(\mu_\bff)$ if and only if for all $a\in\GL_m(F)$ and $\xi\in\Sym_m^{\pi_\bff}$,  
\[c_{\xi[a]}=c_\xi\mu_\bff(\det a)^{-1}\prod_{v\in\frkS_\infty}\sgn_v(\det a)^{(2\ell_v-m)/2}. \] 
\item\label{lem:513} If $T^{(m)}_\ell(\mu_\bff)\neq\{0\}$, then $\mu_\bff(-1)(-1)^{\sum_{v\in\frkS_\infty}(2\ell_v-m)/2}=1$. 
\item\label{lem:514} $\dim C^{(1)}_\ell(\mu_\bff)\leq 1$. 
\item\label{lem:515} If $\{c_t\}\in C^{(1)}_\ell(\mu_\bff)$, then $\{c_{\eta t}\}\in C^{(1)}_\ell(\mu^{}_\bff\hat\chi^\eta_\bff)$ for all $\eta\in F^\times_+$. 
\item\label{lem:516} We choose $0\neq\{c_t\}\in C^{(1)}_\ell(\mu_\bff)$, assuming that $C^{(1)}_\ell(\mu_\bff)\neq\{0\}$. 
Then $c_t\neq 0$ if and only if $t\in F^\times_{\pi_\bff}$ and $L(1/2,\pi^{}_\bff\otimes\hat\chi^t_\bff)\neq 0$. 
\end{enumerate}
\end{lemma}

\begin{proof}
The proof of (\ref{lem:511}) is similar to that of Lemma \ref{lem:41}. 
Proposition \ref{prop:51}(\ref{prop:515}) and (\ref{tag:52}) prove (\ref{lem:512}). 
The third statement is its simple consequence. 
The assertion (\ref{lem:514}) follows from the fact proved by Waldspurger \cite{W} that every irreducible representation of $\Mp(W_1)_\AA$ occurring in the decomposition
of the space of cusp forms on $\Mp(W_1)$ appears with multiplicity one. 
We can realize $\Mp(W_1)_\AA$ as a normal subgroup of a double cover of $\GL_2(\AA)$ constructed by using the Kubota two cocyle. 
The conjugation action of $\{\diag[a,1]\;|\;a\in F^\times\}$ on $\SL_2(\AA)$ has a lift to $\Mp(W_1)_\AA$, which preserves the subgroup $\SL_2(F)$. 
We obtain $A^{\psi_\bff}_1(\mu^{}_\bff\hat\chi^\eta_\bff)$ by conjugation by $\diag[\eta,1]$, which proves (\ref{lem:515}).   
The last assertion is Theorem 4.1 of \cite{PS}.  
\end{proof}

%%%%%%%%%%%%%%%%%%%%%%%%%%%%%%%%%%%%%%%%%%%%%%%%%%%%%%%%%%%%%%%%%%%%%%%%%%%%%%%%
\section{Main theorem}\label{sec:6}

%%%%%%%%%%%%%%%%%%%%%%%%%%%%%%%%%%%%%%%%%%%%%%%%%%%%%%%%%%%%%%%%%%%%%%%%%%%%%%%%
\subsection{Liftings to inner forms of $Sp_{2n}$}\label{ssec:61}

We write $\vrh$ for the right regular action of $G_n(\AA_\bff)$ on the space of fuctions on $G_n(\AA_\bff)$. 

\begin{theorem}\label{thm:61}
Notations and assumptions being as in Theorem \ref{thm:12}, the assignment
\[f\mapsto\sum_{B\in S^+_n}|\nu(B)|^{(\kap+n)/2}c_{\eta\nu(B)}\bfe_\infty(\tau(BZ))w^{\mu^{}_\bff\hat\chi^{(-1)^n\eta}_\bff}_B(\vrh(\Del)f) \]
defines an embedding $A_n\bigl(\mu_\bff^{}\hat\chi^{(-1)^n\eta}_\bff\bigl)\hookrightarrow\frkC^n_{\kap+n}$ for every $n$ and $\eta\in F^\times_+$, where $\Del\in G_n(\AAf)$ and $Z\in\frkH^d_n$. 
\end{theorem}

\begin{remark}\label{rem:61} 
\begin{enumerate}
\renewcommand\labelenumi{(\theenumi)}
\item\label{rem:611} In light of Lemma \ref{lem:43}(\ref{lem:432}) the embedding in Theorem \ref{thm:61} naturally defines an embedding $A_n(\mu_\bff^{}\hat\chi^{(-1)^n\eta}_\bff)\hookrightarrow\frkG^n_{\kap+n}$. 
\item\label{rem:612} The multiplicity of $A_1(\mu_\bff^{}\hat\chi^{-\eta}_\bff)$ in $\frkG^1_{\kap+1}$ is one by Corollary 7.7 of \cite{G}. 
However, we do not know if this result can imply the multiplicity of $A_1(\mu_\bff^{}\hat\chi^{-\eta}_\bff)$ in $\frkC^1_{\kap+1}$. 
If we have assumed that it is one, then we could have proved that the multiplicity of $A_n(\mu_\bff^{}\hat\chi^{(-1)^n\eta}_\bff)$ in $\frkC^n_{\kap+n}$ is one in the same way as in \S \ref{ssec:85}. 
\end{enumerate} 
\end{remark}

%%%%%%%%%%%%%%%%%%%%%%%%%%%%%%%%%%%%%%%%%%%%%%%%%%%%%%%%%%%%%%%%%%%%%%%%%%%%%%%%
\subsection{Compatibility with the Arthur conjecture}\label{ssec:62}

We explain how Theorem \ref{thm:61} can be viewed in the framework of Arthur's conjecture. 
%Let us review Arthur's conjectural multiplicity formula.  
For details of the conjecture, the reader should consult \cite{A}. 
The conjecture specialized to our current case is discussed in \cite{G} and Section 14 of \cite{I2}.  

Let $\call$ be the hypothetical Langlands group over $F$. 
Hypothetically, there is a bijective correspondence between the set of all equivalence classes of $m$-dimensional irreducible representations of $\call$ and the set of all irreducible cuspidal automorphic representations of $\GL_m(\AA)$. 
If $\pi$ is a cuspidal automorphic representation of $\PGL_2(\AA)$, then $\pi$ corresponds to a map $\rho(\pi):\call\to\SL_2(\CC)$. 
Let $\sym^{2n-1}$ be the irreducible $2n$-dimensional representation of $\SL_2(\CC)$. 
As is well-known, we may assume that $\sym^{2n-1}(\SL_2(\CC))\subset Sp_n(\CC)$. 
Thus $\rho(\pi)\boxtimes\sym^{2n-1}$ gives rise to a homomorphism $\call\times\SL_2(\CC)\to\SO_{4n}(\CC)$. 
Embedding $\SO_{4n}(\CC)$ into $\SO_{4n+1}(\CC)=\hat G_n$, we get a homomorphism $\call\times\SL_2(\CC)\to\hat G_n$. 
One postulates that for each place $v$ there is a natural conjugacy class of embeddings $\call_v\hookrightarrow\call$, where $\call_v$ is the Weil group of $F_v$ if $v\in\frkS_\infty$, and the Weil-Deligne group of $F_v$ if $v\notin\frkS_\infty$. 
We obtain a homomorphism $\rho(\pi_v)\boxtimes\sym^{2n-1}:\call_v\times\SL_2(\CC)\to\hat G_n$ for each $v$. 

The Arthur conjecture suggests that there exists a finite set $\vPi_n(\pi_v)=\vPi^+_n(\pi_v)\cup \vPi^-_n(\pi_v)$ of equivalence classes of unitary admissible representations of $G_{n,v}$ associated to $\rho(\pi_v)\boxtimes\sym^{2n-1}$. 
Moreover, it is required that if $G_{n,v}\simeq Sp_{2n}(F_v)$, then $\vPi^+_n(\pi_v)$ contains the Langlands quotient $I^+_n(\pi_v)$ of 
\[\Ind^{Sp_{2n}(F_v)}_{P_{2,2,\dots,2}(F_v)}(\pi_v\otimes\alp^{n-1/2})\boxtimes(\pi_v\otimes\alp^{n-3/2})\boxtimes\cdots\boxtimes(\pi_v\otimes\alp^{1/2}), \]
where $P_{2,2,\dots,2}$ is the standard parabolic subgroup of $Sp_{2n}$ with Levi subgroup $\GL_2\times\cdots\times\GL_2$. 
Choose $\vPi_v\in\vPi^{\eps_v}_n(\pi_v)$ for each $v$. 
Then $\otimes'_v\vPi_v$ is an automorphic representation of $G_n(\AA)$ generated by square-integrable automorphic forms if and only if $\prod_v\eps_v=\vep(1/2,\pi)$. 

Let $v$ be a finite place such that $D_v\simeq\Mat_2(F_v)$. 
Let $\mu\in\Ome(F_v^\times)$. 
If $\sig(\mu)>-\frac{1}{2}$, then we obtain an intertwining map 
\[I_n(\mu)\to\Ind^{Sp_{2n}(F_v)}_{P_{2,2,\dots,2}(F_v)}(I(\mu,\mu^{-1})\otimes\alp^{-(2n-1)/2})\boxtimes\cdots\boxtimes(I(\mu,\mu^{-1})\otimes\alp^{-1/2})\]
by applying Proposition 4.1 and Lemma 5.1 of \cite{Y3} repeatedly. 
Therefore if $-\frac{1}{2}<\sig(\mu)<\frac{1}{2}$, then $I_n(\mu)\simeq I^+_n(I(\mu,\mu^{-1}))$. 
If $\mu^2=\alp$ and $\mu\neq\alp^{1/2}$, then $A_n(\mu)\simeq I^+_n(A(\mu^{},\mu^{-1}))$ by Proposition 3.11(2) of \cite{J}. 
%Let $\calt$ denote the unique common component of $I_1(\alp^{-1/2})$ and $\Ind^{Sp_4(F_v)}_{P_{1,2,1}(F_v)}1\boxtimes A(\alp^{1/2},\alp^{-1/2})|_{\SL_2(F_v)}$, where $P_{1,2,1}$ is the standard parabolic subgroup of $Sp_4$ with Levi subgroup $\GL_1\times\SL_2$. 
On the other hand, $A_n(\alp^{1/2})$ is the Langlands quotient of 
\[\Ind^{Sp_{2n}(F_v)}_{P_{2,2,\dots,2}(F_v)}A(\alp^n,\alp^{n-1})\boxtimes\cdots\boxtimes A(\alp^2,\alp)\boxtimes A_1(\alp^{1/2}) \]
by Proposition 3.10(2) of \cite{J}. 
We guess that $A_n(\alp^{1/2})\in A^-_n(A(\alp^{1/2},\alp^{-1/2}))$. 
We presume that the reasoning above is correct even when $D_v$ is division. 

Let $\pi\simeq(\otimes_{v\in\frkS_\infty}\pi_v)\otimes\pi_\bff$ be an irreducible cuspidal automorphic representation of $\PGL_2(\AA)$ on which we impose the following conditions:  
\begin{enumerate}
\item[(\roman{one})] $\pi_v$ is a discrete series with extremal weight $\pm 2\kap_v$ for all $v\in\frkS_\infty$;
\item[(\roman{two})] $\pi_\bff\simeq A(\mu^{}_\bff,\mu_\bff^{-1})$; 
\item[(\roman{thr})] $\mu_\bff(-1)(-1)^{\sum_{v\in\frkS_\infty}\kap_v}=1$. 
\end{enumerate} 
Let $v\in\frkS_\infty$ and assume that $\kap_v$ is sufficiently large so that $W^{(\kap_v+n)}_B$ generates a holomorphic discrete series representation of $Sp_{2n}(F_v)$. 
The holomorphic discrete series representation with lowest $K$-type $(\det)^{\kap_v+n}$ belongs to $A^{(-1)^n}_n(\pi_v)$. 
Put $l={\shp}\{v\notin\frkS_\infty\;|\;\pi_v\otimes\hat\chi^{(-1)^n\eta}_v=A(\alp^{1/2},\alp^{-1/2})\}$. 
Since 
\[\vep(1/2,\pi\otimes\hat\chi^{(-1)^n\eta})=(-1)^{l+\sum_{v\in\frkS_\infty}\kap_v}\mu^{}_\bff(-1)\hat\chi^{(-1)^n\eta}_\bff(-1)=(-1)^{l+nd} \]
for all $\eta\in F^\times_+$, the restriction (\roman{thr}) is compatible with the Arthur conjecture. 

%%%%%%%%%%%%%%%%%%%%%%%%%%%%%%%%%%%%%%%%%%%%%%%%%%%%%%%%%%%%%%%%%%%%%%%%%%%%%%%%
\section{Fourier-Jacobi modules}\label{sec:7}

%%%%%%%%%%%%%%%%%%%%%%%%%%%%%%%%%%%%%%%%%%%%%%%%%%%%%%%%%%%%%%%%%%%%%%%%%%%%%%%%
\subsection{Jacobi groups}\label{ssec:71}

Fix $0\leq i\leq n$. 
Put $n'=n-i$. 
For $z\in S_i$ and $x,y\in\Mat^i_{n'}(D)$ we use the notation
\begin{align*}
\bfv(x,y;z)&=\left(\begin{array}{c|c}
\begin{matrix} \ono_i & x \\ 0 & \ono_{n'} \end{matrix}
& \begin{matrix} z-yx^* & y \\ -y^* & 0 \end{matrix}\\\hline 
0
& \begin{matrix} \ono_i & 0 \\ -x^* & \ono_{n'} \end{matrix}
\end{array}\right), & 
\eta_i&=\left(\begin{array}{cc|cc}
  &            & \ono_i &            \\
  & \ono_{n'} &   &            \\ \hline 
\ono_i &            &   &            \\
  &            &   & \ono_{n'} 
\end{array}\right)\in G_n. 
\end{align*}
We define some subgroups of $G_n$ by 
\begin{align*}
X_i&=\{\bfv(x,0;0)\;|\; x\in\Mat^i_{n'}(D)\}, & 
Y_i&=\{\bfv(0,y;0)\;|\; y\in\Mat^i_{n'}(D)\}, \\    
Z_i&=\{\bfv(0,0;z)\;|\; z\in S_i\}, &
\caln_i&=\{\bfv(x,y;z)\;|\; x,y\in\Mat^i_{n'}(D),\; z\in S_i\}. 
%V'&=\{\bfv(x,0;z)\;|\; x\in\Mat^i_{n'}(D),\; z\in S_i\}. 
\end{align*}
We identity $X_i$ and $Y_i$ with the space $\Mat^i_{n'}(D)$. 
We view $G_i$ and $G_{n'}$ with subgroups of $G_n$ via the embeddings
\begin{align*}
g_1&\mapsto\left(\begin{array}{cc|cc} 
a_1 & & b_1 & \\ & \ono_{n'} & & \\\hline
c_1  & & d_1 & \\ & & & \ono_{n'} 
\end{array}\right), &
g_2&\mapsto\left(\begin{array}{cc|cc} 
\ono_i & & & \\ & a_2 & & b_2 \\\hline
  & & \ono_i & \\ & c_2 & & d_2 
\end{array}\right),
\end{align*}
where we write a typical element $g_1\in G_i$ in the form $\begin{pmatrix} a_1 & b_1 \\ c_1 & d_1\end{pmatrix}$ with matrices $a_1$, $b_1$, $c_1$, $d_1$ of size $i$ over $D$, and similarly for $g_2\in G_{n'}$. 
Put $\calj_i=G_{n'}\cdot\caln_i$. 
%We identify $Z$ (resp. $X$) with $S_m$ (resp. $\Mat_{mr}(D)$) frequently. 
We sometimes write 
\begin{align*}
%\bfv(x)&=\bfv(x,0;0), & \bfv(y)&=\bfv(0,y;0), \\
\bfm'(A)&=\begin{pmatrix} A & 0 \\ 0& (A^{-1})^*\end{pmatrix}, & 
\bfn'(z)&=\begin{pmatrix} \ono_{n'} & z \\ 0 & \ono_{n'}\end{pmatrix} 
\end{align*} 
for $A\in\GL_{n'}(D)$ and $z\in S_{n'}$. 
We will frequently specialize to the case $i=n-1$ in our application to the proof of main theorems. 

%%%%%%%%%%%%%%%%%%%%%%%%%%%%%%%%%%%%%%%%%%%%%%%%%%%%%%%%%%%%%%%%%%%%%%%%%%%%%%%%
\subsection{Weil representations of Jacobi groups}\label{ssec:72}

Let $F$ be a local field. 
Fix $S\in S^\nd_i$. 
We regard $S$ as a homomorphism $Z_i\to\GG_a$ by $z\mapsto\tau(Sz)$. 
Then $\caln_i/\Ker S$ is a Heisenberg group with center $Z_i/\Ker S$ and a natural symplectic structure $\caln_i/Z_i$. 
The Schr\"{o}dinger representation $\ome_S^\psi$ of $\caln_i$ with central character $\psi^S$ is realized on the Schwartz space $\cals(X_i)$ by 
\beq
(\ome_S^\psi(\bfv(x,y;z))\phi)(u)=\phi(u+x)\psi^S(z)\psi(2\tau(u^*Sy))\label{tag:71}
\eeq 
for $\phi\in\cals(X_i)$. 
By the Stone-von Neumann theorem, $\ome_S^\psi$ is a unique irreducible representation of $\caln_i$ on which $Z_i$ acts by $\psi^S$. 

This embedding and the conjugating action give a homomorphism $G_{n'}\hookrightarrow Sp(\caln_i/Z_i)$ and Kudla \cite{K2} gave an explicit local splitting $G_{n'}\hookrightarrow\Mp(\caln_i/Z_i)$, where $\Mp(\caln_i/Z_i)$ is the metaplectic extension of $Sp(\caln_i/Z_i)$. 
The representation $\ome^\psi_S$ of $\caln_i$ extends to the Weil representation of $\Mp(\caln_i/Z_i)\ltimes \caln_i$. 
The pullback to $P_{n'}$ of this representation is given by
\begin{align}
(\ome^\psi_S(\bfm'(A))\phi)(u)&=\hat\chi^S(\nu(A))|\nu(A)|^i\phi(uA), \notag\\
(\ome^\psi_S(\bfn'(z))\phi)(u)&=\psi^S(uzu^{*})\phi(u), \notag\\
(\ome^\psi_S(J_{n'})\phi)(u)&=\gam^\psi(S)(\calf_S\phi)(u) \label{tag:72}
\end{align} 
for $\phi\in\cals(X_i)$, $u\in X_i$, $A\in\GL_{n'}(D)$ and $z\in S_{n'}$, where $\gam^\psi(S)$ is a certain $8^\mathrm{th}$ root of unity and $\calf_S\phi$ is the Fourier transform defined by 
\[(\calf_S\phi)(u)=\int_{X_i}\phi(x)\psi(2\tau(x^*Su))\,\d x. \]

%%%%%%%%%%%%%%%%%%%%%%%%%%%%%%%%%%%%%%%%%%%%%%%%%%%%%%%%%%%%%%%%%%%%%%%%%%%%%%%%
\subsection{The nonarchimedean case}\label{ssec:73}

Let $F$ be a finite extension of $\QQ_p$. 

\begin{lemma}\label{lem:71}
Let $f\in I_n(\mu)$ and $\phi\in\cals(X_i)$. 
If $\sig(\mu)\gg 0$, then the integral 
\begin{multline*}
\bet^\psi_S(g';f\otimes\bar\phi)=\frac{L\left(n+\frac{1}{2},\mu\right)}{L\left(n'+\frac{1}{2},\mu\hat\chi^S\right)}\prod_{j=1}^iL(2n'+2j-1,\mu^2)\\
\times\int_{Y_i\bsl \caln_i}f(\eta_i vg')\overline{(\ome^\psi_S(vg')\phi)(0)}\, \d v
\end{multline*}
is absolutely convergent. 
Moreover, it is meaningful for all $\Re\mu>-n'-\frac{1}{2}$ by analytic continuation and gives an $\caln_i$-invariant and $G_{n'}$-intertwining map 
\[\bet^\psi_S:I_n(\mu)\otimes\overline{\ome^\psi_S}\to I_{n'}(\mu\hat\chi^S). \]
\end{lemma}

\begin{proof}
The integral over $Z_i$ can be viewed as a Jacquet integral of the restriction of $f$ to $G_i$, which belongs to $I_i(\mu\alp^{n'})$.  
Thus it is entire on the whole of the complex manifold $\Ome(F^\times)$. 
Since $(\ome^\psi_S(xg')\phi)(0)=(\ome^\psi_S(g')\phi)(x)$ for $x\in X_i$, the integral over $X_i$ is convergent for all $\mu$. 
When $D\simeq\Mat_2(F)$, Ikeda showed that $\bet^\psi_S(f\otimes\bar\phi)\in I_{n'}(\mu\hat\chi^S)$ in the proof of Theorem 3.2 of \cite{I}. 
The computation applies equally well to the quaternion case. 
\end{proof}

\begin{lemma}\label{lem:72}
If $n=i+n'$, $S\in S_i^\nd$, $\vXi\in S_{n'}^\nd$, $f\in I_n(\mu)$ and $\phi\in\ome^\psi_S$, then 
\[w^{\mu\hat\chi^S}_{\vXi}(\bet^\psi_S(f\otimes\bar\phi))=C_S|\nu(\vXi)|^{-i/2}\int_{X_i}\overline{\phi(x)}w^\mu_{S\oplus \vXi}(\vrh(x)f)\,\d x. \]
\end{lemma}

\begin{proof}
Since $\eta_i=J_n\cdot J_{n'}$ and $J_{n'}\bfv(0,y,z)J_{n'}=\bfv(y,0,z)$, we observe that 
\begin{align*}
&\int_{\caln_i}f(J_nvJ_{n'}g')\overline{(\ome^\psi_S(vJ_{n'}g')\phi)(0)}\, \d v\\
=&\int_{Z_i}\int_{Y_i}\int_{X_i}f(J_nxzyJ_{n'}g')\overline{(\ome^\psi_S(zyJ_{n'}g')\phi)(x)}\, \d x\d y\d z\\
=&\int_{S_i}\int_{\Mat^i_{n'}(D)}\int_{X_i}f(J_n\bfv(0,y;z)J_{n'}g')\overline{(\ome^\psi_S(J_{n'}\bfv(y,0;z)g')\phi)(x)}\, \d x\d y\d z\\
=&\int_{S_i}\int_{\Mat^i_{n'}(D)}f(\eta_i\bfv(y,0;z)J_{n'}g')\overline{\calf_S(\ome^\psi_S(J_{n'}\bfv(y,0;z)g')\phi)(0)}\, \d y\d z\\
=&\gam^\psi(S)^{-1}\int_{Z_i}\int_{X_i}f(\eta_i\bfv(y,0;z)g')\overline{(\ome^\psi_S(\bfv(y,0;z)g')\phi)(0)}\, \d y\d z  
\end{align*}
by (\ref{tag:71}) and (\ref{tag:72}).  
Since 
\begin{align*}
&|\nu(\vXi)|^{(2n'+1)/4}\frac{L\left(n'+\frac{1}{2},\mu\hat\chi^S\right)}{L\left(\frac{1}{2},\mu\hat\chi^S\hat\chi^{\vXi}\right)}\prod_{j=1}^{n'}L(2j-1,\mu^2)\\
&\times\frac{L\left(n+\frac{1}{2},\mu\right)}{L\left(n'+\frac{1}{2},\mu\hat\chi^S\right)}\prod_{j=1}^iL(2n'+2j-1,\mu^2)\\
=&|\nu(S)|^{-(2n+1)/4}|\nu(\vXi)|^{-i/2}|\nu(S\oplus \vXi)|^{(2n+1)/4}\frac{L\left(n+\frac{1}{2},\mu\right)}{L\left(\frac{1}{2},\mu\hat\chi^{S\oplus \vXi}\right)}\prod_{j=1}^nL(2j-1,\mu^2),
\end{align*}
it suffices to prove   
\begin{align*}
&\int_{S_{n'}}\int_{\caln_i}f(J_nv\bfn'(u))\overline{(\ome^\psi_S(v\bfn'(u))\phi)(0)}\psi^{\vXi}(-u)\, \d v\d u\\
=&\int_{X_i}\int_{S_n}f(J_n\bfn(z)x)\overline{\phi(x)}\psi^{S\oplus \vXi}(-z)\, \d z\d x 
\end{align*}
for $\sig(\mu)\gg 0$. 
The left hand side is equal to 
\begin{align*}
&\int_{S_{n'}}\int_{\caln_i}f(J_n\bfn'(u)v)\overline{(\ome^\psi_S(\bfn'(u)v)\phi)(0)}\psi^{\vXi}(-u)\, \d v\d u\\
=&\int_{S_{n'}}\int_{\caln_i}f(J_n\bfn'(u)v)\overline{(\ome^\psi_S(v)\phi)(0)}\psi^{\vXi}(-u)\, \d v\d u\\
%=&\int_{S_{n'}}\int_{Y_i}\int_{Z_i}\int_{X_i}f(J_n\bfn'(u)\bfv(x,y;z))\overline{(\ome^\psi_S(\bfv(x,y;z)\phi)(0)}\psi^{\vXi}(-u)\, \d x\d z\d y\d u\\
=&\int_{S_{n'}}\int_{Y_i}\int_{Z_i}\int_{X_i}f(J_n\bfn'(u)yzx)\overline{\phi(x)\psi^S(z)}\psi^{\vXi}(-u)\, \d x\d z\d y\d u\\
=&\int_{S_n}\int_{X_i}f(J_n\bfn(z)x)\overline{\phi(x)}\psi^{S\oplus \vXi}(-z)\, \d x\d z. 
\end{align*}
Since this integral is absolutely convergent for $\sig(\mu)\gg 0$, we can exchange the order of integration. 
\end{proof}

\begin{corollary}\label{cor:71}
\begin{enumerate}
\renewcommand\labelenumi{(\theenumi)}
\item\label{cor:711} If $p\neq 2$, $D\simeq\Mat_2(F)$, $\psi$ is of order $0$, $\mu$ is unramified, $S\in S_i^\nd\cap\GL_i(\calo)$, $\phi$ is the characteristic function of $\Mat^i_{n'}(\calo)$ and $f\in I_n(\mu)$ satisfies $f(k)=1$ for all $k\in \Gam_n[\frko]$, then $\bet^\psi_S(\ono_{2n'};f,\phi)=1$. 
\item\label{cor:712} If $-\frac{1}{2}<\sig(\mu)<\frac{1}{2}$, then $\bet^\psi_S(I_n(\mu)\otimes\overline{\ome^\psi_S})=I_{n'}(\mu\hat\chi^S)$. 
\item\label{cor:713} If $\mu^2=\alp$, then $\bet^\psi_S(A_n(\mu)\otimes\overline{\ome^\psi_S})=A_{n'}(\mu\hat\chi^S)$. 
\end{enumerate}
\end{corollary}

\begin{proof}
Since $C_S=1$ in the unramified case, we can derive (\ref{cor:711}) from Lemmas \ref{lem:34} and \ref{lem:72}. 
If $\hat\chi^\vXi=\mu\hat\chi^S\alp^{-1/2}$, then $\hat\chi^{S\oplus \Xi}=\mu\alp^{-1/2}$, and hence 
\[w^{\mu\hat\chi^S}_{\vXi}(\bet^\psi_S(f\otimes\bar\phi))=0\]
for all $f\in A_n(\mu)$ and $\bar\phi\in\overline{\ome^\psi_S}$ by Lemmas \ref{lem:31}(\ref{lem:312}) and \ref{lem:72}. 
It follows from the proof of Proposition \ref{prop:31} that $\bet^\psi_S(A_n(\mu)\otimes\overline{\ome^\psi_S})\subset A_{n'}(\mu\hat\chi^S)$. 

Since the target spaces are irreducible, it is sufficient to show the nonvanishing of the intertwining maps.  
Lemma \ref{lem:31} enables us to take $\vXi\in S_{n'}^\nd$ and $f$ so that $w^\mu_{S\oplus \vXi}(f)\neq 0$. 
Using Lemma \ref{lem:72} and choosing $\phi$ to be supported in a small neighborhood, one can show that $w^{\mu\hat\chi^S}_{\vXi}(\bet^\psi_S(f\otimes\bar\phi))\neq 0$. 
\end{proof}

%%%%%%%%%%%%%%%%%%%%%%%%%%%%%%%%%%%%%%%%%%%%%%%%%%%%%%%%%%%%%%%%%%%%%%%%%%%%%%%%
\subsection{The metaplectic case}\label{ssec:74}

Fix $0\leq i\leq m$. 
Put $m'=m-i$. 
For $z\in\Sym_i$, $x,y\in\Mat^i_{m'}$, $a\in\GL_{m'}$ and $b\in\Sym_{m'}$ we use the notation
\begin{align*}
\frku(x,y;z)&=\left(\begin{array}{c|c}
\begin{matrix} \ono_i & x \\ 0 & \ono_{m'} \end{matrix}
& \begin{matrix} z-y\trs x & y \\ \trs y & 0 \end{matrix}\\\hline 
0
& \begin{matrix} \ono_i & 0 \\ -\trs x & \ono_{m'} \end{matrix}
\end{array}\right), & 
\eta_i&=\left(\begin{array}{cc|cc}
  &            & -\ono_i &            \\
  & \ono_{m'} &   &            \\ \hline 
\ono_i &            &   &            \\
  &            &   & \ono_{m'} 
\end{array}\right), \\
\frkm'(a)&=\begin{pmatrix} a & 0 \\ 0& \trs a^{-1}\end{pmatrix}, & 
\frkn'(b)&=\begin{pmatrix} \ono_{m'} & b \\ 0 & \ono_{m'}\end{pmatrix}.  
\end{align*}
We define the subgroups of $Sp_m$ by $\scrj_i=Sp_{m'}\cdot\scrn_i$ and 
\begin{align*}
\scrx_i&=\{\frku(x,0;0)\;|\; x\in\Mat^i_{m'}\}, & 
\scry_i&=\{\frku(0,y;0)\;|\; y\in\Mat^i_{m'}\}, \\    
\scrz_i&=\{\frku(0,0;z)\;|\; z\in\Sym_i\}, &
\scrn_i&=\{\frku(x,y;z)\;|\; x,y\in\Mat^i_{m'},\; z\in\Sym_i\}. 
%V'&=\{\bfv(x,0;z)\;|\; x\in\Mat^i_{m'},\; z\in\Sym_i\}. 
\end{align*} 

For the time being let $F$ be a local field. 
Fix $R\in\Sym^\nd_i$. 
The pullback to $\PP_{m'}\ltimes \scrn_i$ of the Weil representation $\ome^\psi_R$ of $\Mp(W_{m'})\ltimes \scrn_i$ is given by
\begin{align*}
(\ome_R^\psi(\frku(x,y;z))\phi)(u)&=\phi(u+x)\psi^R(z)\psi(2\tr(\trs uRy)),  \\
(\ome^\psi_R(\til\frkn'(b)\til\frkm'(a))\phi)(u)&=\psi^{R[u]}(b)\gam^\psi(\det a)^i\hat\chi^{\det R}(\det a)|\det a|^{i/2}\phi(ua) 
%\ome^\psi_S\left(\begin{pmatrix} 0 & -\ono_r \\ \ono_r & 0\end{pmatrix}\right)\phi(t)&=\int_{\XX}\phi(x)\psi(-2\tr(\trs tSx))\d x 
\end{align*} 
for $\phi\in\cals(\scrx_i)$, $u\in\scrx_i$, $a\in\GL_{m'}(F)$ and $b\in\Sym_{m'}(F)$. 

\begin{lemma}\label{lem:73}
Let $F$ be a finite extension of $\QQ_p$. 
Assume that $m'$ is odd. 
The following integral 
makes sense for all $\sig(\mu)>-\frac{m'}{2}-1$ and gives an $\scrn_i$-invariant and $\Mp(W_{m'})$-intertwining map $\bet^\psi_R:I^\psi_m(\mu)\otimes\overline{\ome^\psi_R}\to I_{m'}^\psi(\mu\hat\chi^{\det R})$: 
\begin{multline*}
\bet^\psi_R(\til g';h\otimes\bar\phi)=\prod_{j=1}^{[i/2]}L(m'+2j,\mu^2)\int_{\scry_i\bsl\scrn_i}h(\bfs(\eta_i v)\til g')\overline{(\ome^\psi_R(v\til g')\phi)(0)}\, \d v\\
\times\begin{cases} 
1 &\text{if $2\nmid m$, }\\
L\bigl(\frac{m+1}{2},\mu\hat\chi^{(-1)^{m/2}}\bigl) &\text{if $2|m$. }
\end{cases}
\end{multline*} 
Moreover, for $T\in\Sym_{m'}^\nd$, $h\in I_m^\psi(\mu)$ and $\phi\in\ome^\psi_R$ 
\[w^{\mu\hat\chi^{\det R}}_T(\bet^\psi_R(h\otimes\bar\phi))=C_R|\det T|^{-i/4}\int_{\scrx_i}\overline{\phi(x)}w^\mu_{R\oplus T}(\vrh(x)h)\,\d x. \]
 \end{lemma}

We can deduce the following corollary from Proposition \ref{prop:51} and Lemma \ref{lem:73} by the same reasoning as in the proof of Corollary \ref{cor:71}. 

\begin{corollary}\label{cor:72}
\begin{enumerate}
\renewcommand\labelenumi{(\theenumi)}
\item\label{cor:721} If $p\neq 2$, $\psi$ is of order $0$, $\mu$ is unramified, $R\in\Sym_i^\nd\cap\GL_i(\frko)$, $\phi$ is the characteristic function of $\Mat^i_{m'}(\frko)$ and $h\in I_m^\psi(\mu)$ satisfies $h(k)=1$ for all $k\in Sp_m(\frko)$, then $\bet^\psi_R(\ono_{2n'};f,\phi)=1$. 
\item\label{cor:722} If $-\frac{1}{2}<\sig(\mu)<\frac{1}{2}$, then $\bet^\psi_R(I^\psi_m(\mu)\otimes\overline{\ome^\psi_R})=I^\psi_{m'}(\mu\hat\chi^{\det R})$. 
\item\label{cor:723} If $\mu^2=\alp$, then $\bet^\psi_R(A_m^\psi(\mu)\otimes\overline{\ome^\psi_R})=A^\psi_{m'}(\mu\hat\chi^{\det R})$. 
\end{enumerate}
\end{corollary}

%%%%%%%%%%%%%%%%%%%%%%%%%%%%%%%%%%%%%%%%%%%%%%%%%%%%%%%%%%%%%%%%%%%%%%%%%%%%%%%%
\subsection{The archimedean case}\label{ssec:75}

When $m=i+m'$, $R\in\Sym_i^\nd$, $T\in\Sym_{m'}^\nd$ and $x\in\Mat^i_{m'}(F)$, we put 
\[R_{T,x}=\begin{pmatrix} R & 0 \\ 0 & T\end{pmatrix}\left[\begin{pmatrix} \ono_i & x \\ 0 & \ono_{m'}\end{pmatrix}\right]=\begin{pmatrix} R & Rx \\ \trs x R & T+R[x]\end{pmatrix}. \]
 
\begin{lemma}\label{lem:74}
Suppose that $F\simeq\RR$, $m=i+m'$ and $R\in\Sym_i^+$. 
Define $\vph_R\in\cals(\scrx_i)$ by $\vph_R(x)=e^{-2\pi\tr(R[x])}$ for $x\in\scrx_i$.  
Then there is a nonzero constant $C_R$ such that for  $T\in\Sym_{m'}^+$ and $\til g'\in\Mp(W_{m'})$
\[\int_{\scrx_i}W^{(\ell)}_{R\oplus T}(\bfs(x)\til g')\overline{(\ome^\psi_R(\til g')\vph_R)(x)}\,\d x
=C_R^{-1}(\det T)^{i/4}W^{(\ell-i/2)}_T(\til g'). \]
\end{lemma}

\begin{proof} 
Since the Gaussian is an eigenfunction for the action of $\til\calk_{m'}$ with eigencharacter $\til k'\mapsto J_{i/2}(\til k',\bfi)^{-1}$, one can readily verify that 
\[(\ome^\psi_R(\til g')\vph_R)(x)=J_{i/2}(\til g',\bfi_{m'})^{-1}\bfe(\tr(R[x]\til g'(\bfi_{m'})))\]
for $x\in\scrx_i$ and $\til g'\in\Mp(W_{m'})$. 
Put $Y=\Im\til g'(\bfi_{m'})$. 
Since 
\[W^{(\ell)}_{R\oplus T}(\bfs(x)\til g')=W^{(\ell)}_{R\oplus T}\left(\til\frkm\left(\begin{pmatrix}\ono_i & x \\ 0 & \ono_{m'}\end{pmatrix}\right)\til g'\right)=W^{(\ell)}_{R_{T,x}}(\til g') \]
by (\ref{tag:52}), the left hand side is equal to  
\begin{multline*}
\frac{\det(R\oplus T)^{\ell/2}}{J_\ell(\til g',\bfi_{m'})}\int_{\scrx_i}\bfe\left(\tr\left(R_{T,x}\begin{pmatrix} \bfi_i & \\ & \til g'(\bfi_{m'})\end{pmatrix}\right)\right)\overline{\left(\frac{\bfe(\tr(R[x]\til g'(\bfi_{m'})))}{J_{i/2}(\til g',\bfi_{m'})}\right)}\,\d x\\
=W^{(\ell-i/2)}_T(\til g')(\det R)^{\ell/2}(\det T)^{i/4}e^{-2\pi\tr R}(\det Y)^{i/2}\int_{\scrx_i}e^{-4\pi\tr(R[x]Y)}\,\d x. 
\end{multline*}
The factor $(\det Y)^{i/2}\int_{\scrx_i}e^{-2\pi\tr(R[x]Y)}\,\d x$ is a constant independent of $Y$. 
\end{proof}  

%%%%%%%%%%%%%%%%%%%%%%%%%%%%%%%%%%%%%%%%%%%%%%%%%%%%%%%%%%%%%%%%%%%%%%%%%%%%%%%%
\subsection{The global case}\label{ssec:76}

In the rest of this section $D$ is a totally indefinite quaternion algebra over a totally real number field $F$. 
Fix $S\in S^+_i$. 
Take a subgroup $\scra$ of $\caln_i(\AA)$ containing $Z_i(\AA)$ so that $\scra/\Ker S$ is a maximal abelian subgroup of $\caln_i(\AA)/\Ker S$ to which the character $\psi^S$ has an extension $\psi^S_\scra$. 
The Schr\"{o}dinger representation is equivalent to $\Ind^{\caln_i(\AA)}_{\scra}\psi^S_\scra$.  
Having chosen $\scra=Y_i(\AA)\oplus Z_i(\AA)$, we obtain the Schr\"{o}dinger model of the Weil representation $\ome^\psi_S\simeq\otimes'_v\ome_S^{\psi_v}$ realized on $\cals(X_i(\AA))$. 
If we choose $\scra=\caln_i(F)Z_i(\AA)$, then the space $\Ind^{\caln_i(\AA)}_{\scra}\psi^S_\scra=C^\infty_{\psi^S}(\caln_i(F)\bsl\caln_i(\AA))$ consists of smooth functions on $\caln_i(F)\bsl\caln_i(\AA)$ on which the center $Z_i(\AA)$ acts by $\psi^S$. 
The equivariant isomorphism $\scrs(X_i(\AA))\simeq C^\infty_{\psi^S}(\caln_i(F)\bsl\caln_i(\AA))$ is given by 
\[\Tht(\ome_S^\psi(v)\vph)=\sum_{x\in X_i(F)}(\ome_S^\psi(v)\vph)(x). \]

We denote by $\ome_S^{\psi_\bff}\simeq\otimes'_{v\notin\frkS_\infty}\ome_S^{\psi_v}$ the finite part of the global Weil representation $\ome^\psi_S$. 
For $\phi\in\cals(X_i(\AAf))$ we define a Schwartz function $\phi_S$ on $X_i(\AA)$ by 
\begin{align*}
\phi_S(x)&=\phi(x_\bff)\prod_{v\in\frkS_\infty}\vph_S(x_v), & 
x&=(x_v)\in X_i(\AA). 
\end{align*} 
%We define the theta function on $\calj_i(\AA)$ by \[\Tht(\ome_S^\psi(vg')\phi^{}_S)=\sum_{l\in X_i(F)}(\ome_S^\psi(vg')\phi^{}_S)(l)\] for $v\in \caln_i(\AA)$ and $g'\in G_{n'}(\AA)$. 
We here write $\rho$ for the right regular action of $G_n(\AAf)$ on $\frkT^n_{\ell+n}$. 
For $\calf\in \frkT^n_{\ell+n}$ we define the $(S,\phi)^\mathrm{th}$ Fourier-Jacobi coefficient of $\calf$ by 
\[\calf^S_\phi(g')=\int_{\caln_i(F)\bsl\caln_i(\AA)}\calf(vg')\overline{\Tht(\ome^\psi_S(vg')\phi^{}_S)}\,\d v. \] 

\begin{lemma}\label{lem:75}
Let $S\in S^+_i$. 
Put $n'=n-i$. 
%If $\calf\in C(X_i(F)N_n(F)\bsl G_n(\AA))$ has the Fourier expansion of the form 
Let
\[\calf(g)=\sum_{B\in S_n^+}\bfw_B(g_\bff,\calf)W^{(\ell+n)}_B(g_\infty), \]
be the Fourier expansion of $\calf\in\frkT^n_{n+\ell}$. 
Then $\calf^S_\phi(g')$ is equal to 
\[\sum_{\vXi\in S^+_{n'}}N_{F/\QQ}(\nu(\vXi))^{i/2}W^{(\ell+n')}_{\vXi}(g'_\infty)\int_{X_i(\AAf)}\bfw_{S\oplus \vXi}(xg'_\bff,\calf)\overline{(\ome_S^{\psi_\bff}(g'_\bff)\phi)(x)}\,\d x \] 
up to a nonzero constant multiple.   
%Suppose that $\calf\in C(P_n(F)\bsl G_n(\AA))$. 
Moreover, $\calf\in\frkC^n_{\ell+n}$ if and only if $(\rho(\Del)\calf)^S_\phi\in \frkC^{n'}_{\ell+n'}$ for all $\Del\in G_n(\AAf)$, $S\in S^+_i$ and $\phi\in\cals(X_i(\AAf))$. 
\end{lemma}

\begin{proof}
%Actually, an analogous result was proved in \cite{Ic} for more general automorphic forms. 
We abuse notation in writing $w^{(\ell+n)}_B(g,\calf)=\bfw^{}_B(g_\bff,\calf)W^{(\ell+n)}_B(g_\infty)$. 
The calculation in the proof of \cite[Lemma 4.1]{Ic} shows that 
\[\calf^S_\phi(g')=\sum_{\vXi\in S^+_{n'}}\int_{X_i(\AA)}w^{(\ell+n)}_{S\oplus \vXi}(xg',\calf)\overline{(\ome_S^\psi(g')\phi^{}_S)(x)}\,\d x \]
for $g'\in G_{n'}(\AA)$. 
Employing Lemma \ref{lem:74}, we arrive at the stated formula. 

Recall that we regard $G_{n'}$ as a subgroup of $G_n$ as in \S \ref{ssec:71}. 
Note that  
\begin{align*}
\calf(g)&=\sum_{S\in S_i^+}\calf^S(g), &
\calf^S(g)&=\sum_{\vXi\in S^+_{n'}}\sum_{x\in X_i(F)}w^{(\ell+n)}_{S\oplus\vXi}(xg,\calf). 
\end{align*}
Since the subgroups $P_n(F)$ and $G_{n'}(F)$ generate $G_n(F)$, if $\calf^S$ is left invariant under $G_{n'}(F)$ for all $S\in S_i^+$, then $\calf\in \frkC^n_{\ell+n}$. 
For all $\gam\in G_{n'}(F)$, 
\[\calf^S_\phi(\gam g')-\calf^S_\phi(g')=\int_{\caln_i(F)\bsl\caln_i(F_\AA)}(\calf^S(\gam vg')-\calf^S(vg'))\overline{\Tht(\ome^\psi_S(vg')\phi_S)}\,\d v. \]
The function $v\mapsto\calf^S(\gam vg)$ belongs to the subspace of $C^\infty_{\psi^S}(\caln_i(F)\bsl\caln_i(\AA))$ spanned by $\Tht(\ome^\psi_S(v)\phi_S)$. 
Thus $\calf^S$ is left invariant under $G_{n'}(F)$ if and only if $(\rho(\Del)\calf)^S_\phi$ is left invariant under $G_{n'}(F)$ for all $\Del\in G_n(\AAf)$ and $\phi\in\cals(X_i(\AAf))$. 
%Fix $\gam\in G_{n'}(F)$ and $\vXi\in S_{n'}^+$. Recall that we regard $G_{n'}$ as a subgroup of $G_n$ as in \S \ref{ssec:71}. Since the subgroups $P_n(F)$ and $G_{n'}(F)$ generate $G_n(F)$, if $\calf$ is left invariant under $G_{n'}(F)$, then $\calf\in \frkC^n_{\ell+n}$. The formula implies that $w^{(\ell+n)}_{S\oplus\vXi}(\gam g,\calf)=w^{(\ell+n)}_{S\oplus\vXi}(g,\calf)$ for all $g\in G_n(\AA)$ if and only if \[w^{(\ell+n')}_\Xi(\gam g'_\infty, (\rho(\Del)\calf)^S_\phi)=w^{(\ell+n')}_\Xi(g'_\infty, (\rho(\Del)\calf)^S_\phi)\] for all $g'_\infty\in G_{n'}(\AA_\infty)$, $\Del\in G_n(\AAf)$ and $\phi\in\cals(X_i(\AAf))$. 
\end{proof}

\begin{lemma}\label{lem:76}
%Fix $0<i<n$ and put $n'=n-i$. 
Let $\{c_B\}_{B\in S^{\pi_\bff}_n}\in T^n_{\ell+n}(\mu_\bff)$. 
Then $\{c_B\}\in C^n_{\ell+n}(\mu_\bff)$ if and only if $\{c_{S\oplus \vXi}\}_{\vXi\in D^{\pi^{}_\bff\otimes\hat\chi^S_\bff}_-}\in C^1_{\ell+1}(\mu^{}_\bff\hat\chi^S_\bff)$ for all $S\in S^+_{n-1}$.  
\end{lemma}

\begin{proof}
Taking Corollary \ref{cor:71} into account, we define a surjection 
\[\bet^{\psi_\bff}_S=\otimes_{v\notin\frkS_\infty}\bet^{\psi_v}_S: A_n(\mu^{}_\bff)\otimes\overline{\ome_S^{\psi_\bff}}\twoheadrightarrow A_{n'}(\mu^{}_\bff\hat\chi^S_\bff). \]  

Lemma \ref{lem:43} says that $\calf_{\ell+n}(f,\{c_B\})\in\frkT^n_{\ell+n}$ for all $f\in A_n(\mu_\bff)$. 
In view of Lemmas \ref{lem:72} and \ref{lem:75} the $(S,\phi)^\mathrm{th}$ Fourier-Jacobi coefficient of $\calf_{\ell+n}(f,\{c_B\})$ equals $\calf_{\ell+1}(\bet^{\psi_\bff}_S(f\otimes\bar\phi),\{c_{S\oplus \vXi}\})$ up to a nonzero constant multiple. 
Lemma \ref{lem:75} finally proves the equivalence. 
\end{proof}

%%%%%%%%%%%%%%%%%%%%%%%%%%%%%%%%%%%%%%%%%%%%%%%%%%%%%%%%%%%%%%%%%%%%%%%%%%%%%%%%
\section{Proofs of Theorems \ref{thm:11} and \ref{thm:12}}\label{sec:8}

%%%%%%%%%%%%%%%%%%%%%%%%%%%%%%%%%%%%%%%%%%%%%%%%%%%%%%%%%%%%%%%%%%%%%%%%%%%%%%%%
\subsection{Theta lifts from $\Mp(W_1)$}\label{ssec:81}

%We retain the notation of \S \ref{ssec:24}. 
We give a brief discussion of the results of theta correspondence for the dual pair $\Mp(W_1)\times\SO(V)$. 
For a detailed treatment one can consult \cite{G,MVW,Y3}. 
Let $(V,q_V)$ be a quadratic space of dimension $l$. 
In the case of interest in this paper $V=V_D$ or $V=Fe\oplus V_D\oplus Fe'$. 
In the former case $G^+(V)\simeq D^\times$ and in the latter case $G^+(V)\simeq\calg_1$ by Lemma \ref{lem:21}. 
We define the symplectic vector space $(\WW,\ll\;,\;\gg)$ of dimension $2l$ by $\WW=V\otimes W_1$ and $\ll\;,\;\gg=(\;,\;)\otimes\La\;,\;\Ra$. 
We have natural homomorphisms 
\begin{align*}
Sp(W_1)&\hookrightarrow Sp(\WW), &
G^+(V)&\stackrel{\vth}{\twoheadrightarrow}\mathrm{SO}(V)\hookrightarrow \mathrm{O}(V)\hookrightarrow Sp(\WW). 
\end{align*} 
The groups $\mathrm{O}(V)$ and $Sp(W_1)=\SL_2$ form a dual pair inside $Sp(\WW)$. 
 
Fix $\eta\in F^\times$. 
We obtain the representation $\ome^{\psi^\eta}_V=\otimes'_v\ome^{\psi^\eta_v}_{V_v}$ by pulling back to $G^+(V,\AA)\times \Mp(W_1)_\AA$ the global Weil representation of the metaplectic double cover $\Mp(\WW)_{\AA}$ of $Sp(\WW,\AA)$ associated to $\psi^\eta$. 
Note that $\ome^{\psi^\eta}_V\simeq \ome^\psi_{\eta V}$, where $\eta V$ is the space $V$ equipped with the quadratic form $\eta q_V$. 
When $l=1$, the local Weil representation $\ome^{\psi_v}_\eta\simeq\ome^{\psi^\eta_v}_1$ is realized in $\cals(F_v)$ and is the direct sum of two irreducible representations: $\ome^{\psi_v}_\eta=\ome^{\psi^\eta_v}_+\oplus\ome^{\psi^\eta_v}_-$, where $\ome^{\psi_v^\eta}_+$ (resp. $\ome^{\psi^\eta_v}_-$) consists of the even (resp. odd) functions in $\cals(F_v)$.   
Given an irreducible admissible genuine representation $\sig_v$ of $\Mp(W_1)_v$ the maximal quotient of $\ome^{\psi_v}_{V_v}$ on which $\Mp(W_1)_v$ acts as a multiple of $\sig_v$ is of the form $\sig_v\boxtimes\Tht^{\psi_v}_{V_v}(\sig_v)$, where $\Tht^{\psi_v}_{V_v}(\sig_v)$ is a representation of $G^+(V_v)$. 
We say that $\Tht^{\psi_v}_{V_v}(\sig_v)$ is zero if $\sig_v$ does not occur as a quotient of $\ome^{\psi_v}_{V_v}$. 
Let $\tht^{\psi_v}_{V_v}(\sig_v)$ be the maximal semisimple quotient of $\Tht^{\psi_v}_{V_v}(\sig_v)$. 
Then $\tht^{\psi_v}_{V_v}(\sig_v)$ is either zero or irreducible by the Howe conjecture. 

It turns out that there is a natural $Sp_{2l}(F)$-invariant map $\Tht:\ome^\psi_V\to\CC$. 
Let $\sig$ be an irreducible genuine cuspidal automorphic representation of $\Mp(W_1)_{\AA}$. 
For $h\in\sig$ and $\phi\in\ome^\psi_V$ we set 
\[\tht^\psi_V(g;h,\phi)=\int_{\SL_2(F)\bsl\SL_2(\AA)}h(\til g)\Tht(\ome^\psi_V(\til g,g)\phi)\, \d\til g.\]
Then $\tht^\psi_V(h,\phi)$ is an automorphic form on $G^+(V)$. 
We write $\tht^\psi_V(\sig)$ for the subspace of the space of automorphic forms on $G^{+}(V)$ spanned by $\tht^\psi_V(h,\phi)$ for all $h\in\sig$ and $\phi\in\ome^\psi_V$. 

The first part of the following proposition is a simple consequence of the Howe conjecture. 
The second part can be deduced from the Rallis inner product formula.  

\begin{proposition}[Proposition 2.5 and Theorem 2.8 of \cite{G}]\label{prop:81}
Let $\sig$ be an irreducible genuine cuspidal automorphic representation in $\scra_{00}$. 
\begin{enumerate}
\renewcommand\labelenumi{(\theenumi)}
\item\label{prop:811} If $\tht^\psi_V(\sig)$ is nonzero and contained in the space of square-integrable automorphic forms on $G^{+}(V)$, then $\tht^\psi_V(\sig)\simeq\otimes_v'\tht^{\psi_v}_{V_v}(\sig_v^\vee)$.  
\item\label{prop:812} Assume that $l\geq 5$. 
Then $\tht^\psi_V(\sig)$ is nonzero if and only if $\Tht^{\psi_v}_{V_v}(\sig_v)$ is nonzero for all $v$. 
\end{enumerate}
\end{proposition}

%%%%%%%%%%%%%%%%%%%%%%%%%%%%%%%%%%%%%%%%%%%%%%%%%%%%%%%%%%%%%%%%%%%%%%%%%%%%%%%%
\subsection{The work of Waldspurger}\label{ssec:82}
 
Let $\scra_{00}$ denote the space of genuine cusp forms on $\Mp(W_1)_{\AA}$ orthogonal to elementary theta series of the Weil representation $\ome^\psi_\eta$ for any $\eta\in F^\times$. 
This space $\scra_{00}$ satisfies multiplicity one but does not satisfy strong multiplicity one: 
there are nonequivalent cuspidal automorphic representations $\sig$ and $\sig'$ whose local components are equivalent for almost all places. 
We say that such $\sig$ and $\sig'$ are nearly equivalent. 

Waldspurger has described the near equivalence classes of representations in $\scra_{00}$. 
%Let us recall his results. 
In his papers \cite{W,W2} he defined a surjective map $\Wd_{\psi_v}$ from the set of irreducible admissible genuine unitary representations of $\Mp(W_1)_v$ which are not equivalent to $\ome^{\psi_v^\eta}_+$ for any $\eta\in F^\times_v$ to the set of irreducible infinite dimensional unitary representations of $\PGL_2(F_v)$. 
If $\sig_v$ is such a representation of $\Mp(W_1)_v$, then precisely one of the representations $\tht^{\psi_v}_{V_v^+}(\sig_v)$ and $\tht^{\psi_v}_{V_v^-}(\sig_v)$ is nonzero, where $V_v^+$ (resp. $V_v^-$) stands for a three dimensional split (resp. anisotropic) quadratic space over $F_v$ of discriminant $1$. 
We set $\Wd_{\psi_v}(\sig_v)=\tht^{\psi_v}_{V_v^+}(\sig_v)$, provided that it is nonzero. 
Otherwise the representation $\Wd_{\psi_v}(\sig_v)$ corresponds to $\tht^{\psi_v}_{V_v^-}(\sig_v)$ via the Jacquet-Langlands correspondence. 

Given an irreducible infinite dimensional unitary representation $\pi_v$ of $\PGL_{2}(F_v)$, we put $\vPi^{\psi_v}(\pi_v)=\Wd_{\psi_v}^{-1}(\pi_v)$. 
If $\pi_v$ is a discrete series, then ${\shp} \vPi^{\psi_v}(\pi_v)=2$. 
Otherwise $\vPi^{\psi_v}(\pi_v)$ is a singleton. 
In the first case the set $\vPi^{\psi_v}(\pi_v)$ has a distinguished element $\sig^{\psi_v}_+(\pi_v)$, which is characterized by the fact that $\sig_+^{\psi_v}(\pi_{v})\otimes\pi_{v}$ is a quotient of the Weil representation $\ome^{\psi_v}_{V_v^+}$. 
The other element of $\vPi^{\psi_{v}}_{\pi_v}$ will be denoted by $\sig^{\psi_v}_-(\pi_v)$: 
it is characterized by the fact that $\sig_-^{\psi_v}(\pi_v)\otimes\pi_v$ is a quotient of $\ome^{\psi_v}_{V_v^-}$. 
In the second case we shall let $\sig_+^{\psi_v}(\pi_v)$ be the unique element in $\vPi^{\psi_{v}}(\pi_v)$, and set $\sig_-^{\psi_v}(\pi_v)=0$. 
%We therefore have a local packet $A^{\psi_{v}}_{\pi_v}=\{\sig_{\pi_v}^{+},\sig_{\pi_{v}}^{-}\}$, where $\sig_{\pi_{v}}^{-}=0$ if $\pi_{v}$ is not a discrete series. 
This partition of representations of $\Mp(W_1)_v$ into packets and their parametrization in terms of representations of $\PGL_2(F_v)$ depend on the choice of $\psi_v$. 
But it is quite explicit. 

\begin{proposition}[Propositions 4, 5, 9 of \cite{W2}]\label{prop:82}
\begin{enumerate}
\renewcommand\labelenumi{(\theenumi)}
\item\label{prop:821} If $\pi_v$ is an irreducible principal series $I(\mu_v,\mu_v^{-1})$, then $\sig_+^{\psi_v}(\pi_v)\simeq I^{\psi_v}_1(\mu_v)$. 
\item\label{prop:822} If $\pi_v\simeq A(\alp^{1/2},\alp^{-1/2})$, then $\sig_-^{\psi_v}(\pi_v)\simeq A^{\psi_v}_1(\alp^{1/2})$. 
\item\label{prop:823} If $\mu_v^2=\alp$, $\mu_v\neq\alp^{1/2}$ and $\pi_v\simeq A(\mu_v^{},\mu_v^{-1})$, then $\sig_+^{\psi_v}(\pi_v)\simeq A^{\psi_v}_1(\mu_v)$. 
\item\label{prop:824} If $v\in\frkS_\infty$ and $\pi_v$ is a discrete series with extremal weight $\pm 2\kap_v$, then $\sig_+^{\psi_v}(\pi_v)$ is the holomorphic discrete series representation with lowest weight $\kap_v+\frac{1}{2}$. 
%and $\sig^{\psi_v}_-(\pi_v)$ is the anti-holomorphic discrete series representation with highest weight $\kap_v+\frac{1}{2}$. 
\end{enumerate}
\end{proposition}

Bear in mind the assumption that $\psi_v=\bfe|_{F_v}$ for $v\in\frkS_\infty$. 
Given an irreducible cuspidal automorphic representation $\pi=\otimes_v'\pi_v$ of $\PGL_2(\AA)$, we define a set of irreducible unitary representations of $\Mp(W_1)_\AA$ by 
\[\vPi^{\psi}(\pi)=\{\otimes'_v\sig^{\psi_v}_{\eps_v}(\pi_v)\;|\;\eps_{v}\in\{\pm\}, \text{ and for all most all }v,\; \eps_v=+\}. \]
For given $\sig=\otimes'_v\sig^{\psi_v}_{\eps_v}(\pi_v)\in\vPi^\psi(\pi)$, we set $\eps(\sig)=\prod_v\eps_v$. 
Corollaries 1 and 2 on p.~286 of \cite{W2} say that 
\beq
\scra_{00}\simeq\oplus_{\pi}\oplus_{\sig\in\vPi^{\psi}(\pi):\;\eps(\sig)=\vep(1/2,\pi)}\sig, \label{tag:81}
\eeq
where the sum ranges over all irreducible cuspidal automorphic representations $\pi$ of $\PGL_2(\AA)$ such that $L(1/2,\pi\otimes\hat\chi^t)\neq 0$ for some $t\in F^{\times}$. 

This theory includes the special case of Theorem \ref{thm:11} in which $m=1$. 

\begin{lemma}\label{lem:81}
$A^{\psi_\bff}_1(\mu_\bff)$ appears in $\frkC^{(1)}_{(2\kap+1)/2}$ if and only if $A(\mu^{}_\bff,\mu_\bff^{-1})$ appears in $\frkC_{2\kap}$ and $\mu_\bff(-1)(-1)^{\sum_{v\in\frkS_\infty}\kap_v}=1$.  
\end{lemma}

\begin{proof}
For $\ell\in\ZZ$ we denote  for the discrete series representation of $\PGL_2(\RR)$ with extremal weight $\pm 2\ell$ by $D_{2\ell}$ and the holomorphic discrete series representation of the real metaplectic group of rank $1$ with lowest weight $\ell+\frac{1}{2}$ by $\frkD_{\ell+1/2}$. 
Put 
\begin{align*}
\pi&=(\otimes_{v\in\frkS_\infty}D_{2\kap_v})\otimes A(\mu^{}_\bff,\mu_\bff^{-1}), & 
\sig&=(\otimes_{v\in\frkS_\infty}\frkD_{(2\kap_v+1)/2})\otimes A^{\psi_\bff}_1(\mu_\bff). 
\end{align*}
If $\sig$ is a cuspidal automorphic representation, then so is $\pi=\otimes_v'\Wd_{\psi_v}(\sig_v)$. 
Let $\frkS'_{\pi_\bff}$ be the set of nonarchimedean primes $v$ of $F$ such that $\mu_v=\alp^{1/2}$. 
Put $k=\shp\frkS'_{\pi_\bff}$. 
Then
\[\vep(1/2,\pi)=\mu_\bff(-1)(-1)^{k+\sum_{v\in\frkS_\infty}\kap_v}. \]
Proposition \ref{prop:82} says that 
\[\sig\simeq\{\otimes_{v\in\frkS_{\pi_\bff}'}\sig_-^{\psi_v}(\pi_v)\}\otimes\{\otimes'_{v\notin\frkS_{\pi_\bff}'}\sig_+^{\psi_v}(\pi_v)\}\in \vPi^\psi(\pi), \]
Thus $\eps(\sig)=\eps(1/2,\pi)$ if and only if $\mu_\bff(-1)(-1)^{\sum_{v\in\frkS_\infty}\kap_v}=1$. 
We have the desired conclusion by (\ref{tag:81}). 
\end{proof}

\begin{remark}\label{rem:81}
Lemma \ref{lem:81} is consistent with Lemma \ref{lem:51}(\ref{lem:513}). 
\end{remark}

%%%%%%%%%%%%%%%%%%%%%%%%%%%%%%%%%%%%%%%%%%%%%%%%%%%%%%%%%%%%%%%%%%%%%%%%%%%%%%%%
\subsection{Construction of an embedding $A^{\psi_\bff}_m(\mu_\bff)\hookrightarrow\frkC^{(m)}_{(2m+\kap)/2}$}\label{ssec:83}

We hereafter assume that $m>1$. 
Fix $R\in\Sym^+_{m-1}$. 
Corollary \ref{cor:72} gives a $\Mp(W_1)_\bff$-intertwining surjective homomorphism 
\beq
\bet^{\psi_\bff}_R=\otimes_{v\notin\frkS_\infty}\bet^{\psi_v}_R: A_m^{\psi_\bff}(\mu^{}_\bff)\otimes\overline{\ome_R^{\psi_\bff}}\twoheadrightarrow A_1^{\psi_\bff}(\mu^{}_\bff\hat\chi^{\det R}_\bff). \label{tag:82}
\eeq  
We associate to $\phi\in\cals(\scrx_{m-1}(\AAf))$ the function $\phi_R=\phi\otimes(\otimes_{v\in\frkS_\infty}\vph_R)\in\cals(\scrx_{m-1}(\AA))$ and the theta function on $\scrj_{m-1}(F)\bsl\scrj_{m-1}(\AA)$ defined by 
\begin{align*}
\Tht(\ome_R^\psi(v\til g')\phi^{}_R)&=\sum_{l\in\scrx_i(F)}(\ome_R^\psi(v\til g')\phi^{}_R)(l) &
(v&\in\scrn_{m-1}(\AA),\;\til g'\in \Mp(W_1)_{\AA}). 
\end{align*}
The $(R,\phi)^\mathrm{th}$ Fourier-Jacobi coefficient of $\calf\in\frkT^{(m)}_\ell$ is defined by 
\[\calf^R_\phi(\til g')=\int_{\scrn_i(F)\bsl\scrn_i(\AA)}\calf(\bfs(v)\til g')\overline{\Tht(\ome^\psi_R(v\til g')\phi^{}_R)}\,\d v. \]

Lemma \ref{lem:75} gives a nonzero constant $C_R$ such that  
\beq
\calf_{(2\kap+m)/2}(h,\{c_\xi\})^R_\phi=C_R\calf_{(2\kap+1)/2}(\bet^{\psi_\bff}_R(h\otimes\bar\phi),\{c_{R\oplus t}\}) \label{tag:83}
\eeq
for all $\{c_\xi\}\in T^{(m)}_{(2\kap+m)/2}(\mu_\bff)$ and $h\in A_m^{\psi_\bff}(\mu^{}_\bff)$. 

\begin{lemma}\label{lem:82}
If $\{c_t\}\in C^{(1)}_{(2\kap+1)/2}(\mu_\bff)$, then $\{c_{\eta\det\xi}\}\in C^{(m)}_{(2\kap+m)/2}(\mu^{}_\bff\hat\chi^\eta_\bff)$ for all $m$ and $\eta\in F^\times_+$. 
\end{lemma}

\begin{proof} 
The series $i^\eta_m(h)=\calf_{(2\kap+m)/2}(h,\{c_{\eta\det\xi}\})$ is convergent  for all $h\in A_m^{\psi_\bff}(\mu^{}_\bff\hat\chi^\eta_\bff)$ by Lemma \ref{lem:42} and the estimate of $\{c_t\}$ given in Proposition A.6.4 of \cite{Sh3}, and so by Lemma \ref{lem:51}(\ref{lem:512}) $i^\eta_m(A_m^{\psi_\bff}(\mu^{}_\bff\hat\chi^\eta_\bff))\subset\frkT^{(m)}_{(2\kap+m)/2}$. 
Lemma \ref{lem:51}(\ref{lem:515}), (\ref{tag:82}) and (\ref{tag:83}) show that 
\[i^\eta_m(h)^R_\phi=i^{\eta\det R}_1(\bet^{\psi_\bff}_R(h\otimes\bar\phi))\in\frkC^{(1)}_{(2\kap+1)/2}\] 
for all $R\in\Sym^+_{m-1}$ and $\phi\in\cals(\scrx_{m-1}(\AAf))$. 
Lemma \ref{lem:75} therefore concludes that $i^\eta_m(A_m^{\psi_\bff}(\mu^{}_\bff\hat\chi^\eta_\bff))\subset\frkC^{(m)}_{(2\kap+m)/2}$. 
\end{proof}

%%%%%%%%%%%%%%%%%%%%%%%%%%%%%%%%%%%%%%%%%%%%%%%%%%%%%%%%%%%%%%%%%%%%%%%%%%%%%%%%
\subsection{Some lemma on quadratic forms}\label{ssec:84}

Put 
\[\scrs^{\pi_\bff}_m=\{\xi\in\Sym^+_m\;|\;L(1/2,\pi^{}_\bff\otimes\hat\chi^{\det\xi}_\bff)\neq 0,\; \hat\chi^{\det\xi}_v\neq\mu_v\alp^{-1/2}\text{ for }v\in\frkS_{\pi_\bff}\}. \] 

%\begin{lemma}\label{lem:83}If $m\geq 3$ and $\xi_1, \xi_2, \xi_3\in\Sym_m^+$, then there exists a totally positive element of $F$ which is represented by  $\xi_1$, $\xi_2$ and $\xi_3$. \end{lemma}\begin{proof}Let $\frkS$ be the finite set of primes $v$ such that at least one of $\xi_1$, $\xi_2$ and $\xi_3$ is anisotropic over $F_v$. If $t\notin-(\det\xi_i)F_v^{\times2}$ for $v\in\frkS$ and $i=1,2,3$, then $t\in F^\times$ is represented by $\xi_1,\xi_2,\xi_3$. Since $[F_v^\times:F_v^{\times2}]\geq 4$ for finite primes $v$, the independence of the valuations ensures the existence of such $t$. \end{proof}
 
\begin{lemma}\label{lem:84}
If $C^{(1)}_{(2\kap+1)/2}(\mu_\bff)\neq\{0\}$ and $\Xi\in\scrs_3^{\pi_\bff}$ represents $t_1,t_2\in F^\times_+$, then there are a quadratic form $S\in\Sym_2^+$ representing $t_1$ and represented by $\Xi$ and a quadratic form $T\in\scrs_3^{\pi_\bff}$ representing both $S$ and $t_1\oplus t_2$.
\end{lemma}

\begin{proof}
Choose vectors $x,y\in F^3$ such that $\Xi[x]=t_1$ and $\Xi[y]=t_2$. 
If $\Xi(x,y)=\trs x\Xi y=0$, then we can take $S=t_1\oplus t_2$ and $T=\Xi$. 
Suppose that $\Xi(x,y)\neq 0$. 
Define functions $S:F^3\to\Sym_2(F)$ and $T:F^3\to\Sym_3(F)$ by  
\begin{align*}
S(z)&=\begin{pmatrix} \Xi(x,x) & \Xi(x,z) \\ \Xi(x,z) & \Xi(z,z)\end{pmatrix}, &
T(z)&=\begin{pmatrix} \Xi(x,x) & \Xi(x,z) & 0 \\ \Xi(x,z) & \Xi(z,z) & \Xi(y,z) \\ 0 & \Xi(y,z) & \Xi(y,y) \end{pmatrix}. 
\end{align*}
Clearly, $S(z)$ and $T(z)$ fulfill all the requirements besides the condition that $T(z)\in\scrs_3^{\pi_\bff}$. 
We define a quadratic form $Q$ of three variables by 
\[Q[z]=\det T(z)=\Xi(x,x)\Xi(y,y)\Xi(z,z)-\Xi(y,y)\Xi(x,z)^2-\Xi(x,x)\Xi(y,z)^2. \]
By a direct calculation $\det Q=-\det\Xi\cdot\Xi(x,x)^2\Xi(y,y)^2\Xi(x,y)^2\in -F^\times_+$. 
Let $\frkS_Q$ be the set of places of $F$ at which $Q$ is anisotropic. 
If $z\neq 0$ and $\Xi(x,z)=\Xi(y,z)=0$, then $Q[z]=\Xi(x,x)\Xi(y,y)\Xi(z,z)\in F^\times_+$. 
Thus $\frkS_Q$ consists of finite primes. 
Since $T(z)\simeq t_1\oplus t_2\oplus (t_1t_2)^{-1}Q[z]$, if $Q[z]\in F^\times_+$, then $T(z)$ is totally positive definite. 
Then $Q$ represents any element $t\in F^\times_+$ such that $t\notin(\det\Xi)F_v^{\times2}$ for all $v\in\frkS_Q$. 
Let $\pi$ be as in the proof of Lemma \ref{lem:81}. 
Take $\eta\in F^\times_{\pi_\bff}$ such that $\eta\notin(\det\Xi)F_v^{\times2}$ for all $v\in\frkS_Q$. 
Since $\vep(1/2,\pi\otimes\hat\chi^\eta)=1$, for any $\eps>0$, Theorem 4 of \cite{W2} gives $t\in F^\times_+$ which satisfies $|t-\eta|_v<\eps$ and such that $L(1/2,\pi\otimes\hat\chi^t)\neq 0$. 
\end{proof}

\begin{lemma}\label{lem:85}
Suppose that $C^{(1)}_{(2\kap+1)/2}(\mu_\bff)\neq\{0\}$. 
For $\xi_0,\xi_3\in\scrs^{\pi_\bff}_m$ there are $\xi_1,\xi_2\in\scrs^{\pi_\bff}_m$ and $R_1,R_2,R_3\in\Sym_{m-1}^+$ such that $R_i$ is represented by both $\xi_{i-1}$ and $\xi_i$ for all $i=1,2,3$. 
\end{lemma}

\begin{proof}
If $m\geq 3$, then $\xi_0\oplus(-\xi_3)$ must have a totally isotropic subspace of dimension $m-2$ and hence there are $\xi\in\Sym_{m-2}^+$ and $\xi_0',\xi_3'\in\scrs^{\pi^{}_\bff\otimes\hat\chi^\xi_\bff}_2$ such that $\xi_0\simeq \xi\oplus\xi_0'$ and $\xi_3\simeq \xi\oplus\xi_3'$. 
We may therefore assume that $m=2$. 

Set $\Xi=1\oplus\xi_0$ and $\Xi'=1\oplus\xi_3$. 
Choose $R_2\in F^\times_+$ represented by $\Xi$ and $\Xi'$. 
Applying Lemma \ref{lem:84} to $\Xi$, $1$ and $R_2$, we find a quadratic form $S\in\Sym_2^+$ representing $1$ and represented by $\Xi$ and find a quadratic form $T\in\scrs_3^{\pi_\bff}$ representing both $S$ and $1\oplus R_2$. 
Put $R_1=\det S$. 
Then $S\simeq 1\oplus R_1$. 
There is $\xi_1\in\scrs_2^{\pi_\bff}$ such that $T\simeq 1\oplus \xi_1$. 
Then $\xi_1$ represents $R_2$. 
Since both $\Xi=1\oplus\xi_0$ and $T\simeq 1\oplus \xi_1$ represent $S\simeq 1\oplus R_1$,  both $\xi_0$ and $\xi_1$ represent $R_1$. 
Similarly, we can find a quadratic form $\xi_2\in\scrs_2^{\pi_\bff}$ representing $R_2$ and find a totally positive element $R_3$ represented by both $\xi_2$ and $\xi_3$. 
\end{proof}

%%%%%%%%%%%%%%%%%%%%%%%%%%%%%%%%%%%%%%%%%%%%%%%%%%%%%%%%%%%%%%%%%%%%%%%%%%%%%%%%
\subsection{Multiplicity of $A^{\psi_\bff}_m(\mu_\bff)$}\label{ssec:85}

In light of Lemma \ref{lem:51}(\ref{lem:511}), giving a $\Mp(W_m)_\bff$-intertwining map from $A^{\psi_\bff}_m(\mu_\bff)$ into $\frkC^{(m)}_{(2\kap+m)/2}(\mu_\bff)$ is equivalent to giving complex numbers $\{c_\xi\}_{\xi\in\Sym^{\pi_\bff}_m}\in C^{(m)}_{(2\kap+m)/2}$. 
We now prove a stronger result. 

\begin{lemma}\label{lem:86}
Suppose that there is a $\Mp(W_m)_\bff$-intertwining embedding $i:A^{\psi_\bff}_m(\mu_\bff)\hookrightarrow\frkC^{(m)}_{(2\kap+m)/2}$. 
Then $\mu_\bff(-1)(-1)^{\sum_{v\in\frkS_\infty}\kap_v}=1$, $A(\mu^{}_\bff,\mu_\bff^{-1})$ occurs in $\frkC_{2\kap}$ and there is $\{c_t\}\in C^{(1)}_{(2\kap+1)/2}(\mu_\bff)$ such that $i(h)=\calf_{(2\kap+m)/2}(h,\{c_{\det\xi}\})$ for all $h\in A^{\psi_\bff}_m(\mu_\bff)$. 
In particular, $\dim C^{(m)}_{(2\kap+m)/2}(\mu_\bff)=1$. 
\end{lemma}

\begin{proof}
As mentioned above, Lemma \ref{lem:51}(\ref{lem:511}) gives $0\neq\{c_\xi\}\in C^{(m)}_{(2\kap+m)/2}(\mu_\bff)$ such that $i(h)=\calf_{(2\kap+m)/2}(h,\{c_\xi\})$. 
Hence $0\neq\{c_{R\oplus t}\}\in C^{(1)}_{(2\kap+1)/2}(\mu^{}_\bff\hat\chi^{\det R}_\bff)$ for all $R\in\Sym_{m-1}^+$ by (\ref{tag:82}) and (\ref{tag:83}). 
This together with Lemma \ref{lem:81} proves one implication of Theorem \ref{thm:11}. 

Fix a basis vector $\{e_t\}$ of the one dimensional vector space $C^{(1)}_{(2\kap+1)/2}(\mu_\bff)$. 
For each $R\in S_{m-1}^+$ Lemma \ref{lem:51}(\ref{lem:515}) gives a nonzero complex number $\del_R$ such that $c_{R\oplus t}=\del_Re_{t\det R}$ for all $t\in F^\times_{\pi^{}_\bff\otimes\hat\chi^{\det R}_\bff}$. 
Let $\xi\in\Sym^{\pi_\bff}_m$. 
Take $a\in\GL_m(F)$ such that $\xi=(R\oplus t)[a]$. 
Then 
\[c_\xi=\del_Re_{t\det R}\mu_\bff(\det a)^{-1}\prod_{v\in\frkS_\infty}\sgn_v(\det a)^{\kap_v}=\del_Re_{t\det R(\det a)^2}=\del_Re_{\det\xi}\]
by Lemma \ref{lem:51}(\ref{lem:512}). 
Lemma \ref{lem:51}(\ref{lem:516}) tells us that $c_\xi\neq 0$ if and only if $\xi\in\scrs^{\pi_\bff}_m$. 
Lemma \ref{lem:85} now says that 
\[\frac{c_{\xi_0}}{e_{\det\xi_0}}=\del_{R_1}=\frac{c_{\xi_1}}{e_{\det\xi_1}}=\del_{R_2}=\frac{c_{\xi_2}}{e_{\det\xi_2}}=\del_{R_3}=\frac{c_{\xi_3}}{e_{\det\xi_3}}\]
 for all $\xi_0,\xi_3\in\scrs^{\pi_\bff}_m$, which completes our proof.  
\end{proof}

%%%%%%%%%%%%%%%%%%%%%%%%%%%%%%%%%%%%%%%%%%%%%%%%%%%%%%%%%%%%%%%%%%%%%%%%%%%%%%%%
\section{Transfer to inner forms}\label{sec:9}

We retain the notation of \S \ref{ssec:24}. 
Thus $V=Fe\oplus V_D\oplus Fe'$ and $G^+(V)\simeq \calg_1$. 
In the first half of this section we switch to a local setting. 
Thus $F=F_v$ is a local field of characteristic zero. 

%%%%%%%%%%%%%%%%%%%%%%%%%%%%%%%%%%%%%%%%%%%%%%%%%%%%%%%%%%%%%%%%%%%%%%%%%%%%%%%%
\subsection{The Schr\"{o}dinger model vs. the mixed model}\label{ssec:91}

The Weil representation $\ome^\psi_V$ can be realized on $\cals(V)$ and has the following formulas:    
\begin{align}
&(\ome^\psi_V(\zet\til\frkm(t))\Phi)(v)=\zet\gam^{\psi}(t)|t|^{5/2}\Phi(tv), \label{tag:91}\\
&(\ome^\psi_V(\til\frkn(b))\Phi)(v)=\psi(bq_V(v))\Phi(v), \label{tag:92}\\
&(\ome^\psi_V(\bet)\Phi)(v)=\Phi(\vth(\bet)^{-1}v) \label{tag:93}
\end{align}
for $\zet\in\mu_2$, $t\in F^\times$, $b\in F$, $\bet\in\calg_1$ and $v\in V$. 
Since the map 
\[(\scrf\Phi)(x;u,u')=\int_F\Phi(re+x+ue')\psi(ru')\,\d r \]
is a $\CC$-linear isomorphism from $\cals(V)$ onto $\cals(V_D\oplus F^2)$, we can define the action $\Ome^\psi_V$ of $\Mp(W_1)\times\calg_1$ on $\cals(V_D\oplus F^2)$ so that 
\begin{align*}
\Ome^\psi_V(\til g,g)\scrf\Phi&=\scrf(\ome^\psi_V(\til g,g)\Phi), & 
(\til g,g)&\in\Mp(W_1)\times\calg_1.  
\end{align*}
The following formulas are derived easily or read of from Lemma 46 of \cite{W2}: 
\begin{align}
&(\Ome^\psi_V(\bfd(t)\bfm(A))\vph)(x;u,u')=|t\nu(A)|\vph(A^{-1}xA;t\nu(A)u,t\nu(A)u'), \label{tag:94}\\
&(\Ome^\psi_V(\bfn(z))\vph)(x;u,u')=\psi(-(\tau(zx)-u\nu(z))u')\vph(x-uz;u,u'), \label{tag:95}\\
&(\Ome^\psi_V(\til\frkm(t))\vph)(x;u,u')=\gam^\psi(t)|t|^{3/2}\vph(tx;tu,t^{-1}u'), \label{tag:96}\\
&(\Ome^\psi_V(\til\frkn(b))\vph)(x;u,u')=\psi(-b\nu(x))\vph(x;u,u'+ub), \label{tag:97}\\
&(\Ome^\psi_V(\bfs(J))\vph)(x;u,u')=\gam^\psi_D\int_{D_-}\vph(y;-u',u)\psi(\tau(xy))\,\d y \label{tag:98}
\end{align}
for $A\in D^\times$; $t\in F^\times$; $z,x\in D_-$ and $b,u,u'\in F$, where $\gam^\psi_D$ is a certain $8^\mathrm{th}$ root of the unitary and $\d y$ is the self-dual Haar measure on $D_-$ with respect to the Fourier transform defined by
\begin{align*}
(\calf_D\phi)(x)&=\int_{D_-}\phi(y)\psi(\tau(xy))\,\d y, & 
\phi&\in\cals(D_-). 
\end{align*}

\begin{remark}\label{rem:91}
Note that $n_2(z)=\bfn(-z)$ in the notation of \cite{W2}. 
\end{remark} 

%%%%%%%%%%%%%%%%%%%%%%%%%%%%%%%%%%%%%%%%%%%%%%%%%%%%%%%%%%%%%%%%%%%%%%%%%%%%%%%%
\subsection{Compatibility of Jacquet integrals}\label{ssec:92} 

We first discuss the nonarchimedean case. 

\begin{lemma}\label{lem:91}
Let $h\in I^\psi_1(\mu)$ and $\Phi\in\cals(V)$. 
If $\sig(\mu)>-\frac{3}{2}$, then the integral 
\[\Gam^\psi(g;h\otimes\Phi)=L(3/2,\mu\hat\chi^{-1})^{-1}\int_{\UU_1\bsl \SL_2(F)}h(\til g)(\ome^\psi_V(\til g,g)\Phi)(e)\,\d \til g \]
is absolutely convergent. 
It gives a $\Mp(W_1)$-invariant and $\calg_1$-intertwining map $\Gam^\psi:I^\psi_1(\mu)\otimes\ome^\psi_V\to J_1(\mu\hat\chi^{-1})$. 
If $F$ is not dyadic, $D\simeq\Mat_2(F)$, $\psi$ is of order $0$, $\mu$ is unramified, $\Phi$ is the characteristic function of $\frko e\oplus\calr_1\oplus \frko e'$ and $h(k)=1$ for all $k\in\SL_2(\frko)$, then $\Gam^\psi(\ono_2;h\otimes\Phi)=1$.  
\end{lemma}
 
\begin{proof}
The integral defining $\Gam^\psi(g;h\otimes\Phi)$ makes sense by (\ref{tag:92}) and equals   
\begin{align*}
&\int_{F^\times}\int_{\SL_2(\frko)}h(\til\frkm(a)k)(\ome^\psi_V(\til\frkm(a)k,g)\Phi)(e)|a|^{-2}\,\d a\d k\\
=&\int_{F^\times}\int_{\SL_2(\frko)}\hat\chi^{-1}(a)\mu(a)|a|^{3/2}h(k)(\ome^\psi_V(k,g)\Phi)(ae)\,\d k\,\d a 
\end{align*}
by (\ref{tag:91}) and (\ref{tag:51}). 
It is absolutely convergent for $\sig(\mu)>-\frac{3}{2}$. 
It follows from (\ref{tag:91}), (\ref{tag:94}) and Lemma \ref{lem:21} that for $A\in D^\times$, $t\in F^\times$ and $z\in D_-$
\begin{align*}
&L(3/2,\mu\hat\chi^{-1})\Gam^\psi(\bfd(t)\bfm(A)\bfn(z)g;h\otimes\Phi)\\
=&\int_{\UU_1\bsl \SL_2(F)}h(\til g)(\ome^\psi_V(\til g,g)\Phi)(t^{-1}\nu(A)^{-1}e)\,\d \til g\\
=&(\hat\chi^{-1}\mu\alp^{7/2})(t\nu(A))\int_{\UU_1\bsl \SL_2(F)}h(\til\frkm(t\nu(A))^{-1}\til g)(\ome^\psi_V(\til\frkm(t\nu(A))^{-1}\til g,g)\Phi)(e)\,\d \til g\\
=&(\hat\chi^{-1}\mu\alp^{3/2})(t\nu(A))L(3/2,\mu\hat\chi^{-1})\Gam^\psi(g;h\otimes\Phi). 
\end{align*} 
Therefore $\Gam^\psi(h\otimes\Phi)\in J_1(\mu\hat\chi^{-1})$. 
\end{proof}

\begin{lemma}\label{lem:92}
Let $h\in I^\psi_1(\mu)$, $\Phi\in\cals(V)$ and $\vXi\in D_-^\nd$. 
Put $\vph=\scrf\Phi\in\cals(V_D\oplus F^2)$. 
If $\sig(\mu)>-\frac{3}{2}$, then
\[\Gam^\psi(J_2;h\otimes\Phi)=L(3/2,\mu\hat\chi^{-1})^{-1}\int_{\SL_2(F)}h(\til g)(\Ome^\psi_V(\til g)\vph)(0;1,0)\,\d \til g. \] 
 \end{lemma}

\begin{proof} 
The product $L(3/2,\mu\hat\chi^{-1})\Gam^\psi(J_2;h\otimes\Phi)$ equals 
\[\int_{\UU_1\bsl \SL_2(F)}h(\til g)(\ome^\psi_V(\til g,J_2)\Phi)(e)\,\d\til g=\int_{\UU_1\bsl \SL_2(F)}h(\til g)(\ome^\psi_V(\til g)\Phi)(e')\,\d \til g \]
by (\ref{tag:93}) and Lemma \ref{lem:21}.  
The Fourier inversion says that  
\[(\scrf^{-1}\vph)(re+x+r'e')=\int_F\vph(x;r',u)\psi(-ru)\,\d u. \]
We use this formula and (\ref{tag:97}) to see that the left hand side equals 
\begin{align*}
&\int_{\UU_1\bsl \SL_2(F)}h(\til g)\scrf^{-1}(\Ome^\psi_V(\til g)\vph)(e')\,\d \til g\\
=&\int_{\UU_1\bsl \SL_2(F)}h(\til g)\int_F(\Ome^\psi_V(\til g)\vph)(0;1,u)\,\d u\d \til g\\
=&\int_{\UU_1\bsl \SL_2(F)}h(\til g)\int_F(\Ome^\psi_V(\til\frkn(u)\til g)\vph)(0;1,0)\,\d u\d \til g. 
\end{align*}
We combine the integrals over $\UU_1$ and $\UU_1\bsl \SL_2(F)$ into an integral over $\SL_2(F)$ to obtain the stated formula.  
The integral thus obtained equals
\[\int_{\SL_2(\frko)}\int_{F^\times}\int_F|a|\mu(a)h(k)|a|^{3/2}\hat\chi^{-1}(a)(\Ome^\psi_V(k)\vph)(0;a,ua^{-1})|a|^{-2}\,\d u\d a\d k \]
and converges absolutely for $\sig(\mu)>-\frac{3}{2}$, which justifies all the manipulations. 
\end{proof}

\begin{lemma}\label{lem:93}
Notation being as in Lemma \ref{lem:92}, there is a constant $C$ which is independent of $\vXi$ and such that $w^{\mu\hat\chi^{-1}}_\vXi(\Gam^\psi(h\otimes\Phi))$ is equal to 
\[C|\nu(\vXi)|^{1/4}\int_{\UU_1\bsl \SL_2(F)}w^\mu_{\nu(\vXi)}(\vrh(\til g)h)(\Ome^\psi_V(\til g)\vph)(-\vXi;0,1)\, \d\til g. \] 
\end{lemma}

\begin{proof}
By Lemma \ref{lem:92} we can write $w^{\mu\hat\chi^{-1}\alp^s}_\vXi(\Gam^\psi(h^{(s)}\otimes\Phi))$ as the product of 
\[|\nu(\vXi)|^{3/4}\frac{L(2s+1,\mu^2)}{L\left(s+\frac{1}{2},\mu\hat\chi^{-1}\hat\chi^\vXi\right)}\\
=|\nu(\vXi)|^{1/4}\cdot\frac{|\nu(\vXi)|^{1/2}L(2s+1,\mu^2)}{L\left(s+\frac{1}{2},\mu\hat\chi^{\nu(\vXi)}\right)}\]
and the integral
\begin{align*}
&\int_{D_-}\int_{\SL_2(F)}h^{(s)}(J\til g)(\Ome^\psi_V(J\til g,\bfn(z))\vph)(0;1,0)\overline{\psi^\vXi(z)}\,\d \til g\d z\\
=&\int_{D_-}\int_{\SL_2(F)}h^{(s)}(J\til g)(\Ome^\psi_V(J\til g)\vph)(-z;1,0)\psi(-\tau(\vXi z))\,\d \til g\d z 
\end{align*}
by (\ref{tag:95}). 
Since the double integral is absolutely convergent for $\Re s\gg 0$, we may interchange the order of integration. 
Using (\ref{tag:98}) and (\ref{tag:97}), we get 
\begin{align*}
%&\int_{\SL_2(F)}\int_{D_-}h^{(s)}(J\til g)(\Ome^\psi_V(J\til g)\vph)(-z;1,0)\psi(-\tau(\vXi z))\,\d z\d \til g\\
&\frac{1}{\gam^\psi_D}\int_{\SL_2(F)}h^{(s)}(J\til g)(\Ome^\psi_V(J^2\til g)\vph)(\vXi;0,-1)\,\d \til g\\
=&\frac{\gam^\psi(-1)}{\gam^\psi_D}\int_{\SL_2(F)}h^{(s)}(J\til g)(\Ome^\psi_V(\til g)\vph)(-\vXi;0,1)\,\d \til g\\
=&\frac{\gam^\psi(-1)}{\gam^\psi_D}\int_{\UU_1\bsl \SL_2(F)}\int_F h^{(s)}(J\til\frkn(b)\til g)\overline{\psi^{\nu(\vXi)}(b)}\,\d b\, (\Ome^\psi_V(\til g)\vph)(-\vXi;0,1)\,\d \til g. 
\end{align*}
The outer integral converges absolutely for all $s$. 
The proof is complete by evaluating the equality at $s=0$.   
\end{proof}

\begin{corollary}\label{cor:91}
\begin{enumerate}
\renewcommand\labelenumi{(\theenumi)}
\item\label{cor:911} If $-\frac{1}{2}<\sig(\mu)<\frac{1}{2}$, then $\Gam^\psi(I^\psi_1(\mu)\otimes\ome^\psi_V)=J_1(\mu\hat\chi^{-1})$. 
\item\label{cor:912} If $\mu^2=\alp$, then $\Gam^\psi(A^\psi_1(\mu)\otimes\ome^\psi_V)=A_1(\mu\hat\chi^{-1})$. 
\end{enumerate}
\end{corollary}

\begin{proof}
If $\hat\chi^\vXi=\mu\hat\chi^{-1}\alp^{-1/2}$, then $\hat\chi^{\nu(\vXi)}=\mu\alp^{-1/2}$, and hence 
\[w^{\mu\hat\chi^{-1}}_\vXi(\Gam^\psi(h\otimes\Phi))=0\]
for all $h\in A^\psi_1(\mu)$ and $\Phi\in\ome^\psi_V$ by Proposition \ref{prop:51}(\ref{prop:513}) and Lemma \ref{lem:93}. 
We can therefore infer from Proposition \ref{prop:31}(\ref{prop:313}) that $\Gam^\psi(A^\psi_1(\mu)\otimes\ome^\psi_V)\subset A_1(\mu\hat\chi^{-1})$.  
Employing Proposition \ref{prop:51} again, we can take $\vXi\in D_-^\nd$ and a test vector $h$ for which $w^\mu_{\nu(\vXi)}(h)\neq 0$. 
If $\vph=\vph'\otimes\vph''$ with $\vph'\in\cals(V_D)$ and $\vph''\in\cals(F^2)$, then we obtain 
\begin{multline*}
w^{\mu\hat\chi^{-1}}_\vXi(\Gam^\psi(h\otimes\Phi))=C|\nu(\vXi)|^{1/4}\int_F\int_{F\setminus\{0\}}\vph''(c,a)\\
\times w^\mu_{\nu(\vXi)}\left(\vrh\left(\begin{pmatrix} a^{-1} & 0 \\ c & a\end{pmatrix}\right)h\right)\left(\ome^\psi_{V_D}\left(\begin{pmatrix} a^{-1} & 0 \\ c & a\end{pmatrix}\right)\vph'\right)(-\vXi)\, \d a\d c  
\end{multline*}
by rewriting the formula in Lemma \ref{lem:93}, where $\d a$ and $\d c$ are Haar measures on $F$. 
% (cf. \Roman{one}.1 (2) of \cite{W3}). 
This integral can be made nonzero by choosing $\vph''$ to be supported in a small neighborhood of $(0,1)$. 
Thus $\Gam^\psi$ is nonzero, which verifies the claimed results as the target spaces are irreducible. 
\end{proof}

\begin{remark}\label{rem:92}
Corollary \ref{cor:91} implies that 
\begin{align*}
\tht^\psi_V(I^\psi_1(\mu)^\vee)&\simeq J_1(\mu\hat\chi^{-1}), & 
\tht^\psi_V(A^\psi_1(\mu)^\vee)&\simeq A_1(\mu\hat\chi^{-1}). 
\end{align*}
This result is stated in Propositions 5.2 and 6.3 of \cite{G}. 
\end{remark}

%%%%%%%%%%%%%%%%%%%%%%%%%%%%%%%%%%%%%%%%%%%%%%%%%%%%%%%%%%%%%%%%%%%%%%%%%%%%%%%%
%\subsection{Compatibility of Jacquet integrals: the archimedean case}\label{ssec:93} 

In Lemma 7.6 of \cite{Ic2} Ichino explicitly constructed a Schwartz function $\Lam\in\cals(V_D\oplus \RR^2)$ with the following property. 

\begin{lemma}\label{lem:94}
Suppose that $F=\RR$ and $D\simeq\Mat_2(\RR)$. 
There is $\Lam\in\cals(V_D\oplus \RR^2)$ such that for all $\vXi\in D^+_-$
\[W^{(\ell)}_\vXi(g)=2^\ell\nu(\vXi)^{1/4}\int_{\UU_1\bsl \SL_2(\RR)}W^{(\ell-1/2)}_{\nu(\vXi)}(\til g)(\Ome^\psi_V(\til g,g)\Lam)(-\vXi;0,1)\, \d\til g \]
and for all $\vXi\in -D^+_-$
\[\int_{\UU_1\bsl \SL_2(\RR)}W^{(\ell-1/2)}_{\nu(\vXi)}(\til g)(\Ome^\psi_V(\til g,g)\Lam)(-\vXi;0,1)\, \d\til g=0. \]
\end{lemma}

%%%%%%%%%%%%%%%%%%%%%%%%%%%%%%%%%%%%%%%%%%%%%%%%%%%%%%%%%%%%%%%%%%%%%%%%%%%%%%%%
\subsection{Fourier coefficients for Saito-Kurokawa liftings}\label{ssec:94}

Let $F$ be a totally real number field and $D$ a totally indefinite quaternion algebra over $F$. 
We denote by $\ome^\psi_V\simeq\otimes'_v\ome^{\psi_v}_{V_v}$ the Schr\"{o}dinger model of the global Weil representation and by $\Ome^\psi_V$ its mixed model. 
These models are related by the intertwining map $\scrf:\cals(V(\AA))\to\cals(V_D(\AA)\oplus\AA^2)$. 
We write $\ome^{\psi_\bff}_V$ and $\Ome^{\psi_\bff}_V$ for their finite parts. 
For $\Phi\in\cals(V(\AAf))$ we define a Schwartz function $\Phi_\Lam$ on $V(\AA)$ by 
\begin{align*}
\Phi_\Lam(x)&=\Phi(x_\bff)\prod_{v\in\frkS_\infty}(\scrf^{-1}_v\Lam)(x_v), &
x&=(x_v)\in V(\AA). 
\end{align*}

Taking Lemma \ref{lem:91} and Corollary \ref{cor:91} into account, we define a surjective homomorphism 
\beq
\Gam^{\psi_\bff}=\otimes_{v\notin\frkS_\infty}\Gam^{\psi_v}: A^{\psi_\bff}_1(\mu^{}_\bff)\otimes\ome^{\psi_\bff}_V\twoheadrightarrow A_1(\mu^{}_\bff\hat\chi^{-1}_\bff). \label{tag:99}
\eeq 

The following result is nothing but Lemma 47 of \cite{W2} (cf. Remark \ref{rem:91}). 

\begin{lemma}\label{lem:95}
If $\calf\in\scra_{00}$, $\vph\in\cals(V_D(\AA)\oplus\AA^2)$ and $0\neq \vXi\in D_-(F)$, then 
\[W_\vXi(\tht^\psi_V(\calf,\vph))
=\int_{\UU_1(\AA)\bsl\SL_2(\AA)}W_{\nu(\vXi)}(\til g,\calf)(\Ome^\psi_V(\til g)\vph)(-\vXi;0,1)\, \d\til g. \]
\end{lemma}

\begin{lemma}\label{lem:96}
If $0\neq\{c_t\}_{t\in F^\times_{\pi_\bff}}\in C^{(1)}_{(2\kap+1)/2}(\mu_\bff)$, then  
\[0\neq\{c_{\nu(\vXi)}\}_{\vXi\in D^{\pi^{}_\bff\otimes\hat\chi^{-1}_\bff}_-}\in C^1_{\kap+1}(\mu^{}_\bff\hat\chi^{-1}_\bff). \]  
\end{lemma}

\begin{proof}
We have $\calf_{(2\kap+1)/2}(h,\{c_t\})\in\frkC^{(1)}_{(2\kap+1)/2}$ for all $h\in A^{\psi_\bff}_1(\mu_\bff)$ by assumption. 
We consider its Saito-Kurokawa lift
\[\tht^\psi_V(g;\calf_{(2\kap+1)/2}(h,\{c_t\}),\Phi_\Lam)=\sum_{\vXi\in D_-(F)}W_\vXi(g,\tht^\psi_V(\calf_{(2\kap+1)/2}(h,\{c_t\}),\Phi_\Lam)).\]
Anti-holomorphic discrete series representations of $\Mp(W_1)_v$ do not occur in the quotient of the Weil representation $\ome^{\psi_v}_{V_{D_v}}$ for $v\in\frkS_\infty$. 
%by Proposition \ref{prop:82}(\ref{prop:824}). 
Consequently, $\tht^\psi_{V_D}\bigl(\frkC^{(1)}_{(2\kap+1)/2}\bigl)=\{0\}$, and so by the tower property, the space $\tht^\psi_V\bigl(\frkC^{(1)}_{(2\kap+1)/2}\bigl)$ consists of cuspidal automorphic forms on $\calg_1$. 
It follows that 
\[W_0(\tht^\psi_V(\calf_{(2\kap+1)/2}(h,\{c_t\}),\Phi_\Lam))=0. \]

Lemmas \ref{lem:94} and \ref{lem:95} show that 
\[W_\vXi(\tht^\psi_V(\calf_{(2\kap+1)/2}(h,\{c_t\}),\Phi_\Lam))=0 \]
unless $\vXi\in D_-^{\pi^{}_\bff\otimes\hat\chi^{-1}_\bff}$, in which case Lemma \ref{lem:93} gives a nonzero constant $C$ which is independent of $\vXi$ and such that 
\[W_\vXi(g_\infty,\tht^\psi_V(\calf_{(2\kap+1)/2}(h,\{c_t\}),\Phi_\Lam))=C c_{\nu(\vXi)}W^{(\kap+1)}_\vXi(g_\infty)w^{\mu^{}_\bff\hat\chi^{-1}_\bff}_\vXi(\Gam^{\psi_\bff}(h,\Phi)). \]
We conclude that 
\[\tht^\psi_V(g;\calf_{(2\kap+1)/2}(h,\{c_t\}),\Phi_\Lam)=C\calf_{\kap+1}(g_\infty;\Gam^{\psi_\bff}(h,\ome^{\psi_\bff}_V(g_\bff)\Phi),\{c_{\nu(\vXi)}\}) \]
for all $h\in A^{\psi_\bff}_1(\mu_\bff)$ and $\Phi\in\cals(V(\AAf))$. 
We therefore see by (\ref{tag:99}) that $\calf_{\kap+1}(f,\{c_{\nu(\vXi)}\})\in\frkG^1_{\kap+1}\subset\frkC^1_{\kap+1}$ for all $f\in A^{}_1(\mu^{}_\bff\hat\chi^{-1}_\bff)$. 
The map $f\mapsto\calf_{\kap+1}(f,\{c_{\nu(\vXi)}\})$ is nonzero by Proposition \ref{prop:81}(\ref{prop:812}) and Remark \ref{rem:92}.  
\end{proof}

%\begin{remark}\label{lem:92}The proof of Lemma \ref{lem:96} implies that $\tht^\psi_V(A^\psi_{\ell+1/2}(\mu))$ is nonzero (see the proof of Lemma \ref{lem:97} below for notation). This fact can be reproved by combining the Rallis inner product formula and Remark \ref{rem:92}. \end{remark}

%%%%%%%%%%%%%%%%%%%%%%%%%%%%%%%%%%%%%%%%%%%%%%%%%%%%%%%%%%%%%%%%%%%%%%%%%%%%%%%%
\subsection{End of the proof of Theorem \ref{thm:61}}\label{ssec:95}

Let $\{c_t\}_{t\in F^\times_{\pi_\bff}}\in C^{(1)}_{(2\kap+1)/2}(\mu_\bff)$. 
Fix $\eta\in F^\times_+$. 
Put $c_B=c_{\eta\nu(B)}$ for $B\in S^{\pi^{}_\bff\otimes\hat\chi^{(-1)^n\eta}_\bff}_n$. 
Thanks to the estimate of Fourier coefficients given in Proposition A.6.4 of \cite{Sh3}, we can invoke Lemma \ref{lem:42} to guarantee convergence of the series $\calf_{\kap+n}(f,\{c_B\})$ for $f\in A_n(\mu^{}_\bff\hat\chi^{(-1)^n\eta}_\bff)$. 
Moreover, $\{c_B\}\in T^n_{\kap+n}(\mu^{}_\bff\hat\chi^{(-1)^n\eta}_\bff)$ by Lemma \ref{lem:51}(\ref{lem:512}), (\ref{lem:515}). 
Note that $\{c_{\eta\nu(S)t}\}_t\in C^{(1)}_{(2\kap+1)/2}(\mu^{}_\bff\hat\chi^{\eta\nu(S)}_\bff)$ by Lemma \ref{lem:51}(\ref{lem:515}). 
Since 
\[\{c_{\eta\nu(S)\nu(\vXi)}\}_{\vXi\in D^{\pi^{}_\bff\otimes\hat\chi^{-\eta\nu(S)}_\bff}_-}\in C^1_{\kap+1}(\mu_\bff^{}\hat\chi^{-\eta\nu(S)}_\bff)=C^1_{\kap+1}(\mu_\bff^{}\hat\chi^{(-1)^n\eta}_\bff\hat\chi^S_\bff)\] 
for all $S\in S^+_{n-1}$ by Lemma \ref{lem:96}, our proof is complete by Lemma \ref{lem:76}. 

\begin{remark}\label{rem:93} 
When $D(F)\simeq\Mat_2(F)$, we have reproved Theorem \ref{thm:12}. 
\end{remark}

%%%%%%%%%%%%%%%%%%%%%%%%%%%%%%%%%%%%%%%%%%%%%%%%%%%%%%%%%%%%%%%%%%%%%%%%%%%%%%%%
\section{Translation to classical language}\label{sec:10}

The norm and the order of a fractional ideal of $\frko$ are defined by $\frkN(\frkp_v^k)=q_v^k$ and $\ord_v\frkp_v^k=k$. 
We denote the different of $F/\QQ$ by $\frkd$, the Dedekind zeta function of $F$ by $\zet_F(s)$ and the product of all the prime ideals of $\frko$ ramified in $D$ by $\frke^D$. 
Put 
\begin{align*}
\Gam_n[\frkd]&=\left\{\begin{pmatrix} a & b \\ c & d \end{pmatrix}\in G_n(F)\;\biggl|\; a,d\in\Mat_n(\calo),\;b\in\frkd^{-1}\Mat_n(\calo),\;c\in\frkd\Mat_n(\calo)\right\}, \\
\scrr^+_n&=\{B\in S_n^+\;|\;\tau(Bz)\in\frko\text{ for all }z\in S_n(F)\cap\Mat_n(\calo)\}. 
\end{align*}

For $t\in F^\times$ and $B\in S^+_n$ we denote the conductors of $\hat\chi^t$ and $\hat\chi^B$ by $\frkd^t$ and $\frkd^B$, respectively, and define rational numbers $\frkf_t$, $D_B$ and $\frkf_B$ by  
\begin{align*}
\frkf_t&=\sqrt{\frac{|N_{F/\QQ}(t)|}{N(\frkd^t)}}, & 
D_B&=\frkN(\frke^D)^{2[(n+1)/2]}N_{F/\QQ}(\nu(2B)), & 
\frkf_B&=\sqrt{\frac{D_B}{N(\frkd^B)}}.  
\end{align*}
We write $t\equiv\Box\pmod 4$ if there is $y\in\frko$ such that $t\equiv y^2\pmod 4$. 

For $t\in F^\times$ and a finite prime $v$ we define $f_v^t\in\ZZ$ by 
$f^t_v=\frac{1}{2}(\ord_vt-\ord_v\frkd^t)$. 
We define $\Psi(t,X)\in\CC[X+X^{-1}]$ by 
\[\Psi_v(t,X)=\begin{cases} 
\frac{X^{f_v^t+1}-X^{-f_v^t-1}}{X-X^{-1}}+\bigl(L\bigl(\frac{1}{2},\hat\chi^t_v\bigl)^{-1}-1\bigl)\frac{X^{f_v^t}-X^{-f_v^t}}{X-X^{-1}}&\text{if $f_v^t\geq 0$, }\\
0 &\text{if $f_v^t<0$. }
\end{cases}\]
For $B\in S^\nd_n$ we set $\wtl{F}_v(B,X)=X^{-f^B_v}F_v(B,q_v^{-(2n+1)/2}X)$, where 
\[f_v^B=f_v^{(-1)^n\nu(2B)}+
\begin{cases} 
0 &\text{if $v\nmid\frke^D$, }\\
\left[\frac{n+1}{2}\right] &\text{otherwise. }
\end{cases}\]  
%Then $\wtl{F}_v(B,X)=\wtl{F}_v(B,X^{-1})$ by \cite{I5,Y4}. 

For simplicity we consider the parallel weight case. 
Let $\kap$ be an integer greater than one such that $d(\kap+n)$ is even. 
Suppose that $\pi_\bff\simeq\otimes'_{v\notin\frkS_\infty}I(\alp^{s_v},\alp^{-s_v})$ appears in $\frkS_{2\kap}$. 
Then $\pi_\bff$ is isomorphic as a Hecke module to a certain subspace of the space $\frkC_{(2\kap+1)/2}^{(1)}$ by the works of Shimura and Waldspurger among others. 
Theorems 9.4, 10.1 and 13.5 of \cite{HI} give a Hecke eigenform $h\in\frkC_{(2\kap+1)/2}^{(1)}$ whose Fourier expansion is given by 
\[h(\calz)=\sum_{t\in\frko\cap F^\times_+,\;(-1)^nt\equiv\Box\pmod 4}c(t)\bfe_\infty(t\calz)\frkf_{(-1)^nt}^{\kap-1/2}\prod_{v\notin\frkS_\infty}\Psi_v((-1)^nt,q_v^{s_{v}}). \]
A holomorphic function $\calf$ on $\frkH^d_n$ is called a Hilbert-Siegel cusp form of weight $\ell$ with respect to $\Gam_n[\frkd]$ if $\calf|_\ell\gam=\calf$ for all $\gam\in\Gam_n[\frkd]$ and $\calf|_\ell\gam$ has a Fourier expansion of the form (\ref{tag:11}) for all $\gam\in G_n(F)$. 
We extend $c$ to a function on $F^\times_+/F_+^{\times 2}$ when $D\not\simeq\Mat_2(F)$. 
One can obtain the following explicit result from Theorem \ref{thm:61} and Lemma \ref{lem:34}, which is a strengthening of Theorems 3.2 and 3.3 of \cite{I2}.  

\begin{theorem}\label{thm:101}
Notations and assumptions being as above, we define a function $\calf:\frkH^d_n\to\CC$ by 
\[\calf(Z)=\sum_{B\in \scrr_n^+}c(\nu(2B))\bfe_\infty(\tau(BZ))\frkf_B^{\kap-1/2}\prod_{v\notin\frkS_\infty}\wtl{F}_v(B,q_v^{s_v}). \]
Then $\calf$ defines a cuspidal Hecke eigenform of weight $\kap+n$ with respect to $\Gam_n[\frkd]$ whose standard (partial) $L$-function is equal to   
\[\zet^{\frke^D}_F(s)\prod_{i=1}^{2n}L^{\frke^D}\left(s+n-i+\frac{1}{2},\pi\right), \]
where $\zet^{\frke^D}_F$ and $L^{\frke^D}$ are defined with Euler factors for primes in $\frke^D$ removed. 
\end{theorem} 

When $n=1$ and $F=\QQ$, this lifting is explicitly computed in \cite{O,S}. 

%%%%%%%%%%%%%%%%%%%%%%%%%%%%%%%%%%%%%%%%%%%%%%%%%%%%%%%%%%%%%%%%%%%%%%%%%%%%%%%%

\end{document}